\documentclass[leqno,11pt]{amsart}
\usepackage{amsmath,amssymb,amsfonts,amscd}
\usepackage{tikz}
\usepackage{verbatim}
\usepackage{rotating}
\usepackage[colorlinks]{hyperref}
\usepackage{color}

\theoremstyle{plain}
\newtheorem{thm}{Theorem}[section]

\newtheorem{prop}[thm]{Proposition}

\newtheorem{sublem}[equation]{Lemma}
\newtheorem{subcor}[equation]{Corollary}
\newtheorem{subprop}[equation]{Proposition}

\theoremstyle{definition}

\newtheorem{cosa}[thm]{}
\newtheorem{subcosa}[equation]{}

\newtheorem{exs}[thm]{Examples}

\theoremstyle{remark}
\newtheorem*{rem}{Remark}

\numberwithin{equation}{thm}

\newcommand{\D}{\boldsymbol{\mathsf{D}}}
\newcommand{\bE}{\boldsymbol{\mathsf{E}}}

\newcommand{\LL}{\mathsf L}
\newcommand{\R}{\mathsf R}

\renewcommand{\SS}{\mathsf{S}}

\newcommand{\Da}{\boldsymbol{\mathsf{a}}}
\newcommand{\Db}{\boldsymbol{\mathsf{b}}}
\newcommand{\Dd}{{\boldsymbol{\mathsf{d}}}}
\newcommand{\Dc}{\boldsymbol{\mathsf{c}}}

\newcommand{\ZZ}{\mathbb Z}

\newcommand{\Rf}{\R f^{}_{\<\<*}}
\newcommand{\Lf}{\LL f^*}

\newcommand{\cf}{\mathsf{cf}}
\newcommand{\co}{\mathsf{co}}
\newcommand{\fsh}{f^{}_{\mkern-1.5mu\sharp}}
\newcommand{\fst}{{f^{}_{\<\<*}}}

\newcommand{\qc}{\mathsf{qc}}

\newcommand{\Dqc}{\D_{\mathsf{qc}}}

\newcommand\Dqcpl{\D_\qc^{\lift.95,\text{\cmt\char'053},}}

\newcommand\upl{{\lift1,\text{\cmt\char'053},}}

\font\cmt=cmtex10
      
\newcommand{\CC}{\mathcal C}
\newcommand{\cd}[1]{\mathcal D_{\!#1}}

\newcommand{\CH}{\mathcal H}
\newcommand{\CO}{\mathcal O}

\newcommand{\bpic}{\begin{tikzpicture}}
\newcommand{\epic}{\end{tikzpicture}}

\newcommand{\Otimes}[1]{\otimes^\LL_{#1}}
\newcommand{\botimes}[1]{\boldsymbol\otimes_{#1}}
\newcommand{\totimes}{\,\bar\otimes\,}
\newcommand{\sHom}{\CH om}

\newcommand{\set}{\!:=}

\newcommand{\sX}{{\<\<X}}

\newcommand{\sst}{\scriptstyle}
\newcommand{\sss}{\scriptscriptstyle}

\newcommand{\smallcirc}{{\>\>\sst{\circ}\,}}
\newcommand{\liftcirc}{{\lift1,\,\sst{\circ}\,,}}
\newcommand{\ssscirc}{\lift.8,\,{\sss{\circ}\,},}

\newcommand{\<}{\mkern-1mu}
\renewcommand{\>}{\mkern1mu}
\newcommand{\va}[1]{\vspace{#1pt}}

\newcommand{\halfsize}[1]{{\scalebox{.6}{#1}}}

\newcommand{\kf}{\kern.5pt }

\def\lift#1,#2,{\vbox to 0pt{\vskip-#1 ex\hbox{$\scriptstyle #2$}\vss}}

\newcommand{\OX}{\mathcal O_{\<\<X}}
\newcommand{\OY}{\mathcal O_Y}
\newcommand{\OZ}{\mathcal O_{\<Z}}
\newcommand{\OS}{\mathcal O_{\mkern-1.5mu S}}
\newcommand{\OW}{\mathcal O_W}

\newcommand{\uniasterisco}[1]{\eta^{}_{#1}}
\newcommand{\uniasteriscode}[1]{\bar\eta^{}_{#1}\colon \id \to \R {#1}_* \LL {#1}^*}
\newcommand{\couniasterisco}[1]{\epsilon^{}_{#1}}
\newcommand{\couniasteriscode}[1]{\bar\epsilon^{}_{#1}\colon \LL {#1}^*\R {#1}_* \to \id}

\newcommand{\couni}[1]{{\smallint}_{#1}}

\newcommand{\counitimes}[1]{{\smallint}^{}_{\!#1}}

\newcommand\tst{\otimes_{\<*}}

\newcommand{\bchasterisco}[1]{\theta_{\<#1}}

\newcommand{\bchadmirado}[1]{\mathsf{B}_{\mkern-.5mu#1}}

\newcommand{\fundamentalclass}[1]{{\boldsymbol{\mathsf{c}}}_{#1}}
 
\newcommand{\fundamentalclassa}[1]{{\boldsymbol{\mathsf{a}}}_{#1}}

\newcommand{\fundamentalclassb}[1]{{\boldsymbol{\mathsf{b}}}_{#1}}

\newcommand{\biE}[4]{{\operatorname{HH}}^{#1}({#3} \xto{#2} {#4})}

\newcommand{\bsub}[1]{{{#1}}_{\sharp}}

\newcommand{\bup}[1]{#1{}^{\sharp}}

\newcommand{\pf}[1]{{#1}^{}_{\<\<{\star}}}
\newcommand{\pb}[1]{{#1}^{\star}}

\newcommand{\circled}[1]{\textcircled{\scriptsize{#1}}}

\newcommand{\Hsch}[1]{\CH_{\mathstrut#1}}

\newcommand{\lto}{\longrightarrow}
\newcommand{\xto}{\xrightarrow}

\newcommand\iso{{\mkern8mu\longrightarrow \mkern-25.5mu{}^\sim\mkern17mu}}

\DeclareMathOperator{\h}{H}

\DeclareMathOperator{\spec}{Spec}
\DeclareMathOperator{\Hom}{Hom}
\DeclareMathOperator{\tor}{Tor}
\DeclareMathOperator{\ext}{Ext}

\DeclareMathOperator{\id}{id}

\DeclareMathOperator{\HH}{HH}

\DeclareMathOperator{\via}{{\textup{via}}}

\DeclareMathOperator{\ps}{\mathsf{ps}}

\begin{document}
\date{}

%\frontmatter
\title[Bivariance, Grothendieck duality, Hochschild homology]{Bivariance, Grothendieck duality and  Hochschild homology}

\author[L. Alonso]{Leovigildo Alonso Tarr\'{\i}o}
\address{Departamento de \'Alxebra\\
Facultade de Matem\'a\-ticas\\
Universidade de Santiago de Compostela\\
E-15782  Santiago de Compostela, SPAIN}
\email{leo.alonso@usc.es}
\urladdr{http://webspersoais.usc.es/persoais/leo.alonso/index-en.html}

\author[A. Jerem\'{\i}as]{Ana Jerem\'{\i}as L\'opez}
\address{Departamento de \'Alxebra\\
Facultade de Matem\'a\-ticas\\
Universidade de Santiago de Compostela\\
E-15782  Santiago de Compostela, SPAIN}
\email{ana.jeremias@usc.es}

\author[J. Lipman]{Joseph Lipman}
\address{Department of Mathematics\\
Purdue University\\
West Lafayette IN 47907, USA}
\email{lipman@math.purdue.edu}
\urladdr{http://www.math.purdue.edu/\~{}lipman/}

\thanks{Authors partially supported by 
Spain's MICIIN and E.U.'s FEDER research project MTM2008-03465.
Third author also partially supported over time by NSF and NSA}

\dedicatory{To Heisuke Hironaka, on the occasion of his 80th birthday}

\date{\today}

\begin{abstract}
A procedure for  constructing bivariant theories by means of Grothendieck duality is developed. This produces, in particular, a bivariant theory of Hochschild (co)homology 
on the category of schemes that are flat, separated and essentially of finite type  
over a \mbox{noetherian} scheme $S$. The theory takes values in
the category of symmetric graded modules over the graded-commutative ring
$\oplus_i\> \textup{H}^i(S,\OS\<)$. In degree $i$, 
the cohomology and homology  $\textup{H}^0(S,\OS\<)$-modules thereby associated 
to such an $x\colon X\to S$,
with Hochschild complex\vspace{.6pt}  $\mathcal H_x$, are 
\smash{Ext$^i_{\OX}\<\<(\mathcal H_x, \mathcal H_x)$} and 
\smash{Ext$^{-i}_{\OX}\<\<(\mathcal H_x, x^!\mathcal O_{\<\<S})\ (i\in\mathbb Z)$.}\vspace{.6pt} 
This  lays the foundation for a sequel that will treat orientations in bivariant Hochschild theory
through canonical \emph{relative fundamental class maps}, unifying and generalizing previously
known manifestations, via differential forms, of such maps.
\end{abstract}

\maketitle

\tableofcontents

\section*{Introduction}

Grothendieck duality is a cornerstone of cohomology theory for quasi-coherent sheaves in Algebraic Geometry. It relates the classical theory of the canonical linear system of a variety to an analogue of Poincar\'e duality. Indeed, one of the outstanding features of Grothendieck duality  is the interplay between concrete and  abstract aspects of the theory, the former being expressed  in terms of differentials and residues, while the latter are conveyed in terms of a  formalism of certain functors between derived categories---the Grothendieck operations, and a web of relations among them
(see, e.g.,~\cite{li}). These two aspects are linked by the \emph{fundamental class}
of a scheme\kf-map.

In its usual incarnation the fundamental class is, for  a noetherian-scheme map 
\mbox{$x \colon X \to S$}
that is  separated, essentially  finite type, perfect (i.e., of  finite flat dimension
or finite tor-dimension), and  equidimensional of relative dimension~$n$, 
a canonical derived-category map from suitably shifted top-degree relative 
differentials to the relative dualizing complex:
\[
C_{X|S} \colon \Omega^n_{X|S}[n] \to x^!\CO_S,
\]
where $x^!$ is the  twisted inverse image functor which is the principal actor in Grothendieck duality
theory; or equivalently, a map of coherent sheaves
\begin{equation}\label{fund1}
c_{X|S} \colon \Omega^n_{X|S} \to \omega_{X|S}\!:=H^{-n}x^!\OS,
\end{equation}
where $\omega_{X|S}$ is the relative dualizing (or canonical) sheaf associated to $x$.

In case $x$ is a smooth map, $c_{X|k}$ is the isomorphism that is well-known from Serre duality.
More general situations have been studied in various contexts, local and global, algebraic and analytic, e.g., \cite{aez}, \cite{ang}, \cite{anl}, \cite{KW}, \cite{Kd}. In \cite{blue}, there is a concrete treatment of the case
when  $S = \spec(k)$ with $k$ a perfect field and $X$  an integral algebraic scheme over $k$. The map $c_{X|k}$ is  realized there as a globalization of the local
residue maps at the points of~$X\<$,  leading to explicit versions of local and global duality and the relation between them. These results are generalized to certain maps of noetherian schemes
in \cite{HS}. 
In all these approaches, an important role is played---via factorizations  of~$x$ as
smooth$\smallcirc$finite---by the case  $n=0$, where  the notion of  fundamental class is equivalent to   that of traces of differential forms. 

After \cite{white} it became clear  that Hochschild homology and cohomology play a role in this circle of ideas. The connection with differentials comes via canonical maps from differential forms to
sheafified  Hochschild homology. 

Over schemes, the theory of Hochschild homology and cohomology goes back to work of Gerstenhaber and Shack \cite{gs} on deformation problems, see \cite{BF}, \cite{gwet},  \cite{c2} and \cite{cw}. Recently, more refined versions of the theory have been  developed,  in \cite{BF} and \cite{LV}. \va1

Our first main task is to construct, over a fixed noetherian base scheme~$S$, a \emph{bivariant theory} \cite{fmc}, taking values in  derived categories of complexes with quasi-coherent homology,
those categories being enriched by graded modules over the graded-commutative ring $H\set\oplus_{i\in\ZZ}\,H^i(S,\OS)$.

The construction makes use of properties of
the Hochschild complex 
 $\Hsch{x}\>$ of a separated, essentially finite\kf-type, perfect map
 $x\colon X\to S$---that is, the derived-category object 
$\LL{\delta}^*\R{\delta}_*\OX$ where \mbox{$\delta \colon  X\to X \times_S X$} is the diag\-onal map---and on basic facts from Grothendieck duality theory. (Strictly speaking, this~$\Hsch{x}$ should be called the ``Hochschild complex"\va{-.7} only when $x$ is \emph{flat}.) The $H$-module thereby associated to a morphism $f\colon (X\xto{\lift.5,x,\>}S)\to (Y\xto{\lift.7,y,\>}S)$ of such $S$-schemes is\va{-2}
\[
\HH^*(f)\!:=\oplus_{i\in\ZZ}\> \ext^i_{\OX}(\Hsch{x},f^!\>\Hsch{y})
=\oplus_{i\in\ZZ}\> \Hom_{\D(X)}\!\big(\Hsch{x},f^!\>\Hsch{y}[i\>]\big),
\]
\vskip-2pt
\noindent 
so that the associated cohomology groups are\va{-1}
 \[
\HH^i(X|S) \!:= \HH^i(\id_{\<\<X}\<) = \ext^i_{\OX}(\Hsch{x},\Hsch{x})
\] 
\vskip-2pt
\noindent
and the associated homology groups are\va{-1}
\[
\HH_i(X|S) \!:= \HH^{-i}(x) = \ext^{-i}_{\OX}(\Hsch{x},\>x^!\CO_S).
\]
\vskip-2pt
\noindent Over smooth $\mathbb C$-schemes, these bivariant homology groups have been studied in \cite{c1}, and in more sophisticated terms, in \cite{cw}. The bivariant cohomology groups
 form a graded algebra, of which the cohomology algebra  in \cite{c1} is an algebra retract.
(These bivariant groups are not to be confused with the bivariant cohomology groups in \cite[\S5.5.1]{Lo}.)
\vspace{1pt}
 
 The data constituting the bivariant theory are specified in section~\ref{BivHoch},
and the verification of the validity of the bivariant axioms is carried out in section~\ref{axioms}.
The construction is organized around purely category-theoretic properties of the derived direct- and inverse\kf-image pseudofunctors, and of the twisted inverse image pseudofunctor 
(see section~\ref{setup}), and of 
$\Hsch{x}$ (see section~\ref{BivHoch}).  This makes it applicable in other contexts where Grothendieck duality exists, like nonnoetherian schemes and noetherian formal schemes. Moreover, the few simple properties of $\Hsch{x}$ that are needed  are shared, for example, by 
the cotangent complex $L_{X|S}\>$, or  by the 
``true" Hochschild complex in \cite{BF}.

Section~\ref{realization} is devoted to showing that the formal properties in section~\ref{setup}  do come out of Grothendieck duality for separated essentially-finite-type perfect maps of noetherian schemes. It is only recently that duality theory has been made available for essentially-finite-type, rather than just finite-type, maps (see \cite{Nk}), making possible a unified treatment of local and global situations. That theory requires the tedious verification of commutativity of a multitude of
diagrams, and more of the same is needed for our purposes. That accounts in part for the length of 
section~\ref{realization}; but there is more to be checked, for example because of the upgrading of results about derived categories to the $H$-graded context. Thus the bivariant Hochschild theory, though quickly describable, as above, encompasses many relations.

To put the present results in context, let us discuss very briefly
our second main task, to be carried out in the sequel to this paper---namely,  to develop the notion of
the \emph{fundamental class} of an $f$ as above. This is an element\va{-1}
 $$
\varrho(f)\!:=\fundamentalclass{f}(\OY)\in\HH^0(f).
$$
\vskip-2pt
\noindent
In particular, when $y=\id_S$, one gets a map\va{-1}
in $\HH^0(x)=\HH_0(X|S)$,
\begin{equation*}
\varrho(x)\colon \Hsch{x}\to x^!\OS\>,
\end{equation*}
\vskip-2pt
\noindent
which together with a natural map $\Omega^i_{X|S}\to H^{-i}\>\Hsch{x}$ gives a map\va{-1}
$$
\Omega^i_{X|S}\to H^{-i}\>x^!\mathcal O_S\qquad(i\ge 0),
$$
\vskip-2pt
\noindent
that generalizes \eqref{fund1} when $x$ is flat, separated, 
and essentially finite type.

\penalty-5000
Two basic properties of 
the  fundamental class are: 

1)   \emph{Transitivity} vis-\`a-vis a composite map of $S$-schemes \mbox{$X\xto{u\>}Y\xto{v\>} Z$,} i.e.,
\[
\smash{\fundamentalclass{vu} = u^!\fundamentalclass{v} \smallcirc \>\fundamentalclass{u}v^*.}
\]

2)  \emph{Compatibility with essentially \'etale base change}.\va2

Transitivity gives in particular that 
$\fundamentalclass{vu}(\OZ)=u^!\fundamentalclass{v}(\OZ)\>\smallcirc\>\>\fundamentalclass{u}(\OY)$.
\vspace{1pt}
In terms of the bivariant product $\HH^0(u)\times\HH^0(v)\to \HH^0(vu)$, this says:
$$
\varrho(vu)=\varrho(u)\cdot\varrho(v).
$$
Thus the family $\varrho(f)$ is a family of \emph{canonical orientations},  compatible with essentially \'etale base change, for 
the flat maps  in our bivariant theory \cite[p.\,28, 2.6.2]{fmc}.

With this in hand, one can apply the general considerations in \cite{fmc} to obtain, for example,  
\emph{Gysin morphisms,} that provide ``wrong-way" functorialities for homology and 
cohomology.

\section{Review of graded categories and functors}
\label{H-graded}
Let there be given a  \emph{graded-commutative} ring
$H=\oplus_{i\in\ZZ} \,H^i$,
$$
hh'=(-1)^{mn}h'h\in H^{m+n}\qquad(h\in H^n\<,\; h'\in H^m).
$$ 
We will use the language of $H\<$\emph{-graded categories}. So let us recall  some of the relevant basic notions.

\begin{cosa} 
A category $\bE$ is \emph{$H\<\<$-graded} if  

(i) for any  objects $A\<$, $B$ in $\bE$, the set $\bE(A,B)$ of arrows from $A$ to $B$ is equipped with a  \emph{symmetric graded $H\<$-\kern.5pt module} structure: $\bE(A,B)$ is a graded abelian group \looseness=-1
$$
\bE(A,B)=\oplus_{i\in\ZZ} \,\bE^i(A,B)
$$ 
with both left and right graded $H\<$-module structures such that 
$$
h\>\alpha=(-1)^{mn}\alpha \>h\qquad\big(h\in H^n\<,\; \alpha\in\bE^m(A,B)\big),
$$ 
(so  each of these structures determines the other); and further, 

(ii) for any $C\in\bE$, the~composition map $\bE(B,C)\times\bE(A,B)\xto{\lift-.4,\,\ssscirc\,,}\bE(A,C)$ is 
\emph{ graded $H\<\<$-bilinear}\kern.5pt: it is $\ZZ\>\>$-bilinear, and such that for  $\beta\in\bE^m(B,C)$,
$\alpha\in\bE^n(A,B)$, $h\in H$, it~holds that $\beta\<\ssscirc \<\alpha\in\bE^{m+n}(A,C)$ and
$$
(h\>\beta)\<\ssscirc\<\alpha=h(\beta\<\ssscirc\<\alpha),\qquad
\beta\<\ssscirc\<(\alpha h)=(\beta\<\ssscirc\<\alpha)h\>.
$$

It follows that $(\beta\>h)\<\ssscirc\alpha=\beta\<\ssscirc \<(h\>\alpha)$, and then that composition factors uniquely through a homomorphism of symmetric graded $H\<$-modules
$$
\bE(B,C)\otimes_{\<H}\>\bE(A,B)\<\to\bE(A,C).
$$

Any full subcategory of an $H\<$-graded category $\bE$ is naturally $H\<\<$-graded.\va1

\begin{subcosa}\label{H-algebra}
For any object $A$ in an $H\<\<$-graded category $\bE$, $\bE(A,A)$ has a natural graded $H\<$-algebra structure. Indeed, the identity $\id_A$, being idempotent, is in $\bE^0(A,A)$, and the map
 $\tau^{}_{\!A}\colon H\to\bE(A,A)$ such that for all $n$ and $h\in H^n$,
$$
\tau^{}_{\!A}(h)=h\<\id_A=\id_Ah\in\bE^n(A,A)
$$
is a graded-ring homomorphism---since
$$
(h\<\id_A)\ssscirc(h'\id_A)=h(\id_A\<\<\ssscirc(\id_Ah'))=h((\id_A\<\<\ssscirc\<\id_A)h')=hh'\id_A\textup{---}
$$
that takes $H$  into the graded center of $\bE(A,A)$---since for $\alpha\in\bE^m(A,A)$, 
$$
(h\<\id_A)\ssscirc\alpha=h(\id_A\!\ssscirc\>\alpha)=h\alpha
=(-1)^{mn}(\alpha h)\ssscirc\<\<\id_A=(-1)^{mn}\alpha\<\ssscirc(h\<\id_A).
$$
\end{subcosa}

\begin{subcosa}\label{preadditive}
A \emph{preadditive category} is an $H\<$-graded category with $H=\oplus_{i\in\ZZ} \,H^i$,  the graded ring such that $H^0=\ZZ$ and $H^i=(0)$ for all $i\ne 0$.\va1
\end{subcosa}
\end{cosa}

\begin{cosa}\label{graded functors}

Let $\bE_1$ and $\bE_2$ be $H\<$-graded categories. A  functor \mbox{$F\colon\bE_1\to\bE_2$} is said to be
\emph{$H\<\<$-graded} if the maps $\bE_1(A,B)\to \bE_2(FA, FB)\ (A,B\in\bE_1)$
associated to $F$ are graded $H\<$-linear. 

Another  $H\<$-graded functor $G$ being given, a \emph{functorial map\/ $\xi\colon F\to G$ of degree~$n$} is a 
family of arrows $\xi^{}_A\in\bE_2^n(FA,GA)\ (A\in\bE_1)$ such that for any $\alpha\in\bE_1^m(A,B)$,
it holds that $(G\alpha)\<\ssscirc \xi^{}_A = (-1)^{mn} \xi^{}_B\ssscirc (F\alpha)$; in other words, the following diagram commutes up to the sign $(-1)^{mn}\>$:
\begin{equation}\label{(-1)^mn}
\CD
F\<A @>\xi^{}_A>> GA\\
@V F\alpha VV @VV G\alpha V \\
F\<B @>>\lift1.2,\xi^{}_B,> GB
\endCD
\end{equation}

Composing a functorial map of degree $n_1$ with one of degree $n_2$ produces one of degree 
$n_1+n_2$.
\end{cosa}

\begin{cosa}\label{graded center}
The \emph{graded center}  $\CC=\CC_{\bE}$ of an $H\<$-graded category $\bE$ is, to begin with, the 
graded abelian group whose $n$-th degree homogeneous component
$\CC^n$ consists of the  degree\kern.5pt-$n$ self-maps of the identity functor $\id_{\bE}$ of $\bE$.

This $\CC_{\bE}$ does not change when $H$ is replaced by the trivially graded ring~$\ZZ$.\looseness=-1

Composition of functorial maps gives a product 
$$
\CC^m\times \CC^n\to\CC^{m+n}\qquad (m,n\in\ZZ),
$$
for which, evidently, if $\xi\in \CC^m$ and $\zeta\in \CC^n$ then $\xi\zeta=(-1)^{mn}\zeta\xi$. 
Hence $\CC$ can be viewed, via the
graded-ring homomorphism $\tau\colon H\to\CC$ that takes $h\in H^n$ to the family $\tau^{}_{\!A}(h)=h\<\id_A\in\bE^n(A,A)\ (A\in\bE)$, as
a \emph{graded-commutative graded $H\<$-algebra}. 

For $\xi\in\CC^n\<$, composition with $\xi^{}_A$ (resp.~$\xi^{}_B$) maps 
$\bE^m(A,B)$ to $\bE^{m+n}(A,B)$; this produces a symmetric graded $\CC$-module structure on  
$\bE(A,B)$. Hence the~category $\bE$ is
$\CC$-graded. The original $H\<$-grading is obtained from the $\CC$-grading by restricting scalars 
via $\tau$.

In the case $H=\CC$, the above map  $\tau^{}_{\!A}$ becomes the \emph{evaluation map}  
\begin{equation}\label{evaluation}
\postdisplaypenalty10000
\textup{ev}^{}_{\!A}\colon\CC\to \bE(A,A)
\end{equation}
taking $\xi\in\CC^n$ to the map 
$\xi^{}_A\>$.

\end{cosa}

\begin{cosa}\label{unital}
The tensor product $\bE_1\botimes H\bE_2$ of  $H\<$-graded categories is the \mbox{$H\<$-graded} category whose objects are pairs $(A_1,A_2)\ (A_1\in\bE_1,\;A_2\in\bE_2)$, and~such that\looseness=-1
$$
(\bE_1\botimes H\bE_2)\big((A_1,A_2),\:(B_1,B_2)\big)\set \bE_1(A_1,B_1)\otimes_H \>\bE_2(A_2,B_2)
$$
with the obvious symmetric graded $H\<$-module structure, composition 
\begin{multline*}
\big(\bE_1(B_1,C_1)\otimes_H \>\bE_2(B_2,C_2)\big)
\times
\big(\bE_1(A_1,B_1)\otimes_H \>\bE_2(A_2,B_2)\big)\\
\lto\,
\bE_1(A_1,C_1)\otimes_H \>\bE_2(A_2,C_2)
\end{multline*}
being derived from the graded 
$H\<$-quadrilinear map 
$$
\bE_1(B_1,\<C_1)\<\times \bE_2(B_2,\<C_2)\<\times\bE_1(A_1,B_1)\<\times\bE_2(A_2,B_2)
\<\to \bE_1(A_1,\<C_1)\otimes_{\<\<H}\> \bE_2(A_2,\<C_2)
$$
such that for all $A_1\xto{\alpha_1\,}B_1\xto{\beta_1\,}C_1$ in $\bE_1$ and $A_2\xto{\alpha_2\,}B_2\xto{\beta_2\,}C_2$ in $\bE_2$,\va1 with
$\alpha_1\in\bE_1^{m_1}(A_1,B_1)$ and $\beta_2\in\bE_2^{n_2}(B_2,C_2)$, it holds that 
$$
 (\beta_1, \beta_2,\alpha_1,\alpha_2)\mapsto (-1)^{n^{}_2m^{}_1}(\beta_1\ssscirc\alpha_1)\otimes(\beta_2\ssscirc\alpha_2).
 $$
In particular,
$$
(\beta_1\otimes \beta_2)\ssscirc(\alpha_1\otimes \alpha_2)= (-1)^{n^{}_2m^{}_1}(\beta_1\ssscirc\alpha_1)\otimes(\beta_2\ssscirc\alpha_2)\colon A_1\otimes A_2\to C_1\otimes C_2.
$$
\begin{subcosa}
\emph{Notation.} Given  $A_k, B_k\in\bE_k,
\alpha_k\in\bE_k(A_k,B_k)\ (k=1,2)$, and a functor $\totimes\colon\bE_1\botimes H\bE_2\to\bE$, 
set\va{-3}
\begin{gather*}
A_1\totimes A_2\set\totimes\<\<(A_1,A_2),\\[2pt]
\alpha_1\totimes \alpha_2\set\totimes\<\<(\alpha_1\otimes\alpha_2)
\colon A_1\totimes A_2\to B_1\totimes\<B_2\>.
\end{gather*}
\end{subcosa}

\begin{subcosa}
\!A \emph{unital product} on an $H\<$-graded category $\bE$ is a quadruple $(\totimes\!,\<\CO\<\<,\<\lambda,\<\rho)$ where:\va1

\noindent (i)  $\totimes\!\colon\bE\botimes H\bE\to\bE$ is an $H\<$-graded functor,

\noindent(ii) $\CO$ is an object in $\bE$ (whence, by (i), there are $H\<$-graded endofunctors of~
$\>\bE$ taking $A\in\bE$ to $\CO\totimes A$ and to $A\totimes\CO$, respectively), and\va1

\noindent(iii) $\lambda\colon(\CO\totimes -)\iso \id_{\bE}$ and $\rho\colon (-\totimes \CO)\iso \id_{\bE}$ are degree\kern.5pt-0 functorial isomorphisms such that $\lambda^{}_\CO=\rho^{}_{\<\CO}\colon \CO\totimes\CO\iso\CO$.\va1
\end{subcosa}
 
\begin{subcosa}\label{unital->gradedcomm}
Given such a unital product, one verifies that the map that takes $\eta\in\bE^n(\CO,\CO)$ to the family $(\eta^{}_A)^{}_{\<A\in\>\bE}$ in~ $\CC^n$ such that $\eta^{}_A$ is the composite map
$$
A\underset{\<\lambda_A^{\<-\<1}\,\>}\iso \CO\totimes A \xto[\eta\totimes\!\id_{\<A}\<\<]{} 
\CO\totimes A\underset{\lambda_A}\iso A
$$
is a homomorphism of graded $H\<$-algebras, 
right-inverse to $\textup{ev}_\CO\colon \CC\to \bE(\CO,\CO)$ (see~\eqref{evaluation}). 

Thus $\bE(\CO,\CO)$ is a graded-$H\<$-algebra retract of $\>\CC$, and so it is a graded-commutative $H\<\<$-algebra; and the $\CC$-grading on $\bE$ induces an $\bE(\CO,\CO)$-grading.

\end{subcosa}
\end{cosa}

\section{The underlying setup}\label{setup}
We now describe the formalism from which a bivariant theory will emerge in sections
~\ref{BivHoch} and~\ref{axioms}.  The formalism will be illustrated in section~\ref{realization} by several instances involving  Grothendieck duality.

\begin{cosa}\label{the category}
Fix a category $\SS$ and a graded-commutative ring $H$.

An \emph{orientation} of a relation $f{\smallcirc}\>v=u\>{\smallcirc}g$ among four $\SS$\kf-maps is  an ordered pair (right arrow, bottom arrow) whose members are $f$ and~$u$. This can be represented by one of two \emph{oriented commutative squares,} namely $\Dd$ with bottom arrow $u$, and its transpose $\Dd^\prime$ with bottom arrow $f$.
\[
 \begin{tikzpicture}[yscale=.9]
 %primer cuadrado E1 F1 G1 H1, segundo cuadrado E2 F2 G2 H2
       \draw[white] (0cm,0.5cm) -- +(0: \linewidth)
      node (E1) [black, pos = 0.21] {$\bullet$}
      node (F1) [black, pos = 0.39] {$\bullet$}
      node (E2) [black, pos = 0.61] {$\bullet$}
      node (F2) [black, pos = 0.79] {$\bullet$};
      \draw[white] (0cm,2.5cm) -- +(0: \linewidth)
      node (G1) [black, pos = 0.21] {$\bullet$}
      node (H1) [black, pos = 0.39] {$\bullet$}
      node (G2) [black, pos = 0.61] {$\bullet$}
      node (H2) [black, pos = 0.79] {$\bullet$};
      \node (C1) at (intersection of G1--F1 and E1--H1) [scale=0.75] {$\Dd$};
      \draw [->] (G1) -- (H1) node[above, midway, sloped, scale=0.75]{$v$};
      \draw [->] (E1) -- (F1) node[below, midway, sloped, scale=0.75]{$u$};
      \draw [->] (G1) -- (E1) node[left, midway, scale=0.75]{$g$};
      \draw [->] (H1) -- (F1) node[right, midway, scale=0.75]{$f$};
      \node (C2) at (intersection of G2--F2 and E2--H2) [scale=0.75] {$\Dd^\prime$};
      \draw [->] (G2) -- (H2) node[above, midway, sloped, scale=0.75]{$g$};
      \draw [->] (E2) -- (F2) node[below, midway, sloped, scale=0.75]{$f$};
      \draw [->] (G2) -- (E2) node[left, midway, scale=0.75]{$v$};
      \draw [->] (H2) -- (F2) node[right, midway, scale=0.75]{$u$};
 \end{tikzpicture}
\]

Assume that the category $\SS$ is  equipped with a class of maps, whose members are called \emph{confined maps,} and a class of oriented commutative squares, whose members are called \emph{independent squares}; and that these classes  satisfy (A1), (A2), (B1), (B2) and (C) in \cite[\S2.1]{fmc}---identity maps and composites of confined maps are confined, vertical and horizontal composites of independent squares are independent, any $\Dd$ in which $f=g$ and in which $u$ and~$v$ are identity maps  is independent, and if in the independent square~$\Dd$ the map
$f$ (resp.~$u$) is confined then so is $g$ (resp.~$v$).
\end{cosa}

\begin{cosa}\label{f^* and f^!}
With terminology as in \S\ref{H-graded}, assume given: \va1

(i) for each object $W\in\SS$  an \emph{$H\<\<$-graded  category} 
$\D_W$,\va{.5}  and

(ii) \emph{contravariant  $H\<\<$-graded pseudofunctors} $(-)^*$ and~$(-)^!$ over~$\SS$,\va1 with values in the categories $\D_W$---that is,\va{.8} to each  $f\colon X\to Y$ in~$\SS$ 
there are assigned\va{-2} $H\<$-graded functors $f^*$ 
and~$f^!$ from $\D_Y$ to $\D_\sX$;\va1 and to each  $X\xto{f}Y\xto{g}Z$ in~$\SS$ there are assigned functorial isomorphisms of degree~0\looseness=-1
$$
\ps^*\colon f^*\<\<g^*\iso(gf)^*,\qquad \ps^!\colon f^!\<g^!\iso(gf)^!
$$
such that for any $X\xto{f}Y\xto{g\>\>}Z\xto{h\>\>}W$ in $\SS$, the corresponding diagrams
\begin{equation}\label{assoc}
\CD\mkern-30mu
 \bpic[xscale=2.5, yscale=1.5]
   \node(11) at (1,-1){$f^*\<\<g^*\<h^*$};
   \node(12) at (2,-1){$f^*(hg)^*$};
   \node(13) at (3.05,-1){$f^!\<g^!h^!$};
   \node(14) at (4,-1){$f^!(hg)^!$}; 

   \node(21) at (1,-2){$(gf)^*h^*$};
   \node(22) at (2,-2){$(hgf)^*$};
   \node(23) at (3.05,-2){$(gf)^!h^!$};
   \node(24) at (4,-2){$(hgf)^!$}; 
 
   \draw[->](11)--(12);
    \draw[->](11)--(21);
     \draw[->](13)--(14);
      \draw[->](13)--(23);
       \draw[->](12)--(22);
        \draw[->](14)--(24);
         \draw[->](21)--(22);
          \draw[->](23)--(24);

 \epic
 \endCD
\end{equation}
commute.

Replacing $(-)^*$ and~$(-)^!$ by isomorphic pseudofunctors, we may
assume further that if $f$ is the identity map of $X\<$, then $f^*$ (resp.~$f^!$) is the identity functor of 
$\>\D_\sX$, and that $\ps^*$ (resp.~$\ps^!$) is the identity transformation of the functor $g^*$ (resp.~$g^!$); and likewise if $g$ is the identity map of $Y\<$.

Suggesting identification via $\ps^*$ or~$\ps^!$, the notations 
\begin{equation*}
\CD
 \begin{tikzpicture}
      \draw[white] (0cm,0.1cm) -- +(0: \linewidth)
        node (c1) [black, pos = 0.295] {$f^* \<\<g^*$}
        node (c2) [black, pos = 0.42] {$(gf)^*,$}
        node (c3) [black, pos = 0.585] {$f^! \<g^!$}
        node (c4) [black, pos = 0.7] {$(gf)^!,$};
       \draw [-, double distance=2pt] (c1) -- (c2) 
         node[auto,  midway, scale=0.75]{$\ps^*$};
       \draw [-, double distance=2pt] (c3) -- (c4) 
         node[auto,  midway, scale=0.75]{$\ps^!$};
 \end{tikzpicture}
\endCD
\end{equation*}
will be used to represent these  functorial isomorphisms 
or their inverses.\va2

\emph{Henceforth, any pseudofunctor under consideration  will be assumed to have been modified so as to exhibit the above-described simple behavior with respect to identity maps.}

\end{cosa}

\begin{cosa}\label{bchado}
Assume that there is assigned to each independent square
$$
\CD
\bullet @>v>> \bullet\\
@V g VV @VV f V \\
\bullet@>{\vbox to 0pt{\vss\hbox{$\sst\Dd$}\vskip.23in}} > \lift1.2,u, >\bullet
\endCD
$$
a \emph{degree\kern.5pt-$\>0$ isomorphism of $H\<\<$-graded functors}
$$
\bchadmirado\Dd\colon v^*\!f^!\iso g^!u^*.
$$

These $\bchadmirado{\sst\Dd}$  are to satisfy
\emph{horizontal and vertical transitivity}: if the  composite square
 $\Dd^{}_{0}=\Dd^{}_{2}\mkern-.5mu\smallcirc\Dd^{}_{1}$
 (with $g$ resp.~$\<\<v$ deleted)
\[
\CD
\bullet @> \lift.9,v_{\lift .6,\sss1,},>> \bullet@> v^{}_2 >> \bullet
\vadjust{\kern-2pt}\\
@V \lift.85,h, VV @V \lift.4,g, VV @VV f V  \\
\vspace{-22pt}\\
\bullet @>{\vbox to
0pt{\vss\hbox{$\lift1.6,\displaystyle\sst\Dd^{}_{\<1},$}\vskip.18in}}
        >\lift1.2,u_1, >
\bullet@>{\vbox to
0pt{\vss\hbox{$\lift1.6,\displaystyle\sst\Dd^{}_{\<2},$}\vskip.18in}}
        >\lift1.2,u_2, >  \bullet
\endCD
\qquad\quad \lift1.3,{\text{\normalsize\it resp.}}, \qquad\quad\
\CD
\bullet @>w>> \bullet \vadjust{\kern-2pt} \\
@V \lift.35,g^{}_{1},VV @VV f^{}_{\<1} V \\
\bullet @>v>{\vbox to
0pt{\vss\hbox{$\lift4.1,\displaystyle\sst\Dd^{}_{\<1},$}\vskip.205in}}
> \bullet \vadjust{\kern-2pt} \\
@V \lift.35,g^{}_{2}, VV @VV f^{}_{\<2} V \\
\bullet@>{\vbox to 0pt{\vss\hbox{$\sst\Dd^{}_{\<2}$}\vskip.185in}} >\lift 1.2,
u, > \bullet
\endCD
\]
has independent constituents $\Dd^{}_{\<2}$ and $\Dd^{}_{\<1}$ (so that
$\Dd_0$ itself is independent), then 
the corresponding natural  diagram of functorial maps 
commutes\/$:$\looseness=-1
\begin{equation}\label{basechange1}
\qquad
\CD
(v_{\lift .6,\sss2,}v_{\lift .6,\sss1,})^*\<\<f^! @.
\overset{\bchadmirado{\<\sst\Dd^{}_{0}}}{\hbox to 0pt{\hss\hbox to 100pt{\rightarrowfill}\hss}}
@.
  h^! (u_2u_1)^*\\
\hbox to 0pt{\hss$\sst\ps^*\>\>$}@| @. @|\mkern-28mu{\sst h^!\!\ps^*}  \\
v_{\lift .6,\sss1,}^*v_{\lift .6,\sss2,}^*\<\<f^!
@>>\lift1.2,v_1^*\bchadmirado{\<\sst\Dd^{}_2}, >
v_{\lift .6,\sss1,}^*g^! u_2^*
  @>>\lift1.2,\bchadmirado{\<\sst\Dd^{}_1}, >
h^! u_1^*u_2^*
\endCD
\end{equation}
{\it resp.}
\begin{equation}\label{basechange2}
\qquad
\CD
 (g_2g_1)^! u^*
@.\overset{\bchadmirado{\<\sst\Dd^{}_{0}}}{\hbox to 0pt{\hss\hbox to 100pt{\leftarrowfill}\hss}}
@.
 w^*(f_2f_1)^! \\
\hbox to 0pt{\hss$\sst\ps^!\>\>$}@| @. @|\mkern-28mu{\sst w^*\!\ps^!}  \\
g_1^! g_2^! u^* @<<\lift1.2,\,g_1^!\< \bchadmirado{\<\sst\Dd^{}_2}, <  g_1^! v^*\mkern-2.5muf_2^!
@<<\lift1.2,\,\,\bchadmirado{\<\sst\Dd^{}_1}\<\<, <
w^*\mkern-2.5mu f_1^! f_2^!\\[5pt]
\endCD
\end{equation}

\smallskip
Assume further that if $u$ and $v$ are identity maps, or if $f$ and $g$ are identity maps, then $\bchadmirado{\sst\Dd}$ is the identity transformation.
\smallskip
\end{cosa}

\pagebreak[3]
\begin{cosa}\label{f_*}
Assume  given a \emph{covariant $H\<$-graded pseudofunctor} $(-)_*$  
(that is, a contra\-variant $H\<$-graded pseudofunctor
over the opposite category $\SS^{\mathsf o\mathsf p}$), with values in the categories $\D_W$. 
Thus there are degree\kern.5pt-0 functorial isomorphisms $\ps_*\colon(gf)_*\iso g_*f^{}_{\<\<*}$ satisfying the appropriate analogs
of \eqref{assoc} and the remarks after it. This isomorphism or its inverse will be 
represented as\looseness=-1
$$
(gf)_*\overset{\ps_*}{=\!=} g_*\fst.
$$

Assume further that this pseudofunctor is \emph{pseudofunctorially right-adjoint to}~$(-)^*$: for any  $\SS$\kf-map  $f\colon X\to Y\<$, the functor $f^{}_{\<\<*}\colon\D_\sX\to\D_Y$ is graded
right-adjoint to  $f^*\colon\D_Y\to\D_\sX$, that is, there are degree\kern.5pt-0 functorial unit and  counit maps
\begin{equation}\label{etaeps}
\eta=\eta^{}_{\<f}\colon\id \to f^{}_{\<\<*}f^*\quad\text{ and }\quad\epsilon=\epsilon^{}_{\<\<f}\colon f^*\<\<f^{}_{\<\<*}\to\id 
\end{equation}
such that for $A\in\D_Y$ and $C\in\D_\sX$ the corresponding compositions 
$$
f^{}_{\<\<*}A \xto{\eta^{}_{\<f^{}_{\<\<*}{\<A}}\>\>}  f^{}_{\<\<*}f^*\<\<f^{}_{\<\<*}A \xto{f^{}_{\<\<*}\epsilon^{}_{\!A}\>\>} f^{}_{\<\<*}A,\qquad
f^*\<C \xto{f^*\<\eta^{}_{C}\>} f^*\<\<f^{}_{\<\<*}f^*\<C \xto{\epsilon^{}_{\<\<f^*\<C}} f^*\<C
$$
are identity maps---or equivalently, the induced composite maps of symmetric graded $H\<$-modules
\begin{gather*}
\D_Y(A, \fst C)\to \D_\sX(f^*\!A, f^*\<\<\fst C)\to \D_\sX(f^*\!A,  C),\\
\D_\sX(f^*\!A,  C)\to \D_Y(\fst f^*\!A,  \fst C)    \to\D_Y(A, \fst C)
\end{gather*}
are \emph{inverse isomorphisms;}
and  for any $X\xto{f} Y\xto{g\>\>} Z$ in~$\SS$,
 the following diagram commutes:
\begin{equation}\label{adjpseudo}
\CD
 \bpic[xscale=2.7,yscale=1.75]
   \node(11) at (1,-1){$\id$};  
   \node(12) at (2,-1){ $g_*g^*$};
   \node(13) at (3,-1){$g_*(f^{}_{\<*}f^*g^*)$};
  
   \node(21) at (1,-2){$(gf)_*(gf)^*$};
   \node(22) at (2,-2){$g_*f^{}_{\<*}(gf)^*$};
   \node(23) at (3,-2){$ g_*f^{}_{\<*}f^*g^*$};

   \draw[->] (11)--(12) node[midway, above, scale=0.75]{$\eta^{}_g$}; 
   \draw[->] (12)--(13) node[midway, above, scale=0.75]{$\text{via}\;\eta^{}_{\<f}$};
  
   \draw [-, double distance = 2pt](21)--(22) node[midway, below=2.5pt,scale=0.75]{$\ps_*$};
   \draw [-, double distance = 2pt](23)--(22) node[midway, below=.5pt,scale=0.75]{$\text{via}\;\ps^*$};
  
   \draw [->](11)--(21) node[midway, left=1pt,scale=0.75]{$\eta^{}_{g\<f}$};
   \draw [-, double distance = 2pt](13)--(23);
  \epic 
 \endCD
 \end{equation}

Assume also that to each \emph{confined} map $f\colon X\to Y$ in $\SS$\va{-1.5} there is assigned a degree\kern.5pt-0 functorial map 
\begin{equation}\label{trace}
\couni{\<\<\!f}\colon f^{}_{\<\<*}f^!\to \id
\end{equation}
 satisfying \emph{transitivity}:\va1  for any $X\xto{f} Y\xto{g\>\>} Z$ in $\SS$  with $f$ and $g$ confined, the following diagram commutes
\begin{equation}\label{transitivity}
\CD
   \begin{tikzpicture}[xscale=3, yscale=1]
         \node (41) at (1,-1) {\raisebox{5pt}{$ (gf)_*(gf)^!$}};
      \node (43) at (2,-1) {\raisebox{5pt}{$g^{}_* f^{}_{\<\<*}(gf)^!$}};
    
      \node (53) at (3,-1) {\raisebox{5pt}{$ g^{}_* f^{}_{\<\<*}f^!g^!$}};
   
            \node (61) at (1,-3){\raisebox{5pt}{$\id$}};
     \node (63) at (3,-3) {\raisebox{5pt}{$\ \,g^{}_*g^!\,;$}};
           \draw [-, double distance=2pt]
                 (41) -- (43) node[above=1pt, midway, scale=0.75]{$ \ps_*$};
      \draw [-, double distance=2pt]
                 (43) -- (53) node[above=1pt,  midway, scale=0.75]{$\text{via}\;\ps^!$};                        
      \draw [->] (41) -- (61) node[left,  midway, scale=0.75]{$ \couni{\!gf}$};
      \draw [->] (53) -- (63) node[right, midway, scale=0.75]{$ \couni{\<\<\!f}$};
         \draw [->] (63) -- (61) node[below=1pt, midway, scale=0.75]{$ \couni{\!g}$};
  \end{tikzpicture}
 \endCD
\end{equation}
and  if $f$ is the identity map of $X$ then $\couni{\<\<\!f}$ is the identity transformation.\va1

\end{cosa}

\begin{cosa}\label{theta}
Associated to any oriented commutative square in $\SS$ 
\begin{equation*}
    \begin{tikzpicture}[yscale=.85]
      \draw[white] (0cm,0.5cm) -- +(0: \linewidth)
      node (E) [black, pos = 0.41] {$\bullet$}
      node (F) [black, pos = 0.59] {$\bullet$};
      \draw[white] (0cm,2.65cm) -- +(0: \linewidth)
      node (G) [black, pos = 0.41] {$\bullet$}
      node (H) [black, pos = 0.59] {$\bullet$};
      \draw [->] (G) -- (H) node[above, midway, sloped, scale=0.75]{$v$};
      \draw [->] (E) -- (F) node[below, midway, sloped, scale=0.75]{$u$};
      \draw [->] (G) -- (E) node[left,  midway, scale=0.75]{$g$};
      \draw [->] (H) -- (F) node[right, midway, scale=0.75]{$f$};
      \node (C) at (intersection of G--F and E--H) [scale=0.75] {$\Dd$};
    \end{tikzpicture}
  \end{equation*}
is the degree\kern.5pt-0 functorial map 
\[
\bchasterisco{\Dd}:  u^*\!f^{}_{\<\<*}\to g_*v^*
\]
 adjoint to
\[
 f^{}_{\<\<*}\xto{\<f^{}_{\<\<*}\eta^{}_v\>} f^{}_{\<\<*}v_*v^*\overset{\lift1.1,\ps_*,}{=\!=\!=} u_*g_*v^*,
 \]
i.e., $\theta_{\Dd}$ is the composition of the following chain of  functorial maps:
         \begin{equation}\label{def-of-bch-asterisco}
 u^*\!f^{}_{\<\<*}\xto{\!\!\textup{via}\;\eta^{}_v\>} u^*\!f^{}_{\<\<*}v_*v^*
 \overset{\lift1.3,\!\!\textup{via}\;\ps_*,}{=\!=\!=\!=} u^*\<u_*g_*v^*\xto{\epsilon^{}_u\>\>} g_*v^*\<.
\end{equation}

It is postulated that \emph{if\/ $\Dd$ is independent then\/ $\theta_{\Dd}$ is an isomorphism.}

 \end{cosa}
 
 \begin{cosa}\label{comm diag}
Finally, it is postulated that if $\Dd$ in~\ref{theta}  is independent and  $f$ (hence $g$) is confined, then the following diagram commutes
\begin{equation}\label{first diagram}
\CD
u^*\!f^{}_{\<\<*}f^! @>\theta_{\Dd}> >g_*v^*\!f^! \\
@Vu^*\<\<\couni{\<\<\!f}VV @VVg_*\bchadmirado{\Dd}V \\
u^* @<<\couni{\!g}< g_*g^!u^*
\endCD
\end{equation}
that is, the following diagram commutes
\[
\CD
u^*\!f^{}_{\<\<*}f^! @>\eta^{}_g>>g_*g^*u^*f^{}_{\<\<*}f^! 
  \hbox to 0pt{$\mkern10.5mu\overset{\text{via}\;\ps^*}{=\!=\!=\!=}\hss$}@.g_*v^*\!f^*\<\<f^{}_{\<\<*}f^!\\
@Vu^*\<\<\couni{\!\<\<f}VV @. @VV\text{via}\;\epsilon^{}_{\<f}V \\
u^* @<<\couni{\!g}< \ g_*g^!u^*@<<g_*\bchadmirado{\Dd}\<<\ g_*v^*\!f^!\>\>;
\endCD
\]
and if, in addition, $u$ (hence $v$) is confined, then with $\phi_{\Dd}$ the degree\kern.5pt-0 functorial map adjoint to the composite map
\[
v^*\!f^!u_* \xto{\ \bchadmirado{\Dd}\ }g^!u^*u_{\<*} \xto{g^!\epsilon^{}_{u}\,}g^!,
\]
the following diagram commutes
\begin{equation}\label{second diagram}
\CD
f^!u_*u^! @>\phi_{\Dd}>> v_*g^!u^!\\
@Vf^!\<\<\couni{\!u}VV @|\mkern-13mu\sst v_*\!\ps^! \\
f^! @<<\couni{\!v}< v_*v^!\<\<f^!
\endCD
\end{equation}
that is, the following diagram commutes
\[
\CD
f^!u_*u^! @>\eta^{}_v>> v_*v^*\!f^!u_*u^! @>\!\textup{via}\;\bchadmirado{\Dd}\,>>v_*g^!u^*u_{\<*}u^! \\
@V f^!\<\<\couni{\!u} VV @. @VV\text{via}\;\epsilon^{}_{u}V \\
f^! @<<\couni{\!v}< v_*v^!\<\<f^!\hbox to 0pt{$\mkern26.5mu\underset{v_*\!\ps^!}{=\!=\!=\!=}\hss$}@.v_*g^!u^!
\endCD
\]
\end{cosa}
\vskip3pt
This completes the description of the underlying setup.

\begin{rem}
The order of composition of the functors in the domain and target of 
 $\>\bchasterisco{\Dd} \colon u^* f^{}_{\<\<*}\to  g_*v^*$  indicates that
we are considering that orientation of the relation $f{\smallcirc}\>v=u\>{\smallcirc}g$ for which $u$ is the bottom arrow.
So when such a relation is given, we usually simplify notation by  writing $\bchasterisco{} \colon   u^*\!f^{}_{\<\<*}\to  g_*v^*$ instead of $\>\bchasterisco{\Dd} \colon   u^*\<\< f^{}_{\<\<*}\to  g_*v^*$; and likewise for $\bchadmirado{\Dd}$ and $\phi_{\Dd}$. 
 \end{rem}

\section{Defining a bivariant  theory}\label{BivHoch}

\begin{cosa}
In this section,  we define data that will be shown in the next section to constitute a bivariant
theory \cite{fmc}. The approach will be purely formal, but justified by concrete examples
(see~\ref{ex-setup} and ~\S\ref{comparison}).

\begin{subcosa}\label{review setup}
Fix a setup, that is, a category $\SS$  with confined maps and independent squares, a graded-commutative ring $H$,\va{.6} a family $(\D_W)_{W\in\SS}$ of $H\<$-graded categories,
$H\<$-graded $\D_W$-valued pseudofunctors $(-)^*$, $(-)^!$ and $(-)_*$ over~$\SS$ (the first two contravariant and the last covariant),  for each independent square~$\Dd$, degree\kf-0 functorial isomorphisms $\bchadmirado{\Dd}$ and 
$\theta_\Dd\>$,  for each $\SS$\kf-map~$f$, degree\kf-0 functorial maps 
\begin{equation*}
\eta=\eta^{}_{\<f}\colon\id \to f^{}_{\<\<*}f^*\quad\text{ and }\quad\epsilon=\epsilon^{}_{\<\<f}\colon f^*\<\<f^{}_{\<\<*}\to\id ,
\end{equation*}
and for each confined map, a degree\kf-0 functorial map
\begin{equation*}
\couni{\<\<\!f}\colon f^{}_{\<\<*}f^!\to \id,
\end{equation*}
all subject to the conditions specified in \S\ref{setup}. Assume further that $\SS$ has a final object $S$.
\end{subcosa}

\begin{subcosa}\label{Hsch}
One associates to the  pseudofunctor $(-)^*$
the ``fibered category" \mbox{$\mathsf p\colon\mathsf F\to\SS$,} where the category $\mathsf F$ has as objects the pairs $(W, C)$ such that $W\in\SS$ and $C\in\D_W$, and as morphisms the pairs 
$(f,\psi)\colon (X,A)\to(Y,B)$ such that $f\colon X\to Y$ is an $\SS$\kf-map and $\psi\colon  f^*B\to A$
is a $\D_\sX$-map, the composition of such morphisms being defined in the obvious way, and where
the functor $\mathsf p$ is ``projection to the first coordinate."  The bivariant theory will be constructed
 from a section $\mathsf s$---a right inverse---of $\mathsf p$. Such an $\mathsf s$ can be specified  without reference to $\mathsf F$ or $\mathsf p$, see \S\ref{HH}.

For any $W\in\SS$,  set $(W, \Hsch{W})\set \mathsf s(W)$.  
(This notation reflects our original\- motivation, the case where $\Hsch{\<X}$ is a 
\emph{Hochschild complex}, see example~\ref{ex-setup}(b) below.)

\emph{Assume throughout that if\/ $f\colon X\to Y$ is the bottom or top arrow of an independent square, then
the\/ $\mathsf s$\kf-induced map\/ $ f^*\Hsch{Y}\to\Hsch{\<X}$ is an isomorphism.}\va1

We say that an $\SS$\kf-map is \emph{co\kf-confined} if it is represented by the bottom arrow
of some independent square.\va1

To each $\SS$\kf-map $f \colon X \to Y$ is attached the symmetric graded $H\<$-module
\[
\biE* {f}{X}{Y}\set \D_\sX(\Hsch{\<X},\>f^!\>\Hsch{Y})=\oplus_{i\in\ZZ}\, \D_\sX^i(\Hsch{\<X},\>f^!\>\Hsch{Y}).
\] 
We will define graded homomorphisms between such modules---products, pushforwards via confined maps, and pullbacks via independent squares---and then verify in the  next section
that for these operations in the given setup, the axioms of a bivariant theory hold.\looseness=-1
\end{subcosa}

\begin{subcosa}\label{co-hom}
There result  homology groups, 
covariant for confined $\SS$\kern.5pt-maps,
\[
\HH_i(X)  
\set\D^{-i}_{\sX}(\Hsch{\<X},x^!\Hsch{S}) \qquad(i\in\ZZ)
\]
where $x\colon X\to S$ is the unique $\SS$\kf-map; and   cohomology groups, 
contravariant for co\kf-confined $\SS$\kern.5pt-maps,
\[
\HH^i(X) 
\set \D^i_{\sX}(\Hsch{\<X},\Hsch{\<X}\<),
\] 
see \cite[\S2.3]{fmc}. As in \S\ref{H-algebra}, 
$$
\HH^*\<(X)\set\oplus_{i\in\ZZ}\>\HH^i(X)
=\D_\sX\<(\Hsch{\<X},\Hsch{\<X})
$$ 
is a graded $H\<$-algebra. (We will actually focus on the opposite $H\<$-algebra.) Composition of $\D_\sX$-maps makes the symmetric graded $H\<$-module  
$$
\HH_*(X)\set\oplus_{i\in\ZZ}\>\HH_{-i}(X)
=\D_\sX\<(\Hsch{\<X}, x^!\Hsch{S}\<)
$$
into a graded right $\HH^*\<(X)$-module (=\:graded left module over the opposite algebra).\vspace{1pt}

By way of illustration, we will indicate in \S\ref{comparison} the relation  to the present formalism of some previously defined Hochschild homology and cohomology functors on schemes.
\end{subcosa}
\end{cosa}

\begin{cosa}\label{HH} We now begin the detailed description of a bivariant theory.

Fix a setup $(\SS,\;H,\dots)$ as in \ref{review setup}.
Our construction  assumes given: \vspace{2pt}

(i) For each  $X\in\SS$  an object $\Hsch{\<X}\in\D_\sX$.

(ii) For each $\SS$\kern.5pt-map $f\colon X\to Y$  a $\D_\sX$-morphism 
$$
f^\sharp\colon f^*\Hsch{Y}\to\Hsch{\<X},
$$
such that 

(iii) if $f$ is an identity map then so is $f^\sharp$, and

(iv) (transitivity)
for $\SS$\kern.5pt-maps $X\xto{f}Y\xto{g}Z$ the next diagram commutes:
 \begin{equation}\label{trans^sharp}
 \CD
  \bpic[xscale=2.9, yscale=1.8]

   \draw (0,-1) node (11){$(g\<f)^*\>\Hsch{Z}$};
   \draw (1,-1) node (13){$\Hsch{\<X\<}$};

   \draw (0,-2) node (21){$f^*\!g^*\>\Hsch{Z}$};
   \draw (1,-2) node (23){$f^*\>\Hsch{\<Y}$};

   \draw[->] (11)--(13) node[above, midway, scale=0.75]{$(gf)^\sharp$};
   
   \draw[-, double distance=2pt] (11)--(21) node[left=1pt, midway, scale=0.75]{$\ps^*$};
   \draw[->] (23)--(13) node[right=1pt, midway, scale=0.75]{$f^\sharp$};
   
   \draw[->] (21)--(23) node[below, midway, scale=0.75]{$f^*\<\<g^\sharp$};
 \epic
 \endCD
\end{equation}  

It is further assumed that \emph{if\/ $f\colon X\to Y$ is the bottom or top arrow of an independent square, then\/
$f^\sharp$ is an isomorphism.} \va1

The adjoint of the map $f^\sharp$ will be denoted $\fsh\colon\Hsch{Y}\to f^{}_{\<\<*}\Hsch{\<X}\>$.

\begin{sublem}\label{corcomposedsquare}

Let\/  $X\xto{f\>}Y\xto{g\>}Z$ be\/ $\SS$\kern.5pt-maps.
The~next diagram commutes.
  \[
   \begin{tikzpicture}[xscale=3.3, yscale=2.5]
    
      \node (11) at (1,-1) {$\Hsch{Z}$};
      \node (13) at (1,-2) {$ g^{}_*\<\Hsch{Y}$};
    
      \node (21) at (2,-1) {$ (gf)_*\Hsch{\<X}$};
      \node (23) at (2,-2) {$ g^{}_* f^{}_{\<\<*}\<\Hsch{\<X}$};
      
      \draw [->] (11) -- (13) node[left=1pt,  midway, scale=0.75]{$g^{}_\sharp$};
      \draw [->] (13) -- (23) node[below=1pt, midway, scale=0.75]{$ g^{}_*\fsh$};
      \draw [->] (11) -- (21) node[above, midway, scale=0.75]{$(gf)^{}_\sharp$};
      \draw [-, double distance=2pt]
                 (21) -- (23) node[right=1pt, midway, scale=0.75]{$ \ps_*$};
       \end{tikzpicture}
\]
\end{sublem}

\begin{proof}
The diagram  expands as follows:\vspace{-2pt}
\[
  \bpic[xscale=3.5, yscale=2.8]

   \draw (0,-1) node (11)[scale=1]{$\Hsch{Z}$};
   \draw (1,-1) node (12)[scale=1]{$(gf)_*(gf)^*\>\Hsch{Z}$};
   \draw (2,-1) node (13)[scale=1]{$$};
   \draw (3,-1) node (14)[scale=1]{$(gf)_{\<*}\Hsch{\sX}$};

   \draw (0,-2) node (21)[scale=1]{$$};
   \draw (1,-2) node (22)[scale=1]{$(gf)_*f^*\<g^*\>\Hsch{Z}$};
   \draw (1.95,-2) node (23)[scale=1]{$(gf)_{\<*}f^*\>\Hsch{Y}$};
   \draw (3,-2) node (24)[scale=1]{$(gf)_{\<*}\Hsch{\sX}$};
   
   \draw (0,-3) node (31){$g_*g^*\>\Hsch{Z}$};
   \draw (1,-3) node (32)[scale=1]{$g_*f_*f^*\<g^*\>\Hsch{Z}$};
   \draw (2,-3) node (33){$$};
   \draw (3,-3) node (34){$$};

   \draw (0,-4) node (41)[scale=1]{$g_*\Hsch{Y}$};
   \draw (1,-4) node (42){$$};
   \draw (1.95,-4) node (43)[scale=1]{$g_*f_*f^*\>\Hsch{Y}$};
   \draw (3,-4) node (44)[scale=1]{$g_*f_*\Hsch{\sX}$};
   
   \draw[->] (11)--(12) node[above, midway, scale=0.75]{$\eta^{}_{gf}$};
   \draw[->] (12)--(14) node[above, midway, scale=0.75]{$\!\!(gf)_* (gf)^\sharp \quad$};
   
   \draw[->] (11)--(31) node[left, midway, scale=0.75]{$\eta^{}_{g}$};
   \draw[-, double distance=2pt] (12)--(22) node[left=1pt, midway, scale=0.75]{$\via \ps^*$};
   \draw[-, double distance=2pt] (14)--(24);
   
   \draw[->] (22)--(23) node[above, midway, scale=0.75]{$\!\via g^\sharp$};
   \draw[->] (23)--(24) node[above, midway, scale=0.75]{$\!(gf)_{\<*}f^\sharp$};
   
   \draw[-, double distance=2pt] (22)--(32) node[left=1pt, midway, scale=0.75]{$\ps_*$};
   \draw[-, double distance=2pt] (23)--(43) node[right=1pt, midway, scale=0.75]{$\ps_*$};
   \draw[-, double distance=2pt] (24)--(44) node[right=1pt, midway, scale=0.75]{$\ps_*$};
      
    \draw[->] (31)--(32) node[below=1pt, midway, scale=0.75]{$g_*\eta^{}_{\<f}$};
    \draw[->] (31)--(41) node[left=1pt, midway, scale=0.75]{$g_*g^\sharp$};
    \draw[->] (32)--(43) node[left=12pt, below=-2pt, midway, scale=0.75]{$\via g^\sharp$};
   
    \draw[->] (41)--(43) node[below=1pt, midway, scale=0.75]{$g_*\eta^{}_{\<f}$};
    \draw[->] (43)--(44) node[below=1pt, midway, scale=0.75]{$g_*f_*f^\sharp$};
    
    %labels
      \node (1) at (intersection of 11--32 and 31--12) {\kern-25pt\circled1};
      \node (2) at (1.95,-1.5){\circled2};  
  \epic
\]
\vskip-3pt
Commutativity of subdiagram \circled1 is shown in \cite[pp.\:118--119]{li}; of \circled2 is given by ~\eqref{trans^sharp}; and  of the remaining subdiagrams is obvious.
\end{proof}
\end{cosa}

\begin{cosa}\label{data}
Associate to any $\SS$\kern.5pt-map $f \colon X \to Y$  the symmetric graded $H\<$-module
\begin{equation}\label{HH*}
\biE* {f}{X}{Y}\set \D_\sX(\Hsch{\<X},\>f^!\>\Hsch{Y})=\oplus_{i\in\ZZ}\, \D_\sX^i(\Hsch{\<X},\>f^!\>\Hsch{Y}).
\end{equation}

There are three basic bivariant operations on these $H\<$-modules, as follows.

\pagebreak[3]
\begin{subcosa}
    {\bf Product.}
    \label{Product}
Let $f:X\to Y$ and $g:Y\to Z$ be maps in $\SS$.\va1 

For $i,j\in\ZZ$ and $\alpha \in \biE{\>i}{f}{X}{Y}$, $\beta \in \biE{\>j}{g\>}{Y}{Z},$ let the \emph{product}
 \[
    \alpha \<\cdot\<\< \beta \in \biE{i+j}{g f}{X}{Z}
  \]
be $(-1)^{ij}$ times the  composite map
\[
\Hsch{\<X}\xto{\,\,\alpha\,\,}f^!\>\Hsch{Y}\xto{\!f^!\beta\>}f^!g^!\>\Hsch{Z}
\overset{\ps^!}{=\!=}(g f)^!\>\Hsch{Z}.
\]
Since composition $\SS$ is $H\<$-bilinear, since $f^!$ is a graded functor and since 
$\ps^!(\Hsch{Z})$ has 
degree 0, therefore this product gives a graded $H\<$-bilinear map
\[
\HH^*(X\xto{\,f\,\>}Y)\times \HH^*(Y\xto{\,g\,\,}Z)\lto \HH^*(X\xto{gf}Z).
  \]
 
 For the case when $X=Y$ and $f={}$identity,  the identity map of $\Hsch{\<X}$ is a left unit for the product. Similarly when  $Y=Z$ and $g={}$identity, the identity map of $\>\Hsch{\<Z}$ is a
right unit.  

\end{subcosa}

\begin{subcosa}
   {\bf Pushforward.}
   \label{Pushforward}
Let $f \colon X \to Y$ and $g \colon Y \to Z$ be maps in $\SS$, with $f$~confined. 
The \emph{pushforward  by} {$f$}
\[
\pf{f}\colon \biE* {g f}{X}{Z} \to \biE* {\>\>g\,}{Y}{Z}
\]
is the graded $H\<$-linear map such that for $i\in\ZZ$ and $\alpha \in \biE{i}{g f}{X}{Z}$, the image
$\pf{f}\alpha \in \biE{i}{\>\>g\,}{Y}{Z}$
is the natural composition
\[
  \bpic[xscale=2.5, yscale=1]
   \draw (-.07,-1) node (11){$\Hsch{Y}$};
   \draw (0.55,-1) node (12){$f^{}_{\<\<*}\Hsch{\<X}$};
   \draw (1.4,-1) node (13){$ f^{}_{\<\<*}(gf)^!\>\Hsch{Z}$};
   \draw (2.38,-1) node (14){$f^{}_{\<\<*} f^!g^!\>\Hsch{Z}$};
   \draw (3.26,-1) node (15){$g^!\>\Hsch{Z}.$};
  
   \draw [->] (11) -- (12) node[above=0.5mm, midway, scale=0.75]{$\fsh$};
   \draw [->] (12) -- (13) node[above=0.5mm, midway, scale=0.75]{$ f^{}_{\<\<*}\alpha\ $};
   \draw [-, double distance=2pt] (13) -- (14) node[above=0.5mm,midway,scale=0.75]{$\fst\<\<\ps^!$};
   \draw [->] (14) -- (15) node[above=0.5mm, midway, scale=0.75]{$\couni{\!\<\<f}\;$};
   
  \epic
\]
In other words, $\pf{f}\alpha$ is the composition
\begin{equation*}
\CD
   \begin{tikzpicture}
      \draw[white] (0cm,0.5cm) -- +(0: .75\linewidth)
      node (11) [black, pos = 0.33] {$\Hsch{Y}$}
      node (12) [black, pos = 0.5] {$ f^{}_{\<\<*}\Hsch{\<X}$}
      node (15) [black, pos = 0.688] {$ g^!\>\Hsch{Z}$};
      \draw [->] (11) -- (12) node[auto, midway, scale=0.75]{$\fsh$};
      \draw [->] (12) -- (15) node[above=0.25mm, midway, scale=0.75]{$\widetilde{\alpha}$};
   \end{tikzpicture}
\endCD
\end{equation*}
where $\widetilde{\alpha} \colon  f^{}_{\<\<*}\Hsch{\<X} \to g^!\>\Hsch{Z}$ is the map obtained by adjunction from 
\[
   \begin{tikzpicture}
      \draw[white] (0cm,0.5cm) -- +(0: \linewidth)
      node (11) [black, pos = 0.28] {$\Hsch{\<X}$}
      node (12) [black, pos = 0.4285] {$(gf)^!\>\Hsch{Z}$}
      node (15) [black, pos = 0.605] {$f^!g^!\>\Hsch{Z}.$};
      \draw [->] (11) -- (12) node[above=0.25mm, midway, scale=0.75]{$\alpha$};
      \draw [-, double distance=2pt] (12) -- (15)
                              node[above=0.25mm, midway, scale=0.75]{$\ps^!$};
   \end{tikzpicture}
\]

\end{subcosa}

\begin{subcosa}
  {\bf Pullback.}
  \label{Pullback}
Let $\Dd$ be an independent square in $\SS$
  \[
   \begin{tikzpicture}[yscale=.95]
      \draw[white] (0cm,0.5cm) -- +(0: \linewidth)
      node (21) [black, pos = 0.41] {$Y'$}
      node (22) [black, pos = 0.59] {$Y$};
      \draw[white] (0cm,2.65cm) -- +(0: \linewidth)
      node (11) [black, pos = 0.41] {$X'$}
      node (12) [black, pos = 0.59] {$X$};
      \node (C) at (intersection of 11--22 and 12--21) [scale=0.9] {$\Dd$};
      \draw [->] (11) -- (12) node[above, midway, sloped, scale=0.75]{$g'$};
      \draw [->] (21) -- (22) node[below, midway, sloped, scale=0.75]{$g$};
      \draw [->] (11) -- (21) node[left, midway, scale=0.75]{$f'$};
      \draw [->] (12) -- (22) node[right, midway, scale=0.75]{$f$};
   \end{tikzpicture}
  \]
The maps $g^\sharp\colon  {g}^* \Hsch{Y}\to \Hsch{Y'}$ and $g'{}^\sharp\colon g'{}^* \Hsch{\<X}\to \Hsch{X'}$ are isomorphisms (\S\ref{HH}). 

\pagebreak[3]
The \emph{ pullback by $g,$ through $\Dd,$}
\[
\pb{g}\colon \biE* {f}{X}{Y}\lto \biE* {f'}{X'}{Y'}
\]
is the graded $H\<$-linear map such that for $i\in\ZZ$ and $\alpha \in \biE{i}{ f}{X}{Y}$, 
the image $\pb{g}\alpha \in \biE{i}{f'}{X'}{Y'}$
 is the natural composition
 \[
  \bpic[xscale=2.5, yscale=1]
   \draw (-.03,-1) node (11){$\Hsch{X'}$};
   \draw (0.75,-1) node (12){$g'{}^* \Hsch{\<X}$};
   \draw (1.66,-1) node (13){$g'{}^*\!f^! \Hsch{Y}$};
   \draw (2.64,-1) node (14){$f'{}^! {g}^*\Hsch{Y}\>$};
   \draw (3.61 ,-1) node (15){$f'{}^! \Hsch{Y'}\>.$};
  
   \draw [->] (11) -- (12) node[above=0.5mm, midway, scale=0.75]{$\!(g'{}^\sharp){}^{\<-1}$};
   \draw [->] (12) -- (13) node[above=0.5mm, midway, scale=0.75]{$g'^*\alpha\ $};
   \draw [->] (13) -- (14) node[above=0.5mm,midway,scale=0.75]{$\bchadmirado{\Dd}$};
   \draw [->] (14) -- (15) node[above=0.5mm, midway, scale=0.75]{$\<\<f'{}^!\< \big(g^\sharp\big)$};
   
  \epic
\]

For $X=Y\<$, $X'=Y'\<$,  $f$ and $f'$ identity maps,   pullback takes the identity map of $\Hsch{\<X}$ to
that of $\Hsch{\<X'}$. 

Thus identity maps are \emph{units} in the
sense of \cite[p.\,22]{fmc}. \vspace{2pt}

\end{subcosa}
\end{cosa}

\begin{thm}\label{3.4}
The data in sections~\textup{\ref{HH}--\ref{data}} constitute a bivariant theory,
with units, on\/ $\SS,$ taking values in symmetric graded $H\<\<$-modules.
\end{thm}

The proof of Theorem~\ref{3.4}---that is, the verification of the bivariance axioms---is given in \S\ref{axioms}.

In the rest of this section, we discuss some examples, and their associated bivariant
homology-cohomology pairs.

\begin{exs}\label{ex-setup}
In \S\ref{realization} we will show in detail  that there is a setup in which $\SS$ is a category of essentially-finite\kf-type perfect (i.e.,  finite tor-dimension) separated maps of noetherian schemes, closed under fiber product and having a final object $S$,  with proper maps as confined maps, and oriented fiber squares with flat bottom arrow as independent squares; and in which $H\set\oplus_{i\ge0}\>\>H^i(S,\OS)$ with its natural commutative\kf-graded ring structure.  Moreover, for each $X\in\SS$, $\D_\sX$ is  the full subcategory $\Dqc(X)$ of
the derived category $\D(X)$---enriched in the standard way with an $H\<$-graded structure---such that  an $\OX$-complex~$C$ is an object of $\Dqc(X)$ if and only
if all the homology sheaves of $C$ are quasi-coherent; and for any $S$-map $f\colon X\to Y\<$, $f^*$ is the graded enrichment of the derived
inverse\kf-image functor (usually denoted $\Lf$).  

The following examples refer to such a setup.\va2

\goodbreak
(a) Fix an object $\Hsch{S}\in\D_S$.\va{1.3}  For each $X\in\SS$, with its unique $\SS$-map $x\colon X\to S$,
set  $\Hsch{\<X}\set x^*\Hsch{S}\<$. For an $\SS$\kf-map $f$, let $f^\sharp$ be 
\smash{$f^*\<x^*\>\Hsch{S}\overset{\ps^*}{=\!\!\!=}(xf)^*\>\Hsch{S}\>$.}\va1

(b) For each $X\in\SS$
let  $\Hsch{\<X}$ be the Hochschild complex $\mathbb H_{X/S}$, and $f^\sharp$ as explained in the proof of \cite[Theorem 1.3]{BF}.\va{1.5}

(c) For each $X\in\SS$, let $\Hsch{\<X}$ be the cotangent complex $L_{X/S}\>$, and $f^\sharp$ the map given by \mbox{\cite[p.\,132, (1.2.7.2)]{Il}} (with $Y=Y'\set S$).\va{.7}

Examples (b) and (c) are not unrelated---see \cite[Theorem 3.1.3]{BF2}.\va{1.5} 

 (d) There are many  ways to get new
families satisfying \ref{HH}(i)--(iv) from old ones.  For example,\va1 to two such families 
$(\Hsch{\<X,1\>}, f_1^\sharp\>)$ and $(\Hsch{\<X,2\>}, f_2^\sharp\>)$, apply the 
derived tensor product functor,  or the direct sum functor, or \dots\va{1.5}

\end{exs}

\begin{subcosa}\label{indepndt}
In examples~\ref{ex-setup}(b) and (c), if an $\SS$-map $f\colon X\to Y$ is essentially \'etale (see \S\ref{shriek} below) then $f^*\Hsch{Y}\to\Hsch{\<X}$ is an isomorphism.
(The assertion for Example~ (b) will be treated in a sequel to this paper. Example (c) is covered by 
\cite[p.\,135, 2.1.2.1 and~p.\,203, 3.1.1]{Il}.)
So for these examples, one needs, according to \S\ref{HH},  to restrict the class of independent squares to those fiber squares whose bottom (hence top) arrow is essentially \'etale.

No such restriction is needed in Example~(a).
\end{subcosa}

\begin{cosa} \label{ho-co}

The bivariant  theory provides symmetric graded $H\<$-modules 
\[
\HH^*\<(X)\set \HH^*(X\xto{\id}X)=\D_\sX(\Hsch{\<X},\Hsch{\<X})=\oplus_{i\in\ZZ}\,\D_\sX^i(\Hsch{\<X},\Hsch{\<X})
\]
(\emph{bivariant cohomology}), and, with $x\colon X\to S$ the unique $\SS$-map,
\[
\HH_*(X)\set\HH^*(X\xto{x} S) = \D_\sX(\Hsch{\<X},x^!\Hsch{S})=\oplus_{i\in\ZZ}\,\D_\sX^{-i}(\Hsch{\<X},x^!\Hsch{S})
\]
(\emph{bivariant homology}).
 
 For instance, if, in~\ref{ex-setup}(a), $\Hsch{S}=\OS$, then bivariant cohomology is just
 $$
\HH^i(X) =\textup H^i(X\<, \OX);
$$
and homology is the (hyper)cohomology of the relative dualizing complex:
$$
\HH_i(X)=\textup H^{-i}(X, x^!\OX).
$$

For the bivariant Hochschild theory of example~\ref{ex-setup}(b), the corresponding (co)homology is discussed---at least for flat maps---in \S\ref{comparison}.

\smallskip
Functoriality, basic  properties of, and operations between, $\HH^*$ and $\HH_*$ result from the structure of a bivariant theory, and correspond to the usual structure of a theory of cohomology and homology, as follows.\vspace{1pt}

The cup product
\[
\smile \colon \HH^i(X) \otimes \HH^j(X) \lto \HH^{i+j}(X)
\]
is  the product \ref{Product} associated to the
composition $X \xto{\id} X \xto{\id} X$:
for $\alpha\in\D_\sX^i(\Hsch{\<X},\Hsch{\<X})$, 
$\beta\in\D_\sX^j(\Hsch{\<X},\Hsch{\<X})$,
$$\alpha\!\smile\!\beta\set
(-1)^{ij}\beta\ssscirc\alpha\in\D_\sX^{i+j}(\Hsch{\<X},\Hsch{\<X}).
$$
Cup product makes $\HH^*\<(X)$ into a graded ring---\emph{opposite} to $\D_\sX(\Hsch{\<X},\Hsch{\<X})$ with its composition product. Both rings have the same graded center, and so 
$\HH^*\<(X)$ is a graded $H$-algebra.

As in \S\ref{graded center}, both $\HH^*\<(X)$ and $\HH_*(X)$ are actually symmetric graded modules
over the graded center $\CC_\sX$ of $\D_\sX$. In fact, since $\CC_\sX$ is graded-commutative,
the evaluation map \eqref{evaluation} with $A=\Hsch\sX$ sends $\CC_\sX$ to the graded center
of $\HH^*\<(X)$, so that $\HH^*\<(X)$ \emph{is a graded\/ $\CC_{\<X}$-algebra.}\va1

Recall that an $\SS$\kf-map is \emph{co\kf-confined} if it is represented by the bottom arrow
of some independent square.\va1

It results from Proposition~\ref{A13} below (with $X=Y=Z$ and $f=g=1$) that for every co\kf-confined map $f \colon X' \to X$, the graded $H\<$-linear pullback
\[
\postdisplaypenalty 10000
\pb{f} \colon \HH^*\<(X) \lto \HH^*(X')
\]
of~\ref{Pullback} is a ring homomorphism.

Thus:
\begin{subprop}
With\/ $\SS_{\co}$ the subcategory of all co\kf-confined maps in~$\SS,$  bivariant
cohomology  with the cup product gives a \emph{contravariant} functor 
\[
\HH^* \colon \SS_{\co} \to \{\textup{$H\<$-algebras}\}.
\]
\end{subprop}

\pagebreak[3]
\smallskip
For $x \colon X \to S$ in  $\SS$, the cap product
\[\frown \colon \HH^i(X) \otimes \HH_j(X) \lto \HH_{j-i}(X)
\]
is defined to be the product  \ref{Product} associated to the composition $X \xto{\id\>\>} X \xto{x\>\>} S$: for 
$\alpha\in\D_\sX^i(\Hsch{\<X},\Hsch{\<X})$, 
$\beta\in\D_\sX^{-j}(\Hsch{\<X},x^!\OS)$,
$$
\alpha\!\frown\!\beta\set
(-1)^{ij}\beta\ssscirc\alpha\in\D_\sX^{i-j}(\Hsch{\<X},x^!\OS).
$$
With this product, $\HH_*(X)$ \emph{is a graded left\/ $\HH^*\<(X)$-module}.

\pagebreak[3]

Associated to a confined $\SS$-map  $f\colon X' \to X$ 
one has the $H\<$-linear pushforward of~\ref{Pushforward}\kern.5pt:
\[\pf{f} \colon \HH_*(X') \lto \HH_*(X).
\]
Thus:

\begin{subprop}
With $\SS_{\cf}$  the subcategory of all confined maps in~$\SS,$
bivariant\- homology together with the cap product, gives a \emph{covariant} functor 
\[\HH_* \colon \SS_{\cf} \lto \textup{\{symmetric graded $H\<$-modules\}}.
\]
Moreover, for every $X \in \SS$, $\HH_*(X)$ is a graded left\/ $\HH^*\<(X)$-module.
\end{subprop}

Proposition~\ref{A123} (with $Z=S$, $f=\id_X$, $f'=\id_{X'}$) yields:
\begin{subprop}
If\/ $g\colon X'\to X$ is both confined and co-confined, then for all $\alpha\in\HH^*(X)$ and $\beta\in\HH_*(X'),$
\[
g_\star(g^\star\alpha\<\frown\<\<\beta)=\alpha\<\frown\< g_\star\beta.
\]

\end{subprop}

\end{cosa}

\section{Checking the axioms}
\label{axioms}

In this section we prove Theorem~\ref{3.4} by verifying that the axioms for a bivariant theory do hold for the data referred to in that theorem. 

In the diagrams which appear, labels on the arrows are meant to
indicate where the represented maps come from---usually by application of some obvious functors. 

\penalty-2000
Recall from \eqref{HH*} that for an $\SS$-map $f\colon X\to Y$,
\[
\HH^i(X\xto{f\>}Y)\set  \D_\sX^i(\Hsch{\<X},\>f^!\>\Hsch{Y})\qquad (i\in\ZZ).
\]
Following \cite{fmc}, we indicate that $\alpha\in\biE* {f}{X}{Y}\set\oplus_{i\in\ZZ}\,\HH^i(X\xto{f\>}Y)$ by the notation  \vspace{-5pt}
\[
 \begin{tikzpicture}
      \draw[white] (0cm,1cm) -- +(0: \linewidth)
      node (22) [black, pos = 0.43] {$X$}
      node (23) [black, pos = 0.57] {$Y\>.$};
      \draw [->] (22) -- (23) node[auto, midway, scale=0.75]{$f$}
             node[below=1.5mm, midway, shape=circle, draw, scale=0.68]{$\!\alpha\!$};
 \end{tikzpicture}
\]

\begin{prop}
\label{A1}
\emph{$(A_1)$ Associativity of  product:}\va3

For any\/ $\SS$-diagram
  \[
   \begin{tikzpicture}
      \draw[white] (0cm,1cm) -- +(0: \linewidth)
      node (22) [black, pos = 0.29] {$X$}
      node (23) [black, pos = 0.43] {$Y$}
      node (24) [black, pos = 0.57] {$Z$}
      node (25) [black, pos = 0.71] {$W$};
      \draw [->] (22) -- (23) node[above=0.5mm, midway, scale=0.75]{$f$}
                              node[below=2mm, midway, shape=circle, draw,
                                   scale=0.68]{$\!\alpha\!$};
      \draw [->] (23) -- (24) node[above=0.5mm, midway, scale=0.75]{$g$}
                              node[below=2mm, midway, shape=circle, draw,
                                   scale=0.65]{$\!\!\!\!\beta\!\!\!\!$};
      \draw [->] (24) -- (25) node[above=0.5mm, midway, scale=0.75]{$h$}
                              node[below=2mm, midway, shape=circle, draw,
                                   scale=0.7]{$\!\!\gamma\!\!$};
   \end{tikzpicture}
  \]
one has, in\/  $\biE* {h  g  f}{X}{W},$
$$
   (\alpha\< \cdot\<\< \beta)\<\< \cdot\<\gamma=  \alpha\< \cdot\<\< (\beta\< \cdot\< \gamma).
$$

\end{prop}

\begin{proof}
This property results from the obvious commutativity of the following diagram,
with $\alpha \in \biE{i}{f}{X}{Y}$, $\beta \in \biE{j}{g}{Y}{Z}$  and $\gamma \in \biE{k}{h}{Z}{W}$: 
  \[\mkern-1mu
   \begin{tikzpicture}[xscale=.96,yscale=1.2]
      \draw[white] (0cm,2.5cm) -- +(0: .9\linewidth)
      node (11) [black, pos = 0] {$\Hsch{\<X}$}
      node (12) [black, pos = 0.21] {$(gf)^!\>\Hsch{Z}$}
      node (13) [black, pos = 0.615 ] {$(gf)^!h^!\>\Hsch{W}$}
      node (15) [black, pos = 0.92] {$(h gf)^!\>\Hsch{W}$};
      \draw[white] (0cm,0.5cm) -- +(0: .9\linewidth)
      node (31) [black, pos = 0] {$f^!\Hsch{Y}$}
      node (32) [black, pos = 0.21] {$f^!g^!\>\Hsch{Z}$}
      node (33) [black, pos = 0.615] {$f^!g^!h^!\>\Hsch{W}$}
      node (35) [black, pos = 0.92] {$f^!(h g)^!\>\Hsch{W}$};
      \draw [->] (31) -- (32) node[below=1pt, midway, scale=.75]{$f^!\beta$};
      \draw [->] (32) -- (33) node[below=1pt, midway, scale=.75]{$(\<\<-1\<)^{jk}f^!g^!\gamma$};
      \draw [-, double distance=2pt] (35) -- (33)
                              node[below=1pt, midway, scale=.75]{$\ps^!$};
      \draw [->] (11) -- (12) node[above, midway, scale=.75]{$\alpha\<\cdot\<\<\beta$};
      \draw [->] (12) -- (13) node[above=1pt, midway, scale=.75]{$\<\!\!(\<\<-1\<)^{\!(i+\<j)k}\<(gf)^!\gamma$};
      \draw [-, double distance=2pt] (15) -- (13)
                              node[above=1pt, midway, scale=.75]{$\ps^!$};
      \draw [->] (11) -- (31) node[left,  midway, scale=.75]{$(\<\<-1\<)^{i(j+k)}\alpha$};
      \draw [-, double distance=2pt] (15) -- (35)
                              node[right=1pt,  midway, scale=.75]{$\ps^!$};
      \draw [-, double distance=2pt] (12) -- (32)
                              node[left=1pt,  midway, scale=.75]{$(\<\<-1\<)^{ik}\!\ps^!$};
      \draw [-, double distance=2pt] (13) -- (33)
                              node[right=1pt,  midway, scale=.75]{$\ps^!$};
   \end{tikzpicture}
  \]
\vskip-5pt
\end{proof}

\begin{prop}
\label{A2}
\emph{($A_2$) Functoriality of pushforward:}\va2

For\/ $\SS$\kern.5pt-maps $f:X\to Y\<,$ $\>g:Y\to Z$ and\/ $h:Z\to W,$ with $f$ and $g$ confined, and 
$\alpha\in\biE* {h g f}{X}{W},$ 
one has, in  $\biE* {h}{Z}{W},$
$$
\pf{(gf)} (\alpha)=\pf{g\>}\pf{f} (\alpha).
$$
\end{prop}

\begin{proof}
We may assume, $\alpha \in \biE{i}{h g f}{X}{W}$. What is then
asserted is commutativity of the border of the following diagram:\va{-5}

  \[
   \begin{tikzpicture}[yscale=1.1]
      \draw[white] (0cm,9.5cm) -- +(0: \linewidth)
      node (11) [black, pos = 0.13][scale=0.9] {$\Hsch{Z}$}
      node (13) [black, pos = 0.48][scale=0.9] {$ g^{}_*\<\Hsch{Y}$};
      \draw[white] (0cm,7.7cm) -- +(0: \linewidth)
      node (21) [black, pos = 0.13][scale=0.9] {$ (gf)_*\Hsch{\<X}$}
      node (23) [black, pos = 0.48][scale=0.9] {$ g^{}_* f^{}_{\<\<*}\<\Hsch{\<X}$};
      \draw[white] (0cm,5.9cm) -- +(0: \linewidth)
      node (31) [black, pos = 0.13][scale=0.9] {\raisebox{5pt}{$ (gf)_*(hgf)^!\>\Hsch{W}$}}
      node (33) [black, pos = 0.48][scale=0.9] {\raisebox{5pt}{$ g^{}_* f^{}_{\<\<*}(hgf)^!\>\Hsch{W}$}}
      node (34) [black, pos = 0.88][scale=0.9] {\raisebox{5pt}{$ g^{}_* f^{}_{\<\<*}f^!(hg)^!\>\Hsch{W}$}};
      \draw[white] (0cm,4.1cm) -- +(0: \linewidth)
      node (41) [black, pos = 0.13][scale=0.9] {\raisebox{5pt}{$ (gf)_*(gf)^!h^!\>\Hsch{W}$}}
      node (43) [black, pos = 0.48][scale=0.9] {\raisebox{5pt}{$g^{}_* f^{}_{\<\<*}(gf)^!h^!\>\Hsch{W}$}};
      \draw[white] (0cm,2.3cm) -- +(0: \linewidth)
      node (53) [black, pos = 0.48][scale=0.9] {\raisebox{5pt}{$ g^{}_* f^{}_{\<\<*}f^!g^!h^!\>\Hsch{W}$}}
      node (54) [black, pos = 0.88][scale=0.9] 
          {\raisebox{5pt}{$ g^{}_* f^{}_{\<\<*}f^!(hg)^!\>\Hsch{W}$}};
      \draw[white] (0cm,0.5cm) -- +(0: \linewidth)
      node (61) [black, pos = 0.13][scale=0.9] {\raisebox{5pt}{$h^!\>\Hsch{W}$}}
      node (63) [black, pos = 0.48][scale=0.9] {\raisebox{5pt}{$ g^{}_*g^!h^!\>\Hsch{W}$}}
      node (65) [black, pos = 0.88][scale=0.9] {\raisebox{5pt}{$ g^{}_*(hg)^!\>\Hsch{W}$}};
      \node (etiqueta 1) at (intersection of 11--23 and 21--13) [scale=1] {\circled{1}};
      \node (etiqueta 2) at (intersection of 41--63 and 61--43) [scale=1] {\circled{2}};
      \draw [->] (11) -- (13) node[above,  midway, scale=0.75]{$g^{}_\sharp$};
      \draw [->] (13) -- (23) node[right, midway, scale=0.75]{$ g^{}_*\fsh$};
      \draw [->] (11) -- (21) node[left, midway, scale=0.75]{$(gf)^{}_\sharp$};
      \draw [-, double distance=2pt]
                 (21) -- (23) node[above=1pt, midway, scale=0.75]{$ \ps_*$};
      \draw [->] (21) -- (31) node[left,  midway, scale=0.75]{$ (gf)_* \alpha$};
      \draw [->] (23) -- (33) node[right, midway, scale=0.75]{$ g^{}_* f^{}_{\<\<*} \alpha$};
      \draw [-, double distance=2pt]
                 (31) -- (33) node[above=1pt, midway, scale=0.75]{$ \ps_*$};
      \draw [-, double distance=2pt]
                 (33) -- (34) node[above=1pt,  midway, scale=0.75]{$ \ps^!$};
      \draw [-, double distance=2pt]
                 (41) -- (43) node[below=1pt, midway, scale=0.75]{$ \ps_*$};
      \draw [-, double distance=2pt]
                 (43) -- (53) node[right=1pt,  midway, scale=0.75]{$ \ps^!$};                        
      \draw [-, double distance=2pt]
                 (31) -- (41) node[left=1pt,  midway, scale=0.75]{$ \ps^!$};
      \draw [->] (41) -- (61) node[left,  midway, scale=0.75]{$ \couni{\!gf}$};
      \draw [-, double distance=2pt]
                 (33) -- (43) node[right=1pt, midway, scale=0.75]{$ \ps^!$};
      \draw [-, double distance=2pt] (34) -- (54);
      \draw [->] (53) -- (63) node[right, midway, scale=0.75]{$ \couni{\<\<\!f}$};
      \draw [->] (54) -- (65) node[right, midway, scale=0.75]{$ \couni{\<\<\!f}$};
      \draw [-, double distance=2pt]
                 (54) -- (53) node[below=1pt, midway, scale=0.75]{$ \ps^!$};
      \draw [-, double distance=2pt]
                 (65) -- (63) node[below=1pt, midway, scale=0.75]{$ \ps^!$};
      \draw [->] (63) -- (61) node[below=1pt, midway, scale=0.75]{$ \couni{\!g}$};
  \end{tikzpicture}
\]
Commutativity of subdiagram  \circled{1} is given by Lemma~\ref{corcomposedsquare}.
Commutativity of  \circled{2} (without $h^!\>\Hsch{W}$) results from that of \eqref{transitivity}. Commutativity of the unlabeled subdiagrams is clear. 
The result follows.
\end{proof}

\pagebreak[3]

 \begin{prop}
 \label{A3}
 \emph{($A_3$) Functoriality of pullback:}\va2
 
For any\/ $\SS$\kern.5pt-diagram,  with independent squares, 
  \[
   \begin{tikzpicture}
      \draw[white] (0cm,2.5cm) -- +(0: \linewidth)
      node (12) [black, pos = 0.3] {$X''$}
      node (13) [black, pos = 0.5] {$X'$}
      node (14) [black, pos = 0.7] {$X$};
      \draw[white] (0cm,0.5cm) -- +(0: \linewidth)
      node (22) [black, pos = 0.3] {$Y''$}
      node (23) [black, pos = 0.5] {$Y'$}
      node (24) [black, pos = 0.7] {$Y$};
      \draw [->] (12) -- (13) node[above=1pt, midway, scale=0.75]{$h'$};
      \draw [->] (13) -- (14) node[above=1pt, midway, scale=0.75]{$g'$};
      \draw [->] (22) -- (23) node[below=1pt, midway, scale=0.75]{$h$};
      \draw [->] (23) -- (24) node[below=1pt, midway, scale=0.75]{$g$};
      \draw [->] (12) -- (22) node[left,  midway, scale=0.75]{$f''$};
      \draw [->] (13) -- (23) node[left,  midway, scale=0.75]{$f'$};
      \draw [->] (14) -- (24) node[right=.5pt,  midway, scale=0.75]{$f$}
                            node[left, midway]{\lift-.5,$\circled{\lift1.2,\displaystyle\alpha,}$,};
   \end{tikzpicture}
  \]
one has, in  $ \biE* {f''}{X''}{Y''},$
$$
   \pb{(gh)}(\alpha)=\pb{h}\pb{g}(\alpha).
$$
\end{prop}

\begin{proof}
The assertion amounts to
commutativity of the border of the next diagram:
 \[\mkern5mu
   \begin{tikzpicture}[xscale=1, yscale=1.5]
      \draw[white] (0cm,6.5cm) -- +(0: \linewidth)
      node (11) [black, pos = 0.05] {$h'{}^*\>\Hsch{\<X'}$}
      node (12) [black, pos = 0.26] {$h'{}^*\<g'{}^*\>\Hsch{\<X}$}
      node (13) [black, pos = 0.49 ] {$h'{}^*\<g'{}^*f{}^!\>\Hsch{Y}$}
      node (14) [black, pos = 0.76] {$h'{}^*\<\<{f'}{}{}^!\<g^*\>\Hsch{Y}$};
      \draw[white] (0cm,4.5cm) -- +(0: \linewidth)
      node (21) [black, pos = 0.05] {$\Hsch{\<X''}$}
      node (22) [black, pos = 0.26] {$\>(g' h')\<^*\>\Hsch{\<X}$}
      node (23) [black, pos = 0.49 ] {$(g' h')^*\<f{}^!\>\Hsch{Y}$}
      node (24) [black, pos = 0.76] {}
      node (25) [black, pos = 0.91 ] {$h'{}^*\<\<{f'}{}^!\>\Hsch{Y'}$};
      \draw[white] (0cm,2.5cm) -- +(0: \linewidth)
      node (31) [black, pos = 0.05] { }
      node (32) [black, pos = 0.26] { }
      node (33) [black, pos = 0.49 ] {${f''}{}^!(g h)^*\>\Hsch{Y}$}
      node (34) [black, pos = 0.76] {${f''}{}^!{h}^*\<g^*\>\Hsch{Y}$}
      node (35) [black, pos = 0.91 ] {};
      \draw[white] (0cm,0.49cm) -- +(0: \linewidth)
      node (41) [black, pos = 0.05] { }
      node (42) [black, pos = 0.25] { }
      node (43) [black, pos = 0.637] {${f''}{}^!\>\Hsch{Y''}$}
      node (44) [black, pos = 0.76] {}
      node (45) [black, pos = 0.91 ] {${f''}{}^!{h}^*\>\Hsch{Y'}$};
      %labels
      \node (1) at (intersection of 11--22 and 12--21)  {\ \circled{$1$}};
      \node (3) at (intersection of 13--34 and 14--33)  {\ \circled{$2$}};
      \node (5) at (intersection of 33--45 and 34--43)  {\circled{$3$}};
      %arrows
      \draw [->] (23) -- (33) node[left,  midway, scale=0.75]{$\bchadmirado{}$};
      \draw [->] (33) -- (43) node[left=2.5pt, midway, scale=1]{$\lift.5,(gh)^\sharp,$};
      \draw [-, double distance=2pt]
                 (34) -- (33) node[above=1pt, midway, scale=0.75]{$\ps^*$};
      \draw [->] (34) -- (45) node[left=2pt, midway, scale=1]{$\lift.5,g^\sharp,$};
      \draw [<-] (22) -- (21) node[below=1pt, midway, scale=0.75]{$((g' h'){}^\sharp){}^{-\<1}$};
      \draw [->] (12) -- (13) node[above=1pt, midway, scale=0.75]{$\alpha$};
      \draw [->] (13) -- (14) node[above, midway, scale=0.75]{$\bchadmirado{}$};
      \draw [<-] (11) -- (21) node[right, midway, scale=0.75]{$(h'{}^\sharp){}^{-\<1}$};
      \draw [-, double distance=2pt]
                 (12) -- (22) node[right=1pt, midway, scale=0.75]{$\ps^*$};
      \draw [-, double distance=2pt]
                 (13) -- (23) node[left, midway, scale=0.75]{$\ps^*$};
      \draw [<-] (12) -- (11) node[above=1pt, midway, scale=0.75]{$(g'{}^\sharp)^{-\<1}$};
      \draw [->] (22) -- (23) node[below=2pt, midway, scale=0.75]{$\alpha$};
      \draw [->] (14) -- (25) node[right=1pt, midway, scale=1]{$\lift1.7,g^\sharp,$};
      \draw [->] (14) -- (34) node[left, midway, scale=0.75]{$\bchadmirado{}$};
      \draw [->] (25) -- (45) node[right, midway, scale=0.75]{$\bchadmirado{}$};
      \draw [->] (45) -- (43) node[below, midway, scale=0.75]{$h^\sharp$};
  \end{tikzpicture}
  \mkern-10mu
  \]
Subdiagrams \circled{$1$} and \circled{$3$} commute by \ref{HH}(iv);  subdiagram \circled{$2$} commutes by~\eqref{basechange1}; and commutativity of the other two subdiagrams is clear. The desired conclusion results. 
\end{proof}

\bigskip
\begin{prop}
\label{A12}
\emph{($A_{12}$) Product and pushforward commute:}\va2

For any\/ $\SS$-diagram  
\[
 \begin{tikzpicture}
    \draw[white] (0cm,0.5cm) -- +(0: \linewidth)
    node (21) [black, pos = 0.3 ] {$X$}
    node (22) [black, pos = 0.43] {$Y$}
    node (23) [black, pos = 0.56] {$Z$}
    node (24) [black, pos = 0.69] {$W$};
    \draw [->] (22) -- (23) node[above, midway, scale=0.75]{$g$};
    \draw [->] (23) -- (24) node[above, midway, scale=0.75]{$h$}
                            node[below=1mm, midway, shape=circle, draw,
                                                scale=0.65]{$\!\!\!\!\beta\!\!\!\!$};
    \draw [->] (21) -- (22) node[above, midway, scale=0.75]{$f$};
    \draw [->] (21) .. controls +(1,-1) and +(-1,-1) .. (23)
                              node[below=1mm, midway, shape=circle, draw,
                                                     scale=0.68]{$\!\alpha\!$}
                              node[above, midway, scale=0.75]{$g f$};
 \end{tikzpicture}
\]
with $f:X\to Y$ confined, one has, in $\biE* {h g}{Y}{W},$
\[
\pf{f} (\alpha\<\cdot\<\< \beta )=\pf{f} (\alpha)\<\cdot\<\< \beta
\]
\end{prop}

\penalty-2000

\begin{proof} We may assume that $\alpha \in \biE{i}{g f}{X}{Z}$ and $\beta \in \biE{j}{h\>\>}{Z}{W}$. Then what is  asserted is commutativity of the border of the next diagram: 
\[ \mkern-3mu
 \bpic[xscale=2.6,yscale=2.6]
   \draw (.3,1) node (11){$\Hsch{Y}$};
   \draw (1,1) node (12){$f^{}_{\<\<*}\Hsch{\<X}$};
   \draw (1,0) node (13){$ f^{}_{\<\<*}(gf)^!\>\Hsch{Z}$};
   \draw (2.5,0) node (14){$f^{}_{\<\<*}(gf)^!h^!\>\Hsch{W}$};
   \draw (4.1,0) node (15){$f^{}_{\<\<*}(hgf)^!\>\Hsch{W}$};

   \draw (1,-1) node (23){$f^{}_{\<\<*}f^!g^!\>\Hsch{Z}$};
   \draw (2.5,-1) node (24){$f^{}_{\<\<*}f^!g^!h^!\>\Hsch{W}$};
   \draw (4.1,-1) node (25){$f^{}_{\<\<*}f^!(hg)^!\>\Hsch{W}$};
   
    \draw (1,-2) node (33){$g^!\>\Hsch{Z}$};
   \draw (2.5,-2) node (34){$g^!h^!\>\Hsch{W}$};
   \draw (4.1,-2) node (35){$(hg)^!\>\Hsch{W}$};

 %horizontal 
  \draw[->] (11)--(12) node[above=1pt, midway, scale=0.75]{$\fsh$};
  \draw[->] (12)--(13) node[left=1pt, midway, scale=0.75]{$\!f^{}_{\<\<*}\alpha$};
  \draw[->] (13)--(14) node[above=1pt, midway, scale=0.75]{$\beta\,$};
  \draw [-, double distance=2pt] (14)--(15) node[above=1pt, midway, scale=0.75]{$\ps^!$};
  
  \draw[->] (23)--(24) node[below=1pt, midway, scale=0.75]{$\beta\,$};
  \draw [-, double distance=2pt] (24)--(25) node[below=1pt, midway, scale=0.75]{$\ps^!$};

  \draw[->] (33)--(34) node[below=1pt, midway, scale=0.75]{$\beta\,$};
  \draw [-, double distance=2pt] (34)--(35) node[below=1pt, midway, scale=0.75]{$\ps^!$};
  
 %vertical    
   \draw [-, double distance=2pt] (13)--(23) node[left=1pt, midway, scale=0.75]{$\ps^!$};
   \draw [-, double distance=2pt] (14)--(24) node[right=1pt, midway, scale=0.75]{$\ps^!$};
   \draw [-, double distance=2pt](15)--(25) node[right=1pt, midway, scale=0.75]{$\ps^!$};
   
   \draw[->] (23)--(33) node[left=1pt, midway, scale=0.75]{$\couni{\<\<\!f}$};
   \draw[->] (24)--(34) node[right=1pt, midway, scale=0.75]{$\couni{\<\<\!f}$};
   \draw[->] (25)--(35) node[right=1pt, midway, scale=0.75]{$\couni{\<\<\!f}$};

 \epic
\]
The subdiagrams obviously commute, whence the assertion.
\end{proof}

\begin{prop}
\label{A13}
\emph{($A_{13}$) Product and pullback commute:}\va3

For any\/ $\SS$\kern.5pt-diagram  with independent squares, 
\begin{equation*} 
 \begin{tikzpicture}scale=1.3
      \draw[white] (0cm,0.5cm) -- +(0: \linewidth)
      node (E) [black, pos = 0.4] {$Z'$}
      node (F) [black, pos = 0.59] {$Z$};
      \draw[white] (0cm,2.65cm) -- +(0: \linewidth)
      node (G) [black, pos = 0.4] {$Y'$}
      node (H) [black, pos = 0.59] {$Y$};
      \draw[white] (0cm,4.8cm) -- +(0: \linewidth)
      node (I) [black, pos = 0.4] {$X'$}
      node (J) [black, pos = 0.59] {$X$};
      \draw [->] (G) -- (H) node[above, midway, sloped, scale=0.75]{$h'$};
      \draw [->] (E) -- (F) node[below=1pt, midway, sloped, scale=0.75]{$h$};
      \draw [->] (G) -- (E) node[left, midway, scale=0.75]{$g'$};
      \draw [->] (H) -- (F) node[right=1pt, midway, scale=0.75]{$g$}
                                     node[left, midway]{\lift-.3,{$\circled{\lift1.5,\displaystyle\beta,}$},};
      \draw [->] (I) -- (J) node[above, midway, sloped, scale=0.75]{$h''$};
      \draw [->] (I) -- (G) node[left, midway, scale=0.75]{$f'$};
      \draw [->] (J) -- (H) node[right=1pt, midway, scale=0.75]{$f$}
                                 node[left, midway]{\lift-.5,$\circled{\lift1.2,\displaystyle\alpha,}$,};
 \end{tikzpicture}
\end{equation*}
one has, in $ \biE* {g'\<\<f'}{X'}{Z'},$
$$
   \pb{h}(\alpha\<\cdot\<\< \beta)=\pb{h'}(\alpha)\<\cdot\<\pb{h}(\beta).
$$
\end{prop}

\begin{proof}
We may assume that $\alpha \in \biE{i}{f}{X}{Y}$ and $\beta \in \biE{j}{g}{Y}{Z}$. 
Then what is  asserted is commutativity of the border of the next diagram: 
\[\mkern5mu
 \begin{tikzpicture}
      \draw[white] (0cm,8cm) -- +(0: .9\linewidth)
      node (01) [black, pos = 0.22] 
         {$\mkern-90mu\Hsch{\<X'}\xto{\!\!(\bup{h''})^{-\<1}\!\!\!}{h''}^*\Hsch{\<X}$}
      node (13) [black, pos = 0.88] {$f'{}^!\>\Hsch{Y'}$};
     \draw[white] (0cm,6.5cm) -- +(0: .9\linewidth)
      node (11) [black, pos = 0.22] {${h''}^*f^!\>\Hsch{Y}$}
      node (12) [black, pos = 0.55] {$f'{}^!{h'}^*\Hsch{Y}$} 
      node (23) [black, pos = 0.88] {$f'{}^!{h'}^*\Hsch{Y}$};
     \draw[white] (0cm,5cm) -- +(0: .9\linewidth)
      node (21) [black, pos = 0.22] {${h''}^*f^!g^!\>\Hsch{Z}$}
      node (22) [black, pos = 0.55] {$f'{}^!{h'}^*g^!\>\Hsch{Z}$}
      node (33) [black, pos = 0.88] {$f'{}^!{h'}^*{g}^!\>\Hsch{Z}$};
      \draw[white] (0cm,3.5cm) -- +(0: .9\linewidth)
      node (31) [black, pos = 0.22] {${h''}^*(gf)^!\>\Hsch{Z}$};
     \draw[white] (0cm,2cm) -- +(0: .9\linewidth)
      node (41) [black, pos = 0.22] {$(g' f')^!{h}^*\Hsch{Z}$}
      node (43) [black, pos = 0.88] {$f'{}^!g'{}^!{h}^*\Hsch{Z}$};
     \draw[white] (0cm,0.5cm) -- +(0: .9\linewidth)
      node (51) [black, pos = 0.22] {$(g' f')^!\>\Hsch{Z'}$}
      node (53) [black, pos = 0.88] {$f'{}^!g'{}^!\>\Hsch{Z'}$};
  %labels
      \node (1) at (intersection of 21--43 and 41--33) [scale=1] {\circled{$1$}};
  %arrows
      \draw [->] (01) -- (11) node[left=1pt,  midway, scale=0.75]{$\>\alpha$};
      \draw [->] (11) -- (12) node[above=1pt, midway, scale=0.75]{$ \bchadmirado{}$};
      \draw [->] (12) -- (13) node[auto, midway, scale=0.75]{$ \bup{h'}$};
      \draw [-, double distance=2pt]
                 (12) -- (23) node[above=1pt, midway, scale=0.75]{};
      \draw [->] (21) -- (22) node[below=1pt, midway, scale=0.75]{$ \bchadmirado{}$};
      \draw [-, double distance=2pt] (22) -- (33) node[above=1pt, midway, scale=0.75]{};
      \draw [-, double distance=2pt]
                 (43) -- (41) node[above=1pt, midway, scale=0.75]{$ {\ps}^!$};
      \draw [-, double distance=2pt]
                 (53) -- (51) node[below=1pt, midway, scale=0.75]{$ {\ps}^!$};
      \draw [->] (11) -- (21) node[left=1pt,  midway, scale=0.75]{$ \beta$};
      \draw [->] (12) -- (22) node[left,  midway, scale=0.75]{$ \beta$};
      \draw [-, double distance=2pt]
                 (21) -- (31) node[left=1pt,  midway, scale=0.75]{$ {\ps}^!$};
      \draw [->] (31) -- (41) node[left,  midway, scale=0.75]{$ \bchadmirado{}$};
      \draw [->] (41) -- (51) node[left,  midway, scale=0.75]{$ \bup{h}$};
      \draw [->] (13) -- (23) node[right, midway, scale=0.75]{$\, ({\bup{h'}})^{-\<1}$};
      \draw [->] (23) -- (33) node[right, midway, scale=0.75]{$ \beta$};
      \draw [->] (33) -- (43) node[right, midway, scale=0.75]{$ \bchadmirado{}$};
      \draw [->] (43) -- (53) node[right, midway, scale=0.75]{$ \bup{h}$};
      
 \end{tikzpicture}
\]
Subdiagram \circled{$1$} commutes by \eqref{basechange2}; and commutativity of the other subdiagrams is clear. The desired result follows.
\end{proof}

\bigskip
\begin{prop} \emph{($A_{23}$) Pushforward and pullback commute:}\va2

For any\/ $\SS$\kern.5pt-diagram  with independent squares and with\/ $f$ confined, 
\begin{equation*} 
 \begin{tikzpicture}
      \draw[white] (0cm,0.5cm) -- +(0: \linewidth)
      node (E) [black, pos = 0.4 ] {$Z'$}
      node (F) [black, pos = 0.59] {$Z$};
      \draw[white] (0cm,2.65cm) -- +(0: \linewidth)
      node (G) [black, pos = 0.4 ] {$Y'$}
      node (H) [black, pos = 0.59] {$Y$};
      \draw[white] (0cm,4.8cm) -- +(0: \linewidth)
      node (I) [black, pos = 0.4 ] {$X'$}
      node (J) [black, pos = 0.59] {$X$};
      \node (label a) at (intersection of I--H and G--J) [scale=0.75] {$\Da$};
      \node (label b) at (intersection of G--F and H--E) [scale=0.75] {$\Db$};
      \draw [->] (G) -- (H) node[above, midway, scale=0.75]{$h'$};
      \draw [->] (E) -- (F) node[below, midway, scale=0.75]{$h$};
      \draw [->] (G) -- (E) node[left,  midway, scale=0.75]{$g'$};
      \draw [->] (H) -- (F) node[right,  midway, scale=0.75]{$g$};
      \draw [->] (I) -- (J) node[above, midway, scale=0.75]{$h''$};
      \draw [->] (I) -- (G) node[left,  midway, scale=0.75]{$f'$};
      \draw [->] (J) -- (H) node[right,  midway, scale=0.75]{$f$};
      \draw [->] (J) .. controls +(1.25,-1.25) and +(1.25,1.25) .. (F)
                       node[right=.5pt, midway]{\lift-.5,$\circled{\lift1.2,\displaystyle\alpha,}$,}
                       node[left=.5pt,  midway, scale=0.75]{$g  f$};                           
 \end{tikzpicture}
\end{equation*}
one has, in\/ $ \biE* {g'}{Y'}{Z'},$
$$
\pf{f'\!\!}(\pb{h}\<(\alpha))=\pb{h}(\pf{f}\<(\alpha)).
$$
\end{prop} 

\penalty-2000

\begin{proof}
What is  asserted is commutativity of the border of the following diagram, in which  
$\Dc$ denotes the square obtained by pasting $\Da$ and $\Db$:\va{-3}
\[ 
 \bpic [xscale=3.3,yscale=1.5]  
  \draw (0,0) node(11){$\Hsch{Y'}$};
  \draw (.95,0) node(12){$f'_{\<*}\Hsch{\<X'}$};

  \draw(0,-1) node(21){$h'{}^*\Hsch{Y}$};
  \draw(.95,-1) node(22){$f'_{\<*}h''{}^*\Hsch{\<X}$};

  \draw(0,-2) node(31){$h'{}^*\!f^{}_{\<*}\Hsch{\<X}$};
  \draw(.95,-2) node(32){$f'_{\<*}h''{}^*\<(g\<f)^!\Hsch{Z}$};
  \draw(2,-2) node(33){$f'_{\<*}(g'\<\<f')^!h^*\Hsch{Z}$};
  \draw(3,-2) node(34) {$f'_{\<*}(g'\<\<f')^!\Hsch{Z'}$};
  
  \draw(0,-3) node(41){$h'{}^*\!f^{}_{\<*}(g\<f)^!\Hsch{Z}$};
  \draw(.95,-3) node(42){$f'_{\<*}h''{}^*\!f^!g^!\Hsch{Z}$};

  \draw(0,-4) node(51){$h'{}^*\!f^{}_{\<*} f^!g^!\Hsch{Z}$};
  \draw(.95,-4) node(52){$f'_{\<*}f'{}^!h'{}^*\<\<g^!\Hsch{Z}$};
  \draw(2,-4) node(53){$f'_{\<*}f'{}^!g'{}^!h^*\Hsch{Z}$};
  \draw(3,-4) node(54){$f'_{\<*}f'{}^!g'{}^!\Hsch{Z'}$};

  \draw(0,-5) node(61){$h'{}^*\<g{}^!\Hsch{Z}[\>i\>]$};
  \draw(2,-5) node(63){$g'{}^!h^*\Hsch{Z}[\>i\>]$};
  \draw(3,-5) node(64){$g'{}^!\Hsch{Z'}[\>i\>]$};
  
 %horizontal
   \draw [->] (11) -- (12) node[above=1pt,  midway, scale=0.75]{$f'_\sharp$};
   \draw [->] (32) -- (33) node[above=1pt,  midway, scale=0.75]{$\bchadmirado{\Dc}\>\>$};  
   \draw [->] (33) -- (34) node[above=1pt,  midway, scale=0.75]{$h^\sharp$};

   \draw [->] (52) -- (53) node[below=1pt,  midway, scale=0.75]{$\bchadmirado{\Db}$};   
   \draw [->] (53) -- (54) node[below=1pt,  midway, scale=0.75]{$h^\sharp$}; 
     
   \draw [->] (61) -- (63) node[below=1pt,  midway, scale=0.75]{$\bchadmirado{\Db}$};
   \draw [->] (63) -- (64) node[below=1pt,  midway, scale=0.75]{$h^\sharp$};

 %vertical
   \draw [->] (11) -- (21) node[left=1pt,  midway, scale=0.75]{$(h'{}^\sharp)^{-\<1}$};
   \draw [->] (12) -- (22) node[right=1pt,  midway, scale=0.75]{$(h''{}^\sharp)^{-\<1}$};
   
   \draw [->] (21) -- (31) node[left=1pt,  midway, scale=0.75]{$f^{}_{\<\sharp}$};
   \draw [->] (22) -- (32) node[right=1pt,  midway, scale=0.75]{$\alpha$};
   
   \draw [->] (31) -- (41) node[left=1pt,  midway, scale=0.75]{$\alpha$};
   \draw [-, double distance=2pt] (32) -- (42) node[right=1pt,  midway, scale=0.75]{$\ps^!$};
   \draw [-, double distance=2pt] (33) -- (53) node[left=1pt,  midway, scale=0.75]{$\ps^!$};
   \draw [-, double distance=2pt] (34) -- (54) node[right=1pt,  midway, scale=0.75]{$\ps^!$};

   \draw [-, double distance=2pt] (41) -- (51) node[left=1pt,  midway, scale=0.75]{$\ps^!$};
   \draw [->] (42) -- (52) node[right,  midway, scale=0.75]{$\bchadmirado{\Da}$};

   \draw [->] (51) -- (61) node[left=1pt,  midway, scale=0.75]{$\couni{\!\!f}$};
   \draw [->] (53) -- (63) node[left=1pt,  midway, scale=0.75]{$\couni{\!\!f'}$};
   \draw [->] (54) -- (64) node[right=1pt,  midway, scale=0.75]{$\couni{\!\!f'}$};

 %slanted
  \draw [->] (31) -- (22) node[above=5pt, left=-.5pt, midway, scale=0.75]{$\bchasterisco{\Da}$};
  \draw [->] (41) -- (32) node[above=5pt, left=-.5pt, midway, scale=0.75]{$\bchasterisco{\Da}$};
  \draw [->] (51) -- (42) node[above=5pt, left=-.5pt, midway, scale=0.75]{$\bchasterisco{\Da}$};
  \draw [->] (52) -- (61) node[below=4.5pt, right=.5pt, midway, scale=0.75]{$\couni{\!\!f'}$};

 %labels
 \node(1) at (.46,-.75){\circled{1}};
 \node(2) at (1.49,-3.03){\circled{2}};
 \node(3) at (.475,-4){\circled{3}};

 \epic
\]
\vskip-2pt
Commutativity of subdiagram~\circled{2} is given by~\eqref{basechange2},
and of \circled{3}  by~\eqref{first diagram}. Commutativity of the unlabeled subdiagrams is clear. 

Commutativity of subdiagram \circled1 is equivalent to that of its adjoint, and so of
the border of the following diagram, where \mbox{$k\set h'f'=fh''\<$,} so that
commutativity of  \circled4 and \circled5 results from  \eqref{trans^sharp}, and where commutativity of the other subdiagrams results
directly from the definitions of the maps involved.\va{-5}  
\[
  \bpic[xscale=3.5, yscale=1.45]

   \draw (0,0) node (12){${\smash{f'}}^*\>\Hsch{Y'}$};
   \draw (2,0) node (14){$\Hsch{\<X'\<}$};
    
   \draw (1,-1) node (00){$k^*\>\Hsch{Y}$};
 
   \draw (0,-2) node (11){${\smash{f'}}^*\<  h'{}^*\>\Hsch{Y}$};
   \draw (1,-2) node (22){${\smash{h''}}^*\!  {f}^*\>\Hsch{Y}$};
   \draw (2,-2) node (23){${\smash{h''}}^*\>\Hsch{X}$};

   \draw (0,-3) node (31){$  {\smash{f'}}^* \< h'{}^*\!  f^{}_{\<\<*}\Hsch{\<X\<}$};
   \draw (1,-3) node (32){${\smash{h''}}^*\<\<  {f}^*\! f^{}_{\<\<*}\Hsch{\<X\<}$};
   \draw (2,-3) node (33){${\smash{h''}}^*\>\Hsch{X}$};

   \draw (0,-4) node (41){$$};
   \draw (1,-4) node (42){${\smash{f'}}^*\<\<  f'_{\<\<*}\>{\smash{h''}}^*\>\Hsch{X}$};
   \draw (2,-4) node (43){$$};
   
   \draw[->] (11)--(12) node[left, midway,, scale=0.75 ]{$f'{}^*\<h'{}^\sharp$};
   \draw[->] (00)--(14) node[midway, above=-.5pt, scale=0.75]{$k^\sharp\ \ $};
   \draw[->] (12)--(14) node[above=1pt, midway, scale=0.75]{$f'{}^\sharp$};
   
   \draw [->] (11) -- (31) node[left, midway, scale=0.75]{$\<f'{}^* \< h'{}^*\<\<f_\sharp$};
   \draw[-, double distance=2pt] (11)--(22) node[above=1pt, midway, scale=0.75]{$\ps^*$};
   
   \draw[->] (22)--(23) node[above=1pt, midway, scale=0.75]{$h''{}^*\<\<f^\sharp$};
   \draw[->] (22)--(32) node[right, midway, scale=0.75]{$h''{}^*\!f^*\! f_\sharp$};
   \draw[->] (23)--(14) node[right, midway, scale=0.75]{$h''{}^\sharp$};
    \draw[-, double distance=2pt] (23)--(33) ;

   \draw[-, double distance=2pt] (31)--(32) node[below, midway, scale=0.75]{$\ps^*$};    
   \draw[->] (32)--(33) node[below, midway, scale=0.75]{${\smash{h''}}^*\couniasterisco{\<\<f}$};    
   
   \draw[->] (31)--(42) node[below,  midway, scale=0.75]{${\smash{f'}}^*  \bchasterisco{\>\>\Da}\mkern25mu$};
   \draw[->] (42)--(33) node[below=2pt, midway, scale=0.75]{$\mkern15mu\couniasterisco{\<\<f'}$};
    
   \draw[-, double distance=2pt] (11)--(00) node[midway, auto, scale=.75]{$\ps^*$};
   \draw[-, double distance=2pt] (22)--(00) node[midway, auto, scale=.75]{$\ps^*$};

 %labels
   \node (2)  at (.5,-.55) [scale=1]{\circled4};
   \node (3)  at (1.5,-1.3) [scale=1]{\circled5};
   
     \epic
\]

The desired result follows.
\end{proof}

\begin{prop}\label{A123} 
\emph{($A_{123}$) Projection formula:}\va2

For any\/  $\SS$-diagram, with independent square and\/ $g$ confined,
  \[\mkern-75mu
   \begin{tikzpicture}[xscale=1.1, yscale=1.05]
      \draw[white] (0cm,0.5cm) -- +(0: .75\linewidth)
      node (22) [black, pos = 0.35] {$Y'$}
      node (23) [black, pos = 0.6] {$Y$}
      node (24) [black, pos = 0.85] {$Z$};
      \draw[white] (0cm,2.8cm) -- +(0: .75\linewidth)
      node (12) [black, pos = 0.35] {$X'$}
      node (13) [black, pos = 0.6] {$X$};
      \draw [->] (12) -- (13) node[above, midway, scale=0.75]{$g'$};
      \draw [->] (22) -- (23) node[below, midway, scale=0.75]{$g$};
      \draw [->] (23) -- (24) node[below, midway, scale=0.75]{$h$};
      \draw [->] (12) -- (22) node[left,  midway, scale=0.75]{$f'$};
      \draw [->] (13) -- (23) node[right=1pt,  midway, scale=0.75]{$f$}
                             node[left, midway]{\lift-.5,$\circled{\lift1.2,\displaystyle\alpha,}$,};
      \draw [->] (22) .. controls +(1,-1.1) and +(-1,-1.1) .. (24)
                               node[below=5pt, midway]{\lift-.5,$\circled{\lift1.6,\displaystyle\beta,}$,}
                              node[above,     midway, scale=0.75]{$h g$};
      \node (A) at (intersection of 22--13 and 23--12) [scale=0.75] {$\Dd$};                          
                        
   \end{tikzpicture}
  \]
one has, in\/ $\biE* {h f}{X}{Z}.$
 \[
  \pf{g'}(\pb{g}\alpha\<\cdot\<\< \beta)=\alpha\<\cdot\pf{g\>}(\beta).
 \]
\end{prop}

\begin{proof} We may assume that $\alpha \in \biE{i}{f}{X}{Y}$ and $\beta\in \biE{j}{h g}{Y'}{Z}$.
What is asserted is commutativity of the border of the next diagram~\eqref{forA123},
\begin{figure}[ht]
\begin{equation}\label{forA123}
 \mkern-5mu
   \begin{tikzpicture}[xscale=1.01, yscale=1.04]
      \draw[white] (0cm,16cm) -- +(0: \linewidth)
      node (1)  [black, pos = 0.1 ][scale=0.95] {$\Hsch{\<X}$}
      node (2)  [black, pos = 0.35][scale=0.95] {$g'_*\Hsch{X'}$};
      \draw[white] (0cm,13.5cm) -- +(0: \linewidth)
      node (01) [black, pos = 0.1 ][scale=0.95] {$\Hsch{\<X}$}
      node (02) [black, pos = 0.35][scale=0.95] {$g'_*g'{}^*\Hsch{\<X}$};
      \draw[white] (0cm,11cm) -- +(0: \linewidth)
      node (11) [black, pos = 0.1 ][scale=0.95] {$f^!\>\Hsch{Y}$}
      node (12) [black, pos = 0.35][scale=0.95] {$ g'_*g'{}^*f^!\>\Hsch{Y}$}
      node (13) [black, pos = 0.65][scale=0.95] {$ g'_*f'{}^!\<g^*\Hsch{Y}$}
      node (14) [black, pos = 0.9 ][scale=0.95] {};
      \draw[white] (0cm,8.5cm) -- +(0: \linewidth)
      node (21) [black, pos = 0.1 ][scale=0.95] {$f^! g_*\Hsch{Y'}$}
      node (22) [black, pos = 0.35][scale=0.95] {$ g'_*g'{}^*\<\<f^! g_*\Hsch{Y'}$}
      node (23) [black, pos = 0.65][scale=0.95] {$ g'_*f'{}^!\<g^*\<g_*\Hsch{Y'}$}
      node (24) [black, pos = 0.9 ][scale=0.95] {$ g'_*f'{}^!\>\Hsch{Y'}$};
      \draw[white] (0cm,6cm) -- +(0: \linewidth)
      node (31) [black, pos = 0.1 ][scale=0.95] {$f^! g_*(h g)^!\>\Hsch{Z}$}
      node (32) [black, pos = 0.35][scale=0.95] {$ g'_*g'{}^*\<\<f^! g_*(h g)^!\>\Hsch{Z}$}
      node (33) [black, pos = 0.65][scale=0.95] {$ g'_*f'{}^!\<g^*\<g_*(h g)^!\>\Hsch{Z}$}
      node (34) [black, pos = 0.9 ][scale=0.95] {$ g'_*f'{}^!(h g)^!\>\Hsch{Z}$};
      \draw[white] (0cm,3.5cm) -- +(0: \linewidth)
      node (41) [black, pos = 0.1 ][scale=0.95] {$f^! g_*g^!h^!\>\Hsch{Z}$}
      node (42) [black, pos = 0.35][scale=0.95] {$g'_*g'{}^*\<\<f^! g_*g^!h^!\>\Hsch{Z}$}
      node (43) [black, pos = 0.65][scale=0.95] {$g'_*f'{}^!\<g^*\<g_*g^!h^!\>\Hsch{Z}$}
      node (44) [black, pos = 0.9 ][scale=0.95] {$g'_*(hgf')^!\>\Hsch{Z}$};
      \draw[white] (0cm,1cm) -- +(0: \linewidth)
      node (51) [black, pos = 0.1 ][scale=0.95] {$f^! h^!\>\Hsch{Z}$}
      node (52) [black, pos = 0.35][scale=0.95] {$(hf)^!\>\Hsch{Z}$}
      node (53) [black, pos = 0.65][scale=0.95] {$g'_*g'{}^!(hf)^!\>\Hsch{Z}$}
      node (54) [black, pos = 0.9 ][scale=0.95] {$g'_*(hfg')^!\>\Hsch{Z}$};

      %Horizontales
      \draw [->] (1) -- (2)   node[above, midway, scale=0.75]{$\bsub{g'\<\<\!}$};
      \draw [->] (01) -- (02) node[above, midway, scale=0.75]{$\uniasterisco{g'}$};
      \draw [->] (11) -- (12) node[above, midway, scale=0.75]{$\uniasterisco{g'}$};
      \draw [->] (12) -- (13) node[above, midway, scale=0.75]{$\bchadmirado{\Dd}$};
      \draw [->] (21) -- (22) node[above, midway, scale=0.75]{$\uniasterisco{g'}$};
      \draw [->] (22) -- (23) node[above, midway, scale=0.75]{$\bchadmirado{\Dd}$};
      \draw [->] (23) -- (24) node[above, midway, scale=0.75]{$\couniasterisco{\<g}$};
      \draw [->] (31) -- (32) node[above, midway, scale=0.75]{$\uniasterisco{g'}$};
      \draw [->] (32) -- (33) node[above, midway, scale=0.75]{$\bchadmirado{\Dd}$};
      \draw [->] (33) -- (34) node[above, midway, scale=0.75]{$\couniasterisco{g}$};
      \draw [->] (41) -- (42) node[above, midway, scale=0.75]{$\uniasterisco{g'}$};
      \draw [->] (42) -- (43) node[above, midway, scale=0.75]{$\bchadmirado{\Dd}$};
      \draw [-, double distance=2pt] (51) -- (52)
                              node[below=1pt, midway, scale=0.75]{${\ps}^!$};
       \draw [-, double distance=2pt] (54) -- (53)
                              node[below=1pt, midway, scale=0.75]{${\ps}^!$};
            \draw [->] (53) -- (52) node[below=1pt, midway, scale=0.75]{$\couni{\<\!g'}$};
      %verticales
      %columna 1
      \draw [-, double distance=2pt] (1) -- (01) node[left=1pt, midway, scale=0.75]{};
      \draw [->] (01) -- (11) node[left=1pt, midway, scale=0.75]{$\alpha$};
      \draw [->] (11) -- (21) node[left=1pt, midway, scale=0.75]{$\bsub{g^{}\<\<}$};
      \draw [->] (21) -- (31) node[left=1pt, midway, scale=0.75]{$\beta$};
      \draw [-, double distance=2pt] (31) -- (41)
                              node[left=1pt, midway, scale=0.75]{${\ps}^!$};
      \draw [->] (41) -- (51) node[left=1pt, midway, scale=0.75]{$\couni{\<\!g}$};
     
      %columna 2
      \draw [->] (2) -- (02)  node[right=1pt, midway, scale=0.75]{$({\bup{g'}})^{-\<1}$};
      \draw [->] (02) -- (12) node[right=1pt, midway, scale=0.75]{$\alpha$};
      \draw [->] (12) -- (22) node[left=1pt, midway, scale=0.75]{$\bsub{g^{}\<\<}$};
      \draw [->] (22) -- (32) node[left=1pt, midway, scale=0.75]{$\beta$};
      \draw [-, double distance=2pt] (32) -- (42)
                              node[left=1pt, midway, scale=0.75]{${\ps}^!$};
      %columna 3
      \draw [->] (13) -- (23) node[left=1pt, midway, scale=0.75]{$\bsub{g^{}\<\<}$};
      \draw [->] (23) -- (33) node[left=1pt, midway, scale=0.75]{$\beta$};
      \draw [-, double distance=2pt] (33) -- (43)
                              node[left=1pt, midway, scale=0.75]{${\ps}^!$};
      %oblicua entre la columna 3 y 4
      \draw [->] (13) -- (24) node[right=10pt, above=-4.5pt, midway, scale=0.75]{$\bup{g}$};
      %columna 4
      \draw [->] (24) -- (34) node[right=1pt, midway, scale=0.75]{$\beta$};
      \draw [-, double distance=2pt] (34) -- (44)
                              node[right=1pt, midway, scale=0.75]{${\ps}^!$};
      \draw [-, double distance=2pt] (44) -- (54);
      %etiquetas
      \node (A) at (intersection of 1--02  and 01--2)  {\!\!\!\!\circled{1}};
      \node (B) at (intersection of 13--34 and 14--23){\circled{2}};
      \node (C) at (intersection of 41--54 and 44--51) {\circled{3}};
  \end{tikzpicture}
  \end{equation}
 \end{figure}
 in which commutativity of the unlabeled subdiagrams is obvious, and commutativity of subdiagrams \circled1 and \circled2 holds by adjointness of $g'{}^\sharp$ and $g'_\sharp$ (resp.~$g^\sharp$ and $g_\sharp$).
 It  remains then to show that \circled{3} commutes. 
 
 Via the next, obviously commutative, diagram (in which $\Hsch{Z}$ is omitted),
 \[
  \bpic[xscale=1.1, yscale =1.2]
   \node(11) at (1,-1)[scale=0.85]{$ g'_*f'{}^!\<g^*\<g_*(h g)^!$};
   \node(12) at (3.3,-1)[scale=0.85]{$ g'_*f'{}^!(h g)^!$};
   \node(13) at (5.3,-1)[scale=0.85]{$g'_*(hgf')^!$};
   \node(14) at (7.2,-1)[scale=0.85]{$g'_*(hfg')^!$};
   \node(15) at (9.2,-1)[scale=0.85]{$g'_*g'{}^!(hf)^!$};
   \node(16) at (10.95,-1)[scale=0.85]{$(hf)^!$};

   \node(21) at (1, -2)[scale=0.85]{$g'_*f'{}^!\<g^*\<g_*g^!h^!$};
   \node(22) at (3.3, -2)[scale=0.85]{$g'_*f'{}^!\<g^!h^!$};
   \node(23) at (5.3,-2)[scale=0.85]{$g'_*(gf')h^!$};
   \node(24) at (7.2,-2)[scale=0.85]{$g'_*(fg')h^!$};
   \node(25) at (9.2,-2)[scale=0.85]{$g'_*g'{}^!\<f^!h^!$};
   \node(26) at (10.95,-2)[scale=0.85]{$f^!h^!$};

%horizontal
  \draw[->] (11)--(12) node[above=1pt, midway, scale=0.65]{$\couniasterisco{g}$};
  \draw[-, double distance=2pt](12)--(13) node[above=1pt, midway, scale=0.65]{${\ps}^!$};
  \draw[-, double distance=2pt] (13)--(14);
  \draw[-, double distance=2pt] (14)--(15) node[above=1pt, midway, scale=0.65]{${\ps}^!$};
  \draw[->] (15)--(16) node[above=1pt, midway, scale=0.7]{$\couni{\<\!g'}$};

  \draw[->] (21)--(22) node[below=1pt, midway, scale=0.65]{$\couniasterisco{g}$};
  \draw[-, double distance=2pt](22)--(23) node[below=1pt, midway, scale=0.65]{${\ps}^!$};
  \draw[-, double distance=2pt] (23)--(24);
  \draw[-, double distance=2pt] (24)--(25) node[below=1pt, midway, scale=0.65]{${\ps}^!$};
  \draw[->] (25)--(26) node[below=1pt, midway, scale=0.7]{$\couni{\<\!g'}$};
  
 %vertical
  \draw[-, double distance=2pt] (21)--(11) node[left=1pt, midway, scale=0.65]{${\ps}^!$};
  \draw[-, double distance=2pt] (22)--(12) node[left=1pt, midway, scale=0.65]{${\ps}^!$};
  \draw[-, double distance=2pt] (23)--(13) node[left=1pt, midway, scale=0.65]{${\ps}^!$};
  \draw[-, double distance=2pt] (24)--(14) node[right=1pt, midway, scale=0.65]{${\ps}^!$};
  \draw[-, double distance=2pt] (25)--(15) node[right=1pt, midway, scale=0.65]{${\ps}^!$};
  \draw[-, double distance=2pt] (26)--(16) node[right=1pt, midway, scale=0.65]{${\ps}^!$};
 \epic
 \]
 commutativity of \circled3 becomes equivalent to that of 
 \begin{equation}\label{new3}
 \CD
  \bpic[xscale=3, yscale = 1.3]
  
      \node (11) at (1,-1) {$f^! g_*\<g^!h^!$};
      \node (12) at (1.93,-1) {$g'_*g'{}^*\<\<f^! g_*g^!h^!$};
      \node (13) at (3,-1) {$g'_*f'{}^!g^*\< g_*g^!h^!$};
      
      \node (21) at (1,-2) {$f^!h^!$};
      \node (22) at (1.93,-2) {$g'_*g'{}^!\<f^!h^!$};
      \node (23) at (3,-2) {$g'_*f'{}^!g^!h^!$};
      
  %horizontal
      \draw [->] (11) -- (12) node[above, midway, scale=0.75]{$\uniasterisco{g'}$};
      \draw [->] (12) -- (13) node[above, midway, scale=0.75]{$\textup{via}\;\bchadmirado{\Dd}$};

      \draw[->] (22)--(21) node[below=1pt, midway, scale=0.75]{$\couni{\<\!g'}$};
      \draw[-, double distance=2pt] (22)--(23) node[below=1pt, midway, scale=0.75]{$g'_*\ps^!$};
      
  %vertical
        \draw[->] (11)--(21) node[left=1pt, midway, scale=0.75]{$f^!\couni{\<\!g}$};

        \draw [->] (13) -- (23) node[right=1pt, midway, scale=0.75]{$\textup{via}\;\couniasterisco{\<g}$};
  \epic
 \endCD
\end{equation}
which commutativity is an instance of that of \eqref{second diagram}.

The proof of Proposition~\ref{A123} is now complete.
\end{proof}
\vfill

\pagebreak

\section{Realization via Grothendieck duality}\label{realization}

In this section we show that the setup of \S\ref{setup} can be realized in a number of situations involving Grothendieck duality.
  
\begin{cosa}\label{prelim}
 (Notation and summary.)
A \emph{ringed space} is a pair~$(X,\OX)$ such that $X$~is a topological space and $\OX$ is a sheaf of commutative rings on ~$X\<$. Though  only schemes will be of interest in this paper, some  initial results make sense for arbitrary ringed spaces, enabling us to treat several situations simultaneously. For example, it may well be possible to go through all of this section in the context of
noetherian formal schemes, see \cite{AJL}, \cite[7.1.6]{Nk0}.\looseness=-1

A map of ringed spaces $\bar f\colon(X, \OX)\to(Y,\OY)$ is a continuous map
$f\colon X\to Y$ together with a homomorphism of sheaves of rings $\OY\to \fst\OX$. Composition of such maps is defined in the obvious way.  Ordinarily,
$\OX$ and~$\OY$ are omitted from the notation, and one just speaks of ringed-space maps $f\colon X\to Y\<$, the rest
being understood.

For a ringed space $(X,\OX)$,  let $\D(X)$ be the derived category of the abelian category of sheaves of $\OX$-modules, and $T=T_\sX$ its  usual translation automorphism. For $A\in\D(X)$ (object or arrow) and $i\in\ZZ$, set $A[i\>]\set T^{\>i}\<\<A$. 

We take for granted the formalism of relations among the derived functors $\R\sHom$ and $\Otimes{}$
and the derived direct-  and~inverse-image pseudofunctors $\R(-)_*$ resp.~$\LL(-)^*\<$, as presented e.g., in \cite[Chapter 3]{li}.%
\footnote{We will often use \cite {li} as a convenient compendium of needed facts about Grothendieck duality
for schemes. This does not mean that  referred-to results cannot be found in other earlier sources.}
For instance, 
for any $f\colon X\to Y\<$ as above, the functor $\LL f^*\colon\D(Y)\to\D(X)$ is 
left-adjoint to~$\R f^{}_{\<\<*}$, see \cite[3.2.3]{li}; in particular, there are unit and  counit maps  
\begin{equation}\label{baretaeps}
\bar\eta=\uniasteriscode{f},\qquad \bar\epsilon=\couniasteriscode{\<f}. 
\end{equation}

For any $f\colon X\to Y\<$, there are canonical functorial isomorphisms 
$$
\Rf\smallcirc T_\sX\iso T_Y\<\smallcirc \Rf\>,\qquad \LL f^*\<\<\smallcirc T_Y\iso T_\sX\<\smallcirc \LL f^*.
$$
Accordingly, for any $A\in\D(X)$, $B \in\D(Y)$ and $i\in\ZZ$, we will allow ourselves to abuse notation by writing
$$
\Rf \<\big(A[i\>] \>\big)= (\Rf A)[i\>],\qquad \LL f^*\<\<\big(B[i\>] \>\big)= (\LL f^*\<\<B)[i\>].
$$

\begin{subcosa}\label{E_X} 
Let $\bE_\sX$ be the preadditive category whose objects $A,B,C,\dots$ are just those of $\D(X)$, with 
$$
\bE_\sX^i(A,B)\set\Hom_{\D(X)}\!\big(A,B[i\>]\big)\cong \ext_\sX^i(A,B),
$$
and composition determined by the graded $\ZZ\>\>$-bilinear Yoneda product
\[
\bE^i_\sX(B,C)\times \bE^j_\sX(A,B) 
\to
\bE^{i+j}_\sX(A,C)
\]
taking a pair of $\D(X)$-maps $\beta\colon B\to C[i\>]$, $\alpha\colon A\to B[\>j\>]\ (i,j\in\ZZ)$ to the map
$$
(\beta\<\ssscirc\<\alpha)\colon A\xto{\alpha\>\>} B[\>j\>]\xto{\beta[\>j\>]\>}C[i\>][\>j\>]=C[i+j\>].
$$
\end{subcosa}

\begin{subcosa}\label{HX}
In subsection~\ref{D_W}, using their interaction with translation functors,
we enrich the derived direct-  and inverse\kf-image pseudofunctors to an adjoint pair of 
\emph{$\>\>\ZZ$-graded} pseudofunctors $(-)^*$ and $(-)_{\<*}$ on the category of ringed spaces, taking values in the categories $\bE_\sX$.  

\penalty-2000
Then we show in Proposition~\ref{E is H-graded} that 
\[
H_\sX\set\bE_\sX(\OX,\OX)=\oplus_{i\ge0}\,\ext^i_\sX\<(\OX,\OX\<)\cong \oplus_{i\ge 0} \h^i(X,\OX)
\]
with its Yoneda product is a commutative\kern.5pt-graded ring, and that the category $\bE_\sX$ is 
naturally $H_\sX$-graded---whence so is any full subcategory. In fact, Proposition~\ref{tensor graded} gives that $H_\sX$ can be identified with the subring of the graded center of~$\bE_\sX$
consisting of all ``tensor-compatible" elements. Furthermore, Proposition~\ref{* and T} gives that  
for any map $f\colon X\to Y\<$, the functors
$f^*$ and $\fst$ respect such graded structures.
\end{subcosa}

\begin{subcosa} \label{finite pres}
A scheme\kern.5pt-map  $f\colon X\to Y$ is \emph{essentially of finite presentation} if it is quasi-compact and quasi-separated, and if for all $x\in X\<$, the local ring $\mathcal O_{\<\<X\<,\>x}$ is a ring of fractions of a finitely-presentable $\mathcal O_{Y\<\<,\>f(x)}$-algebra. The last condition is equivalent to the existence of affine open neighborhoods $\spec L$ of~$x$ and $\spec K$ of~$f(x)$ such that $L$ is a ring of fractions of a finitely generated $K$-algebra. 

For maps of noetherian schemes, we use in place of ``finite presentation" the equivalent term ``finite type."  
 
\end{subcosa} 

\begin{subcosa}\label{examples}
Now fix a scheme $S$, and
let $\SS$  be one~of: \va1

(a) The category of essentially-finite\kf-type separated perfect (i.e.,  finite tor-dimension) maps of 
noetherian $S$-schemes,  with proper maps as confined maps, and oriented fiber squares with flat bottom arrow as independent squares;

(b) The category of composites of \'etale maps and flat quasi-proper (equivalently, flat quasi-perfect) maps of arbitrary quasi-compact quasi-separated $S$-schemes (see \cite[\S4.7]{li}), with quasi-proper maps confined and all oriented fiber squares\va1  independent.
(The reader who wishes to avoid the technicalities involved  can safely ignore this case (b).)

Conditions (A1), (A2), (B1), (B2) and (C) in \S\ref{the category} are then easily checked.\looseness=-1

As is customary, we will usually denote an object $W\xto{w}S$ in $\SS$ simply by~$W\<$,  with the understanding
that $W$ is equipped with a ``structure map"~$w$.

For any such $W\<$, let $\D_W$ be the full subcategory of $\bE_W$ whose objects are just those of $\Dqc(W)$, 
that is,  $\OX$-complexes whose homology sheaves are all quasi-coherent.
Since for $f\colon X\to Y$ in $\SS$ it holds that $\Lf\Dqc(Y)\subset\Dqc(X)$ \cite[3.9.1]{li} and 
$\Rf\Dqc(X)\subset\Dqc(Y)$ \cite[3.9.2]{li}, 
it follows that the 
pseudo\-functors $(-)^*$ and $(-)_{\<*}$ in~\ref{HX} can  be restricted  to take values in the categories~
$\D_W$. It is assumed henceforth that they are so restricted.

Let $H$ be the commutative\kf-graded ring $H_S\set \bE_S(\OS,\OS)$. 
For any $\SS$-object $w\colon W\to S$, the natural composite map
$$
\bE_S(\OS,\OS)\to\bE_W(w^*\<\OS,w^*\<\OS)\iso\bE_W(\OW,\OW)
$$
is a graded-ring homomorphism from $H_S$ to $H_W$. Hence $\D_W$ is $H\<$-graded, and the adjoint pseudofunctors $(-)^*$ and $(-)_{\<*}$ are $H\<$-graded, see ~\ref{HX}.

We note in Proposition~\ref{theta2} that for an independent square $\Dd$, the associated functorial map~$\theta_{\<\Dd}$ (\S\ref{theta}) is a degree\kf-0 isomorphism.

Thus, we have in place all those elements of a setup that do not involve the pseudofunctor $(-)^!$. 
\end{subcosa}

\begin{subcosa}
In subsections \S\ref{shriek}--\ref{2nd diag}, we treat those elements involving $(-)^!$ by using the twisted inverse\kern.5pt-image pseudofunctor from Grothendieck duality. The twisted inverse\- image is generally defined only for bounded-below complexes. But we want a pseudofunctor with values on
all of $\D_W$. (For instance, we have in mind Hochschild homology, which involves complexes that are bounded above, not below.) That is why we restrict in the examples~\ref{examples}(a) and (b) to maps of finite tor-dimension:   the  twisted inverse image functor $f_{\<\<\upl}^!\>$ that is attached to such a map $f\colon X\to Y$  \emph{extends} to a functor $f^!\colon \Dqc(Y)\to \Dqc(X)$ with
$$
f^!\<C:= f_{\<\<\upl}^!\OY\Otimes{\sX}\LL f^*\<C \qquad (C\in\Dqc(Y)).
$$
``Extends" means that for cohomologically bounded-below $C\in\Dqc(Y)$,  there is a canonical functorial isomorphism
$$
f^!C\iso f_{\<\<\upl}^!C.
$$
(For case (a), see \cite[5.9]{Nk}; for (b),  \cite[4.7.2]{li}).
This extension can be made pseudofunctorial (\S\ref{shriek}), and $H\<$-graded,
the latter as a consequence of the compatibility of $\Otimes\sX$ and $\Lf$ with the $H_\sX\<$-grading on $\bE_\sX$ (Propositions~\ref{tensor graded} and ~\ref{* and T}).

In \S\ref{bchadmirado} we associate to each independent square $\Dd$ a base-change isomorphism 
$\bchadmirado\Dd$ as in \S\ref{bchado}, for which the diagrams ~\eqref{basechange1}
and~ \eqref{basechange2} commute. In \S\ref{integral}, we associate to each confined map $f$ a  degree\kf-0 functorial map
$\couni{\<\<\!f}\colon f^{}_{\<\<*}f^!\to \id$ that satisfies transitivity (see \S\ref{f_*}).

We conclude by showing that with the preceding data, diagrams~\eqref{first diagram} 
and~\eqref{second diagram} commute, thereby establishing all the properties of a setup. 
\end{subcosa}

\end{cosa}

\begin{cosa}\label{D_W}

Let $f\colon X\to Y$ be a ringed-space map. For any object $C$ in $\bE_Y$, denote the derived inverse image $\LL f^*\<C\in\bE_\sX$ simply by $f^*\<C$.
(Despite this notation, it should not be forgotten that we will be dealing throughout with derived functors.) To any map $\gamma\colon C\to D[i\>]$ in $\bE_Y^i(C,D)$ assign the map
$$
f^*\<\<\gamma\colon f^*\<C\xto{\LL f^*\<\<\gamma\>\>}f^*\<\big(D[i\>]\big)=(f^*\<\<D\big)[i\>]
$$
in $\bE_\sX^i(f^*\<C,f^*\<\<D)$.  Using  functoriality of the  isomorphism  represented by ``="  (see \S\ref{prelim}), one checks that this assignment is compatible with composition in $\bE_Y$ and~$\bE_\sX$; so one gets
 a $\ZZ\>$-graded functor $f^*\colon \bE_Y\to \bE_\sX$.

In a similar manner,  the derived direct image functor $\R f^{}_{\<\<*}$ gives rise to a $\ZZ\>$-graded 
functor $f^{}_{\<\<*}\colon\bE_\sX\to \bE_Y$.\va1

\begin{subprop}\label{adjunction} 
There is an adjunction $f^*\<\dashv f^{}_{\<\<*}\>$ for which the corresponding unit and counit maps
\begin{equation*}
\eta\colon\id \to f^{}_{\<\<*}f^*\quad\text{ and }\quad\epsilon\colon f^*\<\<f^{}_{\<\<*}\to\id 
\end{equation*}
are degree-$\>0$ maps of\/  $\>\>\ZZ$-graded functors.
\end{subprop}

\pagebreak[3]

\begin{proof} \!\!Let $\eta^{}_C\<\<\in\bE_Y^0(C, f^{}_{\<\<*}f^*\<C)$  be\va1 the $\D(Y)$-map 
$\bar\eta^{}_C\colon \<C\<\to \<f^{}_{\<\<*}f^*\<C$ (see~\!\eqref{baretaeps}\!) and 
$\epsilon^{}_{\!A}\in\bE_\sX^0(f^*\<\<f^{}_{\<\<*}A, A)$  the $\D(X)$-map
$\bar\epsilon^{}_{\!A}\colon f^*\<\<f^{}_{\<\<*}A\to A$. \va1

That the compositions
$$
f^{}_{\<\<*}A \xto{\eta^{}_{\<f^{}_{\<\<*}{\<A}}\>\>}  f^{}_{\<\<*}f^*\<\<f^{}_{\<\<*}A \xto{f^{}_{\<\<*}\epsilon^{}_{\!A}\>\>} f^{}_{\<\<*}A,\quad
f^*\<C \xto{f^*\<\eta^{}_{C}\>} f^*\<\<f^{}_{\<\<*}f^*\<C \xto{\epsilon^{}_{\<\<f^*\<C}} f^*\<C
$$
are identity maps
follows from the corresponding properties of $\bar\eta$ and $\bar\epsilon$.
It remains then to show that
the family $\eta^{}_C\ (C\in\bE_Y)$ (resp.~$\epsilon^{}_{\!A}\ (A\in\bE_\sX)$) constitutes a
degree\kern.5pt-0 map of graded  functors. For $\eta^{}_C$ this means that for any $\D(Y)$-map
$\gamma\colon C\to D[i\>]\ (i\in\ZZ)$ the following $\D(Y)$-diagram commutes:
\begin{equation}\label{compid}
\CD
C     @>\gamma>> D[i\>]  @>\bar\eta^{}_{\<D}[i\>]>>(f^{}_{\<\<*}f^*\<D)[i\>] \\
@V\bar\eta^{}_{C} V\mkern65mu\mbox{\circled1}V    @V\bar\eta^{}_{\<D[i\>]\<} 
V\mkern78mu\mbox{\circled2}V  @|  \\
f^{}_{\<\<*}f^*\<C  @>> \Rf\LL f^*\gamma\>>  f^{}_{\<\<*}f^*\mkern-1.5mu\big(D[i\>]\>\big) 
      @= f^{}_{\<\<*}\big((f^*\<D)[i\>]\>\big)
\endCD
\end{equation}

Commutativity of subdiagram \circled1 is clear.  

For commutativity of \circled 2, replace~$D$ by a quasi-isomorphic q-flat complex, and note that the natural map from the derived inverse image to the underived inverse image of $D$ is then an isomorphism, see \cite[paragraph surrounding 2.7.3.1]{li}. Then, with $\tilde f^{}_{\<\<*}$~denoting the underived direct-image functor, consider the following cube, in which the front face is~\circled2 and the maps are the natural ones:
\[
  \bpic[xscale=1.6, yscale=1.25]

   \draw (0,-1) node (11){$$};
   \draw (1.05,-1) node (12){$D[i\>]$};
   \draw (2,-1) node (13){};
   \draw (3,-1) node (14){$(\tilde f^{}_{\<\<*}f^*\<\<D)[i\>]$};
    
   \draw (0,-2) node (21){$D[i\>]$};
   \draw (1.05,-2) node (22){};
   \draw (2,-2) node (23){$(\R\tilde f^{}_{\<\<*}f^*\<\<D)[i\>]$};
   \draw (3,-2) node (24){};
    
   \draw (0,-3) node (31){$$};
   \draw (1.05,-3) node (32){$\tilde f^{}_{\<\<*}f^*\<\big(D[i\>]\big)$};
   \draw (2,-3) node (33){};
   \draw (3,-3) node (34){$\tilde f^{}_{\<\<*}\<\big((f^*\<\<D)[i\>]\big)$};
   
   \draw (0,-4) node (41){$\R\tilde f^{}_{\<\<*}f^*\<\big(D[i\>]\big)$};
   \draw (1.05,-4) node (42){};
   \draw (2,-4) node (43){$\R\tilde f^{}_{\<\<*}\<\big((f^*\<\<D)[i\>]\big)$};
   \draw (3,-4) node (44){}; 
   
  %horizontal
  
   \draw[->] (12)--(14);
  
   \draw[->] (21)--(23);

   \draw[-, double distance=2pt] (32)--(33);
   \draw[-, double distance=2pt] (33)--(34);

   \draw[-, double distance=2pt] (41)--(43);

 %vertical
 
   \draw[-] (12)--(22);
   \draw[-, double distance=2pt](14)--(34) ;

   \draw[->](22)--(32) ;
   \draw[->] (21)--(41) ;
   \draw[-, double distance=2pt] (23)--(43) ;

 %slanted
 
   \draw [-, double distance=2pt] (12) -- (21) ;
   \draw[->] (14)--(23);
   
   \draw[->] (32)--(41);
   \draw[->] (34)--(43);

 %labels
      
     \epic
\]
Commutativity of the bottom face is clear. Commutativity of the top and left faces results from \cite[3.2.1.3]{li}. To make commutativity of the right face clear, replace the complex $f^*\<\<D$ by 
a quasi-isomorphic q-injective complex~$J$, and note that the canonical map 
$\tilde \fst J\to\R \tilde \fst J$ is a $\D(Y)$-isomorphism (see \cite[2.3.5]{li}).
Commutativity of the rear face, which involves only
underived functors, is an easy consequence of the definition of the standard functorial map
$\id\to \tilde f^{}_{\<\<*}f^*\<$. Commutativity of the front face follows from that of the others.

An analogous argument, using \cite[3.2.1.2]{li}, applies to the family $\epsilon^{}_{\!A}\>$.
\end{proof}

\begin{subcor}\label{pseudofunctoriality}
There exist pseudofunctorially adjoint\/ $\ZZ$-graded 
pseudofunctors that associate the  functors\/ $f^*$ and\/ $\fst$ to any $\SS$\kf-map\/ $f\colon X\to Y$.
\end{subcor}

\begin{proof}
For any $X\xto{f}Y\xto{g}Z$ in $\SS$, there are functorial  isomorphisms 
$$
\ps_*\colon (gf)_*\iso g^{}_*\<\fst,\qquad \ps^*\colon f^*\<\<g^*\iso(gf)^*
$$
such that for $A\in\bE_\sX$, the corresponding map $(gf)_*A\iso g^{}_*\fst A$ is the natural 
$\D(Z)$-isomorphism $\>\overline{\<\ps\<\<}\>\>_*\colon\R(gf)_*A\iso \R g^{}_*\Rf A$, and such that for $C\in\bE_Z$, the corresponding map $f^*\<\<g^*C\iso (gf)^*C$ is the natural 
$\D(X)$-isomorphism $\>\overline{\<\ps\<\<}\>\>^*\colon\LL f^*\LL g^*C\iso\LL (gf)^*C$.  Now,
$\>\overline{\<\ps\<\<}\>\>_*$ is a map of so-called $\Delta$-functors (see \cite[2.2.7]{li}); and it follows readily that $\ps_*$ is of degree 0.  A similar argument applies to $\ps^*\<$.

That the first diagram in \eqref{assoc} commutes, as does its analog for $(-)_*$, follows from the corresponding facts for the pseudofunctors
$\LL (-)^*$ and $\R(-)_*$. Hence $\ps^*$ makes $(-)^*$ into a contravariant $\ZZ\>$-graded pseudofunctor, and $\ps_*$ makes $(-)_*$ into a covariant $\ZZ\>$-graded pseudofunctor. 
The adjointness of these pseudofunctors, that is, commutativity of~ 
\eqref{adjpseudo}, results from that of the corresponding diagram for the adjoint pseudofunctors
$\LL(-)^*$ and $\R(-)_*$ (see \cite[3.6.10]{li}).
\end{proof}

 From  \cite[3.9.5]{li}), one gets:
\begin{subprop}\label{theta2}
With $f^*\dashv \fst$  as above, 
for any independent $\SS$-square
\begin{equation*}\CD
\bullet @>v>> \bullet\\
@V g VV @VV f V \\
\bullet@>{\vbox to 0pt{\vss\hbox{$\sst\Dd$}\vskip.21in}} > \lift1.2,u, >\bullet
\endCD
\end{equation*}
the map $\theta_{\<\Dd}\colon u^*\!\fst\to g_*v^*$ in \S\ref{theta} is a functorial isomorphism
of degree\/ $0$.

\begin{proof}
That $\theta_\Dd$ has degree 0 results from the fact that it is a composition of three functorial maps
$$
u^*\!\fst\xto{\eta^{}_g\>} g_*\>g^*u^*\!\fst\overset{\ps^*}{=\!=}g_*v^*\!f^*\!\fst\xto{\epsilon^{}_{\!f}\>} g_*v^*
$$
all of which are of degree 0 (see ~\ref{adjunction} and the proof of~\ref{pseudofunctoriality}).

The rest is clear. 
\end{proof}\end{subprop}

\end{cosa}

\begin{cosa}
For a scheme $(X,\OX)$, if $A$ and $B$ are $\OX$-complexes and $i,j,n\in\ZZ$, then since
$$
\big(A[i\>]\otimes_\sX B[\>j\>]\big){}^{\<n}=\bigoplus_{p+q=n+i+j}\!A^p\otimes_\sX B^q
= \big(A\otimes_\sX \<B\big)[i\<+\<j\>]^n,
$$
therefore there is a unique isomorphism of graded $\OX$-modules 
$$
\vartheta'_{ij}\colon A[i\>]\otimes_X B[\>j\>]\iso \big(A\otimes_\sX B\big)[i\<+\<j\>]
$$
whose restriction to $A^p\otimes_\sX B^q\ (p,q\in\ZZ)$ is multiplication by $(-1)^{(p-i)j}$.
One~checks that $\vartheta'_{ij}$ is actually a bifunctorial isomorphism of $\OX$-complexes.

\begin{sublem} \label{vartheta}
For any\/ $i,j\in\ZZ$ there exists  a unique bifunctorial isomorphism\/ $\vartheta_{ij}$
such that for any\/ $\OX$-complexes\/ $A$ and\/ $B,$ the following diagram in~$\D(X)$ commutes.
$$
\CD
A[i\>]\Otimes{\sX} B[\>j\>] @>\vartheta_{ij} >> \big(A\Otimes{\!X} B\big)[i\<+\<j\>]\\
@V\textup{canonical} VV @VV\textup{canonical} V\\
A[i\>]\otimes_\sX\< B[\>j\>] @>> \lift1.4,\vartheta'_{\<\<ij},>  \big(A\otimes_\sX\< B\big)[i\<+\<j\>]
\endCD
$$
\end{sublem}

\noindent\emph{Proof.} The idea is to apply $\vartheta'_{\<\<ij}$  to suitable q-flat resolutions of $A$ and $B$. \va1

More precisely, every $\OX$-complex is the target of a quasi-isomorphism from a q-flat complex,
and for q-flat complexes the canonical functorial map from the derived tensor product 
 $\Otimes{\!X}$  to  the ordinary tensor product $\otimes_\sX$ is an isomorphism \cite[\S2.5]{li}; hence the assertion follows from \cite[2.6.5]{li} (a~general method for  constructing maps of derived multifunctors), dualized---i.e., with arrows reversed, in which, with abbreviated notation, take 
 \begin{itemize}
\item $L_1''=L_2''$ to be the homotopy category $\mathbf K(X)$ 
of $\OX$-complexes, 
\item $L_k'\subset L_k''\ (k=1,2)$ the full subcategory whose objects are the q-flat complexes,
\item $\bE\set\D(X)$, 
\item $H$ the  functor taking 
$(A,B)\in L_1''\times L_2''$ to $\big(A\otimes_\sX B\big)[i\<+\<j\>]\in\D(X)$
(and acting in the obvious way on arrows), 
\item $G$ the  functor 
$(A,B)\in \D(X)\times \D(X)\mapsto A[i\>]\Otimes{\!X} B[\>j\>]\in\D(X)$,
\item $F$ the functor  
$(A,B)\in \D(X)\times \D(X)\mapsto\big(A\Otimes{\!X} B\big)[i\<+\<j\>]\in\D(X)$,
\item $\zeta\colon F\to H$ the canonical functorial map, and
\item  $\beta\colon G\to H$ the canonical functorial composite\va{-3}
\begin{xxalignat}{3}
&{\phantom{\square\quad}}
&A[i\>]\Otimes{\!X}\<\< B[\>j\>]\lto A[i\>]\otimes_\sX\<\< B[\>j\>]\xto{\vartheta'_{\<\<ij}} 
\big(A\otimes_\sX \<\<B\big)[i\<+\<j\>].
&&\qed
\end{xxalignat}

\end{itemize}
\end{cosa}

\begin{prop}\label{E is H-graded}
The ring
$$
H_\sX\set\bE_\sX(\OX,\OX)=\oplus_{i\ge0}\,\ext^i_\sX\<(\OX,\OX\<)
\cong \oplus_{i\ge0}\,\textup{H}^i(X\<\<,\OX\<)
$$
is  canonically a graded-ring retract of the graded center\/ $\CC_{\bE_\sX}\<.$ Hence\/ $H_\sX\<$ is graded-commutative, and\/
$\bE_\sX$ \emph{is $H_\sX\<$-graded.}\va1
\end{prop}

\begin{proof} By \S\ref{unital->gradedcomm}, the assertion follows from the existence of a unital product\/  
$(\totimes\!,\OX,\lambda,\rho)$---to be constructed---on the\/ preadditive category\/~$\bE_\sX$.\va1

Define a $\ZZ\>$-graded functor 
\begin{equation}\label{totimes}
\totimes\!\colon \bE_\sX\!\botimes\ZZ\bE_\sX\to\bE_\sX 
\end{equation}
as follows. (Notation will be as in \S\ref{unital}.)\va1

First, for any object $(A,B)\in\bE_\sX\!\botimes\ZZ\bE_\sX$,  
$A\totimes\<B\set\totimes\!(A,B)$ is the derived tensor product $A\Otimes{\!X}B$, 
which lies in $\bE_\sX$ \cite[p.\,64, 2.5.8.1]{li}. 

\penalty-2000
Next, the map taking $(\alpha^{}_1, \alpha^{}_2)\in\bE_\sX^i(A_1,B_1)\times \bE_\sX^j(A_2,B_2)$
to the map
$$
\alpha_1\totimes\alpha_2\in\bE_\sX^{i+j}(A_1\totimes A_2, B_1\totimes\<B_2)
$$ 
given by the
composite $\D(X)$-map 
$$
A_1\Otimes{\!X}A_2\xto{\alpha^{}_1\>\Otimes{\!X}\>\alpha^{}_2\>\>}B_1[i\>]\Otimes{\!X}B_2[\>j\>]
\xto[\eqref{vartheta}]{\vartheta_{ij}}\big(B_1\<\Otimes{\!X}B_2\big)[i+j\>]
$$
is $\ZZ\>$-bilinear, so factors uniquely through a map
$$
\totimes^{\!ij}\colon\bE_\sX^i(A_1,B_1)\otimes_\ZZ\bE_\sX^j(A_2,B_2)\to 
\bE_\sX^{i+j}(A_1\totimes A_2, B_1\totimes\<B_2)
$$
taking $\alpha_1\<\otimes\<\alpha_2$ to $\alpha_1\totimes\alpha_2\>$; and $\totimes^{\!ij}$  extends uniquely to a $\ZZ\>$-linear map
\begin{multline*}
\totimes\!\colon(\bE_\sX\botimes \ZZ\bE_\sX)\big((A_1,A_2),\:(B_1,B_2)\big)= \\
\bE_\sX(A_1,B_1)\otimes_\ZZ \>\bE_\sX(A_2,B_2)
\to 
\bE_\sX(A_1\totimes A_2\>, B_1\totimes\<B_2).
\end{multline*}

For functoriality, it needs to be checked that for all $A_1\xto{\alpha_1\,}B_1\xto{\beta_1\,}C_1$  and $A_2\xto{\alpha_2\,}B_2\xto{\beta_2\,}C_2$ in $\bE_\sX$,\va1 with
$\alpha_1\in\bE_\sX^{m_1}\<(A_1,B_1)$ and $\beta_2\in\bE_\sX^{n_2}(B_2,C_2)$, it holds that 
$$
(\beta_1\otimes \beta_2)\ssscirc(\alpha_1\otimes \alpha_2)= (-1)^{n^{}_2m^{}_1}(\beta_1\ssscirc\alpha_1)\otimes(\beta_2\ssscirc\alpha_2)\colon A_1\otimes A_2\to C_1\otimes C_2.
$$
This straightforward verification is left to the patient reader.

Specializing, one gets the $\ZZ\>$-graded endofunctor $\OX\totimes^{}-$ of $\bE_\sX$, taking an object
 $A\in\bE_\sX$ to $\OX\Otimes{\!X} A$, and a $\D(X)$-map $\alpha\colon A\to B[\>j\>]$ in 
 $\bE_\sX^j(A,B)$ to the composite $\D(X)$-map, in $\bE_\sX^j(\OX\totimes A,\OX\totimes\<B)$,
$$
\OX\Otimes{\!X} A\xto{\!\textup{via}\:\alpha\>\>} \OX\Otimes{\!X} B[\>j\>]
\overset{\vartheta_{\<0j}}{=\!\!=} \big(\OX\Otimes{\!X} B\big)[\>j\>].
$$
Similarly, one has the $\ZZ\>$-graded endofunctor $-\totimes\OX$.
There are obvious degree\kern.5pt-0 functorial isomorphisms
$$
\lambda\colon(\OX\totimes -)\iso \id_{\bE_\sX},\qquad\rho\colon (-\totimes \OX)\iso \id_{\bE_\sX}\<.
$$

It is immediate that $(\totimes\!,\OX,\lambda,\rho)$ is a unital product, so we are done.
\end{proof}

\begin{subcor}\label{E+graded}
Any full subcategory of\/ $\bE_\sX$ has an
$H_\sX\<\<$-grading,  inherited from the preceding one on\/ $\bE_\sX$.\qed
\end{subcor}
\vskip1pt

The preceding $\ZZ\>$-graded unital product is in fact $H_\sX\<$-graded. This results from the following characterization of $H_\sX\subset \CC_{\bE_\sX}$.

\begin{prop}\label{tensor graded} 
With notation as in \ref{E is H-graded} and its proof, the following conditions on\/ 
$\xi\in \CC^n_{\bE_\sX}$ are equivalent:\va2

{\rm(i)} $\xi\in H^n_\sX=H^n(X\<,\OX)$.

{\rm(ii)} For all\/ $(\alpha, \beta)\in\bE_\sX^i(A,C)\times \bE_\sX^j(B,D),$ it holds that
$$
(\xi\>\alpha)\totimes\beta=\xi(\alpha\totimes\beta),\quad
\alpha\totimes(\beta \>\xi)=(\alpha\totimes\beta)\xi\>,\quad\textup{and}\quad
(\alpha\>\xi)\totimes\beta=\alpha\totimes(\xi\>\beta).
$$
\end{prop} 

\begin{proof}
(i)$\,\Rightarrow\>\>$(ii). Since
$$
(\xi\>\alpha)\totimes\beta=(\xi^{}_C\totimes \id_D)\smallcirc (\alpha\totimes\beta)
\quad\textup{and}\quad 
\xi(\alpha\totimes\beta) = (\xi^{}_{C\totimes D})\smallcirc(\alpha\totimes\beta)
$$
therefore, for the first equality, one need only show that 
\begin{equation}\label{id case}
(\xi^{}_C\totimes \id_D)=\xi^{}_{C\totimes D\>}.
\end{equation} 

Similarly, the second equality reduces to 
\begin{equation}\label{id case2}
(\id_C\totimes \xi^{}_D)=\xi^{}_{C\totimes D\>}.
\end{equation} 

The third equality results from the first two, since the hom-sets  $\bE_\sX(-,-)$ are symmetric graded $\CC^n_{\bE_\sX}$-modules.\va1

In other words, one need only treat the case 
where $\alpha\colon A=C\to C$ and  $\beta\colon B=D\to D$ are the identity maps $\id_C$ and $\id_D$ respectively.  

The equality \eqref{id case} is equivalent to the obvious commutativity of the natural 
$\D(X)$-diagram, where $\otimes\set\Otimes{\sX}\>$, 
\[
 \bpic [xscale=4.2, yscale=2.3]
 
  \node (11) at (1,-2) {$\OX \otimes C \otimes D$};
  \node (12) at (1,-1) {$C \otimes D$};
  \node (13) at (2,-1) {$\OX \otimes C \otimes D$};
  \node (14) at (3,-1) {$\OX[n] \otimes C \otimes D$};
  
  \node (21) at (1,-3) {$\OX[n] \otimes C \otimes D$};
  \node (23) at (3,-2) {$C[n] \otimes D$};

  \node (31) at (2,-3) {$(\OX \otimes C \otimes D)[n]$};
  \node (33) at (3,-3) {$(C\otimes D)[n]$};

  \draw[->] (12)--(11);
  \draw[->] (12)--(13);  
  \draw[->] (13)--(14) node[midway, above=1pt, scale=.75]{$\!\xi\Otimes\sX \<\<\id_{\>C\<\otimes \<D\>}$};
 
  \draw[->] (23)--(33);
  \draw[->] (31)--(33); 
  
  \draw[->] (11)--(21) node[midway, left=1pt, scale=.75]{$\xi\Otimes\sX \<\<\id_{\>C\<\otimes \<D}$};
  \draw[-, double distance=2pt] (21)--(31) node[midway, below=1pt, scale=.75]{$\vartheta_{n0}$};

  \draw[-, double distance=2pt] (23)--(33) node[midway, right=1pt, scale=.75]{$\vartheta_{n0}$};
  \draw[-, double distance=2pt] (14)--(23) node[midway, right=1pt, scale=.75]{$\vartheta_{n0}\Otimes\sX \id_D$};
  
   \draw[-, double distance=2pt] (13)--(11);
   \draw[-, double distance=2pt] (14)--(21);
 \epic
\]

As for \eqref{id case2}, let $\tau'(A,B)\colon A\otimes_\sX B\iso B\otimes_\sX A$ be the unique bifunctorial isomorphism of 
$\OX$-complexes that restricts on 
$A^p\otimes_\sX B^q$ to the map taking $a\otimes b$ to $(-1)^{pq} (b\otimes a)\in B^q\otimes A^p\ (p,q\in\ZZ)$.  One shows as in Lemma~\ref{vartheta} that there~is a unique bifunctorial $\D(X)$-isomorphism 
$\tau(A,B)$
such that for any  $A$ and~$B$ the following $\D(X)$-diagram commutes:
\[
 \bpic[xscale=3.2, yscale=2.3]
   
   \node(11) at (1,-1){$A\Otimes\sX B$};
   \node(12) at (2,-1){$B\Otimes\sX A$};
   \node(21) at (1,-2){$A\otimes_\sX B$};
   \node(22) at (2,-2){$B\otimes_\sX A$};
 
  \draw[->]   (11)--(12)   node[midway, above=1pt,  scale=.75]{$\tau$};
  \draw[->]   (11)--(21)  node[midway, left=1pt,  scale=.75]{canonical};
  \draw[->]   (12)--(22)   node[midway, right=1pt, scale=.75]{canonical};
  \draw[->]   (21)--(22)  node[midway, below=1pt,  scale=.75]{$\tau'$};

 \epic
\]
The equality \eqref{id case2} is equivalent to commutativity of the border of the natural diagram 
\[
 \bpic [xscale=4.4, yscale=1.5]
 
  \node (11) at (1,-2) {$\OX \otimes C \otimes D$};
  \node (12) at (1,-1) {$C \otimes D$};
  \node (13) at (1.9,-1) {$C \otimes \OX \otimes D$};
  \node (14) at (3,-1) {$C \otimes \OX[n] \otimes D$};
  
  \node (21) at (1,-3) {$\OX[n] \otimes C \otimes D$};
  \node (23) at (3,-2) {$C[n] \otimes D$};

  \node (31) at (1.9,-3) {$(\OX \otimes C \otimes D)[n]$};
  \node (33) at (3,-3) {$(C\otimes D)[n]$};

  \draw[->] (12)--(11);
  \draw[->] (12)--(13);  
  \draw[->] (13)--(14) node[midway, above=1pt, scale=.75]{$\<\id_C\Otimes\sX\>\xi\Otimes\sX \<\<\id_D$};
 
  \draw[->] (23)--(33);
  \draw[->] (31)--(33); 
  
  \draw[->] (11)--(21) node[midway, left=1pt, scale=.75]{$\xi\Otimes\sX \<\<\id_{\>C\<\otimes \<D}$};
  \draw[-, double distance=2pt] (21)--(31) node[midway, below=1pt, scale=.75]{$\vartheta_{n0}$};

  \draw[-, double distance=2pt] (23)--(33) node[midway, right=1pt, scale=.75]{$\vartheta_{n0}$};
  \draw[-, double distance=2pt] (14)--(23) node[midway, right=1pt, scale=.75]{$\vartheta_{0n}\Otimes\sX \id_D$};
  
   \draw[->] (11)--(13) node[midway, right] {$\lift.5,{\tau(\OX\<,\>\>C)\Otimes\sX\id_D},$};
   \draw[->] (21)--(14) node[midway, right] {$$\lift.5,{\tau(\OX[n],\>C)\Otimes\sX\id_D},$$};
   
   \node at (1.25,-1.35){\circled1};
  \node at (1.6,-2){\circled2};
  \node at (2.3,-2.55){\circled3};
 \epic
\]
Commutativity of subdiagram \circled1 is easily checked. Commutativity of \circled2 holds by functoriality of $\tau$. For commutativity of \circled3, one checks, taking signs into account, that both paths from $\OX[n] \otimes C \otimes D$ to
$(C\otimes D)[n]$ have the same restriction to each  $\OX[n] \otimes C^p \otimes D^q\ (p,q\in\ZZ)$.

The desired conclusion results. \va2
 
(ii)$\,\Rightarrow\>\>$(i). For $\alpha=\id_{\OX}\in\bE^0_\sX(\OX,\OX)$ and $\beta=\id_{\<A}\in\bE^0_\sX(A,A)$  the identity maps, the third equality in condition (ii) yields 
$$
\xi^{}_{\OX}\totimes \id_{\<A}=\id_{\OX}\totimes \xi^{}_A\>.
$$
In other words, in the following $\D(X)$-diagram---where  unlabeled arrows represent the natural 
isomorphisms---subdiagram \circled4 commutes:\va2
\[
 \bpic[xscale=4.2, yscale=1.75]
  
  \node(11) at (1,-1){$A$};
  \node(13) at (3,-1){$A[n]$};
  
  \node(22) at (2,-2){$\OX\<\Otimes\sX\< A[n]$}; 
  
  \node(31) at (1,-2){$\OX\<\Otimes\sX\< A$};
  \node(33) at (3,-2){$(\OX\<\Otimes\sX\< A)[n]$};
  
  \node(42) at (2,-3.1){$\OX[n]\<\Otimes\sX\< A$};   
  
  \draw[->] (11)--(13) node[midway, above=.5pt, ]{$\lift1.5,\xi^{}_{\<\<A},$};
  \draw[->] (11)--(31);
  \draw[->] (31)--(22) node[midway, above]{$\lift2,\id_{\OX}\!\Otimes{\<\sX}\>\xi^{}_{\<\<A},$};
  \draw[->] (22)--(13);
  \draw[-, double distance=2pt] (22)--(33) node[midway, above]{$\lift1.5,\vartheta_{0n},$};
  \draw[->] (31)--(42) node[midway, left=-1pt]{$\lift.1,\xi^{}_{\OX}\!\Otimes{\<\sX} \id_{\<A},$};
  \draw[->] (33)--(13);
  \draw[-, double distance=2pt] (42)--(33) node[midway, right=1pt]{$\lift-.2,\vartheta_{n0},$};
  
  \node at (2,-2.55){\circled4};
 \epic
\]
The other two subdiagrams clearly commute, so the border commutes. But by definition, the counterclockwise path from the upper left corner to the upper right corner is $\xi'_A\>$, where $\xi'$ is the canonical image
in $\CC^n_{\bE_\sX}$ of the element $\xi^{}_{\OX}\in\bE^n_\sX(\OX,\OX)=H^n(X,\OX)$. Thus, 
after identification of~$H^n(X,\OX)$ with its image in~$\CC^n_{\bE_\sX}$, we have
$\xi=\xi'\in H^n(X,\OX)$.
\end{proof}

\begin{cosa}
Let $f\colon X\to Y$ be a ringed-space map. The natural composition
$$
\mu^{}_{\<\<f}\colon\bE_Y(\OY,\OY)\to\bE_\sX(f^*\<\OY, f^*\<\OY)\iso\bE_\sX(\OX,\OX)
$$
is a graded-ring homomorphism from $H_Y$ to $H_\sX$. Hence, from~\ref{E+graded}, one gets an
$H_Y$-grading on any  full subcategory of $\bE_\sX$.

\end{cosa}

\pagebreak[3]
The graded functors $f^*$ and $f^{}_{\<\<*}$ of \S\ref{D_W} are actually $H_Y$-graded:

\begin{subprop}\label{* and T}    
Let\/ $f\colon \!X\<\to\< Y$ be a ringed-space map, and\/ $C\<\in\D(Y),$ \mbox{$D\<\in\D(Y),$} 
$A\in\D(X)$ and\/ $B\in\D(X)$. 
\vspace{2pt}

{\rm(i)} The  map\/ $f^*\colon\bE_Y(C,D)\to$ 
 $\bE_\sX( f^*\<C, f^*D)$ 
is\/ $H_Y$-linear.\vspace{1pt}

{\rm(ii)} The map\/ $f^{}_{\<\<*}\colon\bE_\sX(A,B)\to$ 
 $\bE_Y( f^{}_{\<\<*}A, f^{}_{\<\<*}B)$ 
is\/ $H_Y$-linear.\vspace{1pt}

{\rm(iii)} If\/ $C=D$ $($respectively $A=B)$ then the map in\/ {\rm(i)}
$($respectively {\rm(ii))}
is a homomorphism of graded\/ $H_Y$-algebras.
\end{subprop}

\begin{proof} (i) We need to show,  for 
$$
\gamma\colon C\to D[i\>] \textup{ in }  \bE_Y^i(C,D) \quad\textup{and} \quad 
h\colon\OY\to\OY[n] \textup{ in } \bE_Y^n(\OY,\OY)=H^n_Y, 
$$
that $f^*\<(\gamma h)= (f^*\gamma)h$---whence by symmetry, $f^*\<(h\gamma)= h(f^*\gamma)$.
Underlying definitions show that the equality in question amounts to commutativity of  the border of the next diagram~\eqref{diag(i)}, where the unlabeled maps are natural (see \cite[3.2.4(i)]{li}), and ``$=$"  represents various canonical isomorphisms.\va{-6}

\begin{figure}[ht]

\begin{equation}\label{diag(i)}
\def\1{$f^*\<C$}
\def\2{$f^*(\OY\<\<\Otimes Y C)$}
\def\3{$f^*\OY\<\<\Otimes \sX\<\< f^*\<C$}
\def\4{$\OX\<\<\Otimes \sX\<\< f^*\<C$}
\def\5{$f^*(\OY[n]\<\<\Otimes Y C)$}
\def\6{$f^*(\OY[n])\<\<\Otimes {\sX}\<\<f^*\<C$}
\def\7{$f^*(\OY)[n]\<\<\Otimes {\sX}\<\<f^*\<C$}
\def\8{$\OX[n]\<\<\Otimes {\sX}\<\<f^*\<C$}
\def\9{$f^*\<\<\big((\OY\<\<\Otimes Y C)[n]\big)$}
\def\ten{$\mkern-5mu(\OX\<\<\Otimes {\sX}\<\<f^*\<C)[n]$}
\def\lvn{$f^*\<\<\big(C[n]\big)$}
\def\twv{$(f^*\<C)[n]$}
\def\thn{$f^*\<\<\big(D[i\>][n]\big)$}
\def\frn{$\big(f^*\<D[i\>]\big)[n]$}
\def\ffn{$f^*\<\<\big(D[i+n]\big)$}
\def\sxn{$(f^*D)[i+n]$}
\def\stn{$(f^*D)[i\>][n]$}
\!
 \bpic[xscale=2.47,yscale=1.5]
 
  \node(13) at (3,-1)[scale=.95] {\1};

  \node(21) at (1,-2) [scale=.95] {\2};
  \node(22) at (2.33,-2) [scale=.95] {\3};
  \node(25) at (5,-2) [scale=.95] {\4};

  \node(31) at (1,-3) [scale=.92] {\5};
  \node(32) at (2.33,-3)[scale=.92]  {\6};
  \node(34) at (3.75,-3) [scale=.92] {\7};
  \node(35) at (5,-3) [scale=.92] {\8};

  \node(41) at (1,-4) [scale=.95] {\9};
  \node(45) at (5,-4)[scale=.95]  {\ten};
 
  \node(51) at (1,-5) [scale=.95] {\lvn};
  \node(55) at (5,-5) [scale=.95] {\twv};

  \node(61) at (1,-6) [scale=.95] {\thn};
  \node(65) at (5,-6) [scale=.95]  {\frn};

  \node(71) at (1,-7) [scale=.95] {\ffn};
  \node(73) at (3,-7) [scale=.95] {\sxn};
  \node(75) at (5,-7) [scale=.95] {\stn};
 
 %horizontal
  \draw[->] (21)--(22) ;
  \draw[->] (22)--(25) ;
  
  \draw[->] (31)--(32) ;
  \draw[-, double distance=2pt] (32)--(34);
  \draw[->] (34)--(35) ;

  \draw[-, double distance=2pt] (51)--(55);

  \draw[-, double distance=2pt] (61)--(65);

  \draw[-, double distance=2pt] (71)--(73);
  \draw[-, double distance=2pt] (73)--(75);
  
 %vertical
    \draw[->] (21)--(31) node[left=1pt, midway, scale=.7]{$\textup{via}\;h$};
    \draw[->] (22)--(32) node[right=1pt, midway, scale=.7]{$\textup{via}\;h$};
    \draw[->] (25)--(35) node[right=1pt, midway, scale=.7]{$\textup{via}\;\mu_{\<\<f}(h)$};

    \draw[-, double distance=2pt] (31)--(41) node[left=1pt, midway, scale=.7]{$\LL f^*\vartheta_{n0}$};
    \draw[-, double distance=2pt] (35)--(45) node[right=1pt, midway, scale=.7]{$\vartheta_{n0}$};

    \draw[->] (41)--(51) ;
    \draw[->] (45)--(55) ;

    \draw[->] (51)--(61) node[left=1pt, midway, scale=.7]{$\LL f^*\big(\gamma[n]\big)$};
    \draw[->] (55)--(65) node[right=1pt, midway, scale=.7]{$(\LL f^*\gamma)[n]$};

    \draw[-, double distance=2pt] (61)--(71);
    \draw[-, double distance=2pt] (65)--(75);

 %slanted
   \draw[->] (13)--(21)  ;
   \draw[->] (13)--(25)  ;

%labels
  \node at (3,-1.55){\circled1};
  \node at (3,-4){\circled2};

 \epic
\end{equation}

\end{figure}

In the subdiagrams \circled1 and \circled2 of~\eqref{diag(i)} one can replace $C$ by a q-flat resolution~$P_C$ that belongs to a family of q-flat resolutions that commute with~translation (see \cite[2.5.5]{li},
and thereby reduce the question of commutativity to the analogous one in which all derived functors are replaced by ordinary functors of complexes. The latter question is easily disposed of.

Commutativity of the other subdiagrams is straightforward to verify.\va1

(ii) As in (i), given $\alpha\colon A\to B[i\>]$ (in $\bE_\sX^i(A,B)$) and 
$h\colon\OY\to\OY[n]$, one wants commutativity of the border of the next diagram~\eqref{diag(ii)}, in which $p^{}_2(F,G)$~is the bifunctorial map adjoint to the natural composition in $\D(X)$
\[
f^*\<(F\Otimes Y \fst G)\to f^*\<\<F\Otimes\sX f^*\!\fst G\to f^*\<\<F\Otimes\sX G
\qquad(F,G\in\D(Y));
\]
and where unlabeled maps are the natural ones (see \cite[3.2.4(ii)]{li}).

\begin{figure}[ht]

\begin{equation}\label{diag(ii)}
\def\1{$f^{}_{\<\<*}\<A$}
\def\2{$f^{}_{\<\<*}(\OX\<\<\Otimes {\<X} A)$}
\def\3{$f^{}_{\<\<*}\OX\<\<\Otimes Y\<\< f^{}_{\<\<*}\<A$}
\def\4{$\OY\<\<\Otimes Y\<\< f^{}_{\<\<*}\<A$}
\def\5{$f^{}_{\<\<*}(\OX[n]\<\<\Otimes {\<X} A)$}
\def\6{$f^{}_{\<\<*}\big((f^*\OY)[n])\Otimes\sX A\big)$}
\def\7{$f^{}_{\<\<*}(f^*(\OY[n])\Otimes\sX A)$}
\def\8{$\OY[n]\<\<\Otimes {Y}\<\<f^{}_{\<\<*}A$}
\def\9{$f^{}_{\<\<*}\<\<\big((\OX\<\<\Otimes {\<X} A)[n]\big)$}
\def\ten{$\mkern-5mu(\OY\<\<\Otimes {Y}\<\<f^{}_{\<\<*}\<A)[n]$}
\def\lvn{$f^{}_{\<\<*}\<\<\big(A[n]\big)$}
\def\twv{$(f^{}_{\<\<*}\<A)[n]$}
\def\thn{$f^{}_{\<\<*}\<\<\big(B[i\>][n]\big)$}
\def\frn{$\big(f^{}_{\<\<*}\<B[i\>]\big)[n]$}
\def\ffn{$f^{}_{\<\<*}\<\<\big(B[i+n]\big)$}
\def\sxn{$(f^{}_{\<\<*}B)[i+n]$}
\def\stn{$(f^{}_{\<\<*}B)[i\>][n]$}
\def\etn{$f^{}_{\<\<*}f^*\OY\<\<\Otimes Y\<\< f^{}_{\<\<*}A$}
\def\ntn{$f^{}_{\<\<*}(f^*\OY\<\<\Otimes \sX\<\< A)$}
\!
 \bpic[xscale=2.5,yscale=1.3]
 
  \node(13) at (3,0)[scale=.9] {\1};

  \node(21) at (1,-1) [scale=.9] {\2};
  \node(22) at (2.33,-1) [scale=.9] {\3};
  \node(24) at (3.75,-1) [scale=.9] {\etn};
  \node(25) at (5,-1) [scale=.9] {\4};
  
  \node(23) at (3.75,-2) [scale=.9] {\ntn};

  \node(31) at (1,-3) [scale=.88] {\5};
  \node(32) at (2.35,-3.45)[scale=.88]  {\6};
  \node(34) at (3.75,-4) [scale=.88] {\7};
  \node(35) at (5,-3) [scale=.88] {\8};

  \node(41) at (1,-4) [scale=.9] {\9};
  \node(45) at (5,-4)[scale=.9]  {\ten};
 
  \node(51) at (1,-5) [scale=.9] {\lvn};
  \node(55) at (5,-5) [scale=.9] {\twv};

  \node(61) at (1,-6) [scale=.9] {\thn};
  \node(65) at (5,-6) [scale=.9]  {\frn};

  \node(71) at (1,-7) [scale=.9] {\ffn};
  \node(73) at (3,-7) [scale=.9] {\sxn};
  \node(75) at (5,-7) [scale=.9] {\stn};
 
  \node at (1,-7.3) {$$};
 %horizontal
  \draw[->] (22)--(21);
  \draw[->] (24)--(22);
  \draw[->] (25)--(24) node[below=-.75pt, midway, scale=.7]{\kern2.75pt(\kern-.6pt\ref{baretaeps}\kern-1pt)};  
  \draw[->] (1.7,-3.32)--(1.45,-3.21) ;
  \draw[-, double distance=2pt] (3.07,-3.9)--(2.7, -3.75);
  \draw[->] (35)--(34) node[above=-7.5pt, midway, scale=.65]{\kern-6pt\rotatebox{23}{$\!\!p^{}_2(\OY[n],\<A)$}};

  \draw[-, double distance=2pt] (51)--(55);

  \draw[-, double distance=2pt] (61)--(65);

  \draw[-, double distance=2pt] (71)--(73);
  \draw[-, double distance=2pt] (73)--(75);
  
 %vertical
    \draw[->] (21)--(31) node[left=1pt, midway, scale=.7]{$\textup{via}\;\mu_{\<\<f}(h)$};
    \draw[->] (24)--(23) ;
    \draw[->] (23)--(34) node[left=1pt, midway, scale=.7]{$\textup{via}\;h$};
    \draw[->] (25)--(35) node[right=1pt, midway, scale=.7]{$\textup{via}\;h$};

    \draw[-, double distance=2pt] (31)--(41) node[left=1pt, midway, scale=.7]{$\Rf \vartheta_{n0}$};
    \draw[-, double distance=2pt] (35)--(45) node[right=1pt, midway, scale=.7]{$\vartheta_{n0}$};

    \draw[->] (41)--(51) ;
    \draw[->] (45)--(55) ;

    \draw[->] (51)--(61) node[left=1pt, midway, scale=.7]{$\Rf \big(\alpha[n]\big)$};
    \draw[->] (55)--(65) node[right=1pt, midway, scale=.7]{$(\Rf \alpha)[n]$};

    \draw[-, double distance=2pt] (61)--(71);
    \draw[-, double distance=2pt] (65)--(75);

 %slanted
   \draw[->] (13)--(21)  ;
   \draw[->] (13)--(25)  ;
  
   \draw[->] (3.2, -1.9)--(1.44, -1.22);
   %(23)--(21)  ;
   \draw[->] (25)--(23) node[below=-7.5pt, midway, scale=.65]{\kern10pt\rotatebox{23}{$\quad p^{}_2(\OY\<,\<A)$}}; 

%labels
  \node at (3,-.55){\circled3};
  \node at (3,-4.5){\circled4};
%  \node at (3,-5.5){\circled3};
  
 \epic
\end{equation}

\end{figure}

Commutativity of the unlabeled subdiagrams of \eqref{diag(ii)} is easily checked.

Commutativity of subdiagram \circled3  is shown in \cite[p.\,104]{li}.\va1

As for \circled4,   it suffices to prove commutativity of the adjoint diagram, namely
the border of the natural $\D(X)$-diagram~\eqref{diag(iii)} below.

\begin{figure}[ht]

\begin{equation}\label{diag(iii)}
\!\<
  \bpic[xscale=4.6, yscale=1.4]

  \node at (0,-.6){$$};
  
   \draw (1,-6) node (12){$f^*\!\fst\big(A[n]\big)$};
 
   \draw (2,-5) node (21){$A[n]$};
   \draw (1,-5) node (22){$(f^*\!\fst A)[n]$};
   \draw (0,-5) node (23){$f^*\big((\fst A)[n]\big)$};

   \draw (2,-4) node (31){$(\OX\Otimes{\<\<X}A)[n]$};
   \draw (1,-4) node (32){$(\OX\Otimes{\<\<X}f^*\!\fst A)[n]$};

   \draw (2,-3) node (41){$\OX[n]\Otimes{\<\<X}A$};
   \draw (1,-3) node (42){$\OX[n]\Otimes{\<\<X}f^*\!\fst A$};
   \draw (0,-3) node (43){$f^*\<\<\big((\OY\<\Otimes{Y}\<\fst A)[n]\big)$};
   
   \draw (2,-2) node (51){$(f^*\OY\<\<)\<[n]\Otimes{\<\<X}A$};
   \draw (1,-2) node (52){$(f^*\OY\<\<)\<[n]\Otimes{\<\<X}f^*\!\fst A$};

   \draw (2,-1) node (61){$f^*\<\big(\OY[n]\big)\Otimes{\<\<X}A$};
   \draw (1,-1) node (62){$f^*\<\big(\OY[n]\big)\Otimes{\<\<X}f^*\!\fst A$};
   \draw (0,-1) node (63){$f^*\<\<\big(\<\OY[n]\<\Otimes{Y}\<\fst A\<\big)$};
 
 %horizontal
   \draw[->] (22)--(21) ;
   \draw[-, double distance=2pt] (22)--(23) ;

   \draw[->] (32)--(31) ;
   
   \draw[->] (42)--(41) ;
      
   \draw[->] (52)--(51) ;

   \draw[->] (62)--(61) ;
   \draw[->] (63)--(62) ;

   %vertical
 
  \draw[->] (31)--(21);
  \draw[->] (32)--(22);
  
  \draw[-, double distance=2pt] (41)--(31) node[left=1pt, midway, scale=.7]{$\vartheta_{n0}$};
  \draw[-, double distance=2pt] (42)--(32) node[right=1pt, midway, scale=.7]{$\vartheta_{n0}$};
  \draw[->] (43)--(23);
  \draw[-, double distance=2pt] (43)--(63) node[left=1pt, midway, scale=.7]{$\LL f^*\vartheta_{n0}$};

  \draw[-, double distance=2pt] (51)--(61);
  \draw[-, double distance=2pt] (52)--(62);  
  \draw[->] (51)--(41) ;
  \draw[->] (52)--(42) ;

 %slanted

   \draw[->] (12)--(21) ;
   
   \draw[-, double distance=2pt] (23)--(12) ;

 %labels
  \node(1) at (1,-5.52){\circled6} ;
  \node(2) at (.5,-3.05) {\quad\circled5};
     
 \epic
\end{equation}

\end{figure}

Diagram \circled5 is the commutative   diagram \circled2 in~\eqref{diag(i)}, with
$C=\fst A$. 

Diagram \circled6 is ``dual" to diagram \circled 2 in \eqref{compid}, so  its commutativity can be proved
as indicated in the last line of the proof of Proposition~\ref{adjunction}.

Commutativity of the remaining subdiagrams is straightforward to verify.

Thus \circled4 commutes, and (ii) results.\va{1.5}

(iii) This follows from (i) (respectively (ii)) and  functoriality of $f^*\<$.
\end{proof}

\penalty-2000
\begin{cosa}\label{shriek} 
Recall examples (a) and (b) in \S\ref{examples}. These examples 
support a \emph{twisted inverse\kf-image pseudofunctor} $(-)_{\<\upl}^!\>$, as follows.

A scheme\kf-map $f\colon X\to Y$ is \emph{essentially smooth} (resp.~\emph{essentially \'etale}) if it is essentially of finite presentation (\S\ref{finite pres}) and formally smooth (resp. \emph{formally \'etale}), i.e.,  for each $x\in X$,  the local ring~$\CO_{\<\<X\<,\>x}$ is formally smooth (resp.~\emph{formally \'etale}) over $\CO_{Y\<,\>fx}$ for the discrete topologies, see \cite[p.\,115, 19.10.2]{EGA0} and cf.~\cite[\S17.1 and Thm.\,17.6.1]{EGA4}.  From \cite[Theorems (17.5.1) and (17.6.1)]{EGA4} it follows that any  essentially smooth or essentially \'etale map is flat.

For a ringed space $X\<$, let $\Dqcpl(X)\subset\Dqc(X)$ be the full subcategory with objects those
complexes $G\in\Dqc(X)$ such that $H^n(G)=0$ for all $n\ll0$.

In case  (a), \cite[5.3]{Nk} gives a contravariant $\Dqcpl$-valued pseudofunctor~$(-)_{\<\upl}^!$ 
over~$\SS$, uniquely determined up to isomorphism by the properties:\va1

(i) When restricted to proper maps, $(-)_{\<\upl}^!$  is pseudo\-functorially right-adjoint to the
right-derived direct-image\va1 pseudofunctor~$\Rf\>$.\looseness=-1

Thus for proper $f\colon X\to Y$, $f^!_\upl$ is defined on all of  $\Dqc(Y)$, and there is a counit map 
\begin{equation}\label{couni+}
\bar{\couni{\<\<\!f}}\colon \Rf f_{\<\<\upl}^!\>\to \id_{\Dqc(Y)}
\end{equation}
 such that \eqref{transitivity}, \emph{mutatis mutandis}, commutes (cf.~\cite[proof of 4.1.2]{li}); and to any independent $\SS$-square~$\Dd$
as in~Prop\-osition~\ref{theta2}, there is associated the functorial isomorphism $\theta_{\<\Dd}\colon u^*\!f_*\iso g_*v^*\<$, whose restriction $\LL u^*\Rf\iso\R g_*\LL v^*$ to derived-category functors we denote by $\bar\theta_\Dd$. 

\penalty-2000
There results the \emph{base-change map}
\begin{equation}\label{base-change}
\!\!\!\bar{\,\,\,\bchadmirado\Dd}\colon v^*\!f_{\<\<\upl}^!\>\to g_{\upl}^!u^*
\end{equation}
that is adjoint to the natural composition 
$$ 
\R g_*v^*\!f_{\<\<\upl}^!\> \underset{\bar\theta_{\<\<\Dd}^{-\<\<1\>}}\iso u^*\Rf f_{\<\<\upl}^!
\xto[\:\bar{\couni{\<\<\!f}}\:]{}u^*.
$$

(ii) When restricted to essentially \'etale maps, $(-)_{\<\upl}^!$   is equal to the usual
inverse-image pseudofunctor (derived or not).\va1

(iii) For each independent $\SS$-square $\Dd$ as in \ref{theta2},
with $f$ (hence $g$) proper and $u$ (hence $v$) essentially \'etale,  $\!\!\!\bar{\,\,\,\bchadmirado\Dd}$
is  the natural composite isomorphism
$$
v^*\!f_{\<\<\upl}^!\>=v_{\<\upl}^!f_{\<\<\upl}^!\>\iso(fv)_{\<\upl}^!=(ug)_{\<\upl}^!\iso g_{\upl}^!u_{\<\upl}^!=g_{\upl}^!u^*.
$$

There is a similarly-characterized pseudofunctor $(-)_{\<\upl}^!$  in case (b)---argue as in~ \cite[Theorem~ 7.3.2]{Nk0}, using \cite[4.7.4 and 4.8.2.3]{li}.\va1

The purpose of this subsection is to extend $(-)^!_\upl$ to an  $H_Y$-graded pseudo\-functor $(-)^!$  taking values in the categories $\D_W$.\va2

For any map $f\colon X\to Y$ in $\SS$, denote the ``relative dualizing complex" $f_{\<\<\upl}^!\OY$ by $\cd f$. Recalling from \S\ref{D_W} that we write  $f^*\<C$ for~$\Lf\< C$, and with $\totimes$ as in \eqref{totimes}, set
\begin{equation}\label{def!}
f^!C:= \cd f\totimes f^*\< C\qquad(C\in\D_Y).
\end{equation}
\vskip1pt
 It follows from Propositions~\ref{tensor graded} and~\ref{* and T}(i) that $f^!(-)$ 
 \emph{is an $H_Y$-graded functor from $\D_Y$ to\/ $\D_\sX$.}\va2

Next, for any $X\xto{f}Y\xto{g}Z$  in $\SS$, we need a degree\kf-0 functorial isomorphism $\ps^!\colon f^!\<g^!\iso (gf)^!$.\va1

The functor $g^!_{\<\upl}$ is bounded above, so $\cd g=g^!_{\<\upl}\OZ\in\Dqcpl(Y)$, see \cite[4.9.4(iv)]{li} 
in case~(a), or  \cite[top of p.\,191]{li} in case~(b).  
By \cite[5.8]{Nk} (in case~(a)), or by \cite[4.7.2]{li} (in case~(b)), there is a canonical functorial isomorphism 
\begin{equation}\label{chi}
\chi^f_C\colon \cd f\Otimes\sX f^*\<C\iso  f_{\<\upl}^!C\qquad (C\in\Dqcpl(Y)).
\end{equation}
There is, in particular, an isomorphism
$$
\chi^f_{\cd g}\colon\cd f\Otimes\sX f^*\cd g\iso\cd{gf}.
$$

We can now define a degree\kf-0 functorial isomorphism
\begin{equation}\label{ps!}
\ps^!\colon f^!\<g^!E\iso (gf)^!E\qquad (E\in\D_Z)
\end{equation} 
to be the natural functorial composite
$$
\cd f\Otimes\sX f^*(\cd g\Otimes Y g^*\<\<E)\iso
(\cd f\Otimes\sX f^*\cd g)\Otimes \sX f^*\<\<g^*\<\<E\iso
\cd{gf}\Otimes\sX (gf)^*\<E.
$$
By the proof of \cite[4.9.5]{li}, when $E\in\Dqcpl(Z)$, this $\ps^!$ can be identified via $\chi^f_{g^!\<E}\>$, 
$\chi^g_E$ and $\chi^{gf}_E$ with the  isomorphism given by 
$\ps_\upl^!\colon f_\upl^!g_\upl^!\iso(gf)_\upl^!\>$. 

\pagebreak[3]
Furthermore, for any $X\xto{f\>}Y\xto{g\>}Z\xto{h\>}W$ in $\SS$, the following natural diagram \emph{commutes,}
\begin{equation}\label{hidden ps^!}
\mkern-25mu
\CD
 \bpic[xscale=3, yscale=1.4]

  \node (11) at (1,-1) {$\cd f\<\Otimes\sX\< f^*\cd g\Otimes\sX f^*\<g^*\cd {\>\>h} $};
  \node (12) at (1,0) {$\cd f\<\Otimes\sX\< f^*(\cd g\Otimes Y g^*\cd {\>\>h})$};
  \node (13) at (3,0) {$\cd f\<\Otimes\sX \<f^*\cd{\>\>hg}$};

  \node (21) at (1,-2) {$ \cd{gf}\Otimes\sX (gf)^*\cd{\>\>h}$};
  \node (23) at (3,-2) {$ \cd{\>\>hgf}$};

  \draw[->] (12)--(11);
  \draw[->] (12)--(13) node[midway, above, scale=.7]{$\!\id\Otimes\sX \>\chi^g_{\cd {\>\>h}\,} $};

  \draw[->] (21)--(23) node[midway, below=1pt, scale=.7] {$\chi^{gf}_{\<\cd {\>\>h}}$};

  \draw[->] (11)--(21) node[midway, left=1pt, scale=.7] {$\chi^f_{\<\cd{\>g}} \!\Otimes\sX\ps^*$};
  \draw[->] (13)--(23) node[midway, right=1pt, scale=.7] {$\chi^f _{\<\cd {\>\>hg}}$};

 \epic
\endCD
\end{equation}
since it is isomorphic to the  natural diagram
\[ \mkern20mu
 \bpic[xscale=5, yscale=1.4]
 
   \node(13) at (3.05,-1){$f_\upl^!\<g_\upl^!h_\upl^!\OW$};
   \node(14) at (4,-1){$f_\upl^!(hg)_\upl^!\OW$};

   \node(23) at (3.05,-2){$(gf)_\upl^!h_\upl^!\OW$};
   \node(24) at (4,-2){$(hgf)_\upl^!\OW$};

     \draw[-, double distance=2pt](13)--(14) node[midway, above=1pt, scale=.75]{$f^!_\upl\ps^!_\upl$};
      \draw[-, double distance=2pt](13)--(23) node[midway, left=1pt, scale=.75]{$\ps^!_\upl$};
        \draw[-, double distance=2pt](14)--(24) node[midway, right=1pt, scale=.75]{$\ps^!_\upl$};
          \draw[-, double distance=2pt](23)--(24) node[midway, below=1pt, scale=.75]{$\ps^!_\upl$};

 \epic
\]
which commutes because $(-)^!_\upl$ and $\ps^!_\upl$ form a pseudofunctor.\va1

To show that $(-)!$ and $\ps^!$ form a pseudofunctor,  use \eqref{hidden ps^!} to verify that the following expansion \eqref{expansion} of the second diagram in \eqref{assoc} commutes. 
\begin{figure}
\rotatebox{-90}
{\begin{minipage}{8.5in}
\vskip-7pt
\begin{equation}\label{expansion}
\mkern-20mu
\CD
 \bpic[xscale=4,yscale=2]
  \node(11) at (1.25,-1) [scale=.87] {$\cd f\<\Otimes\sX\<\< f^*\<\<\big(\cd g\<\Otimes Y\<\< g^*(\cd{\>\>h}\<\Otimes{\<\<Z} \!h^*\<)\big)$};
  \node(12) at (2.5,-1) [scale=.87] {$\cd f\<\Otimes\sX\<\< f^*\<\<\big(\cd g\<\Otimes Y\<\< g^*\cd{\>\>h}\<\Otimes Y\<\< g^*h^*\big)$};
  \node(13) at (3.81,-1) [scale=.87] {$\cd f\<\Otimes\sX\<\< f^*\<\big((\cd g\<\Otimes Y\<\< g^*\cd{\>\>h})\<\Otimes Y \!(hg)^*\big)$};
  \node(14) at (5,-1) [scale=.87] {$\cd f\<\Otimes\sX\<\< f^*\<\<\big(\cd{\>\>hg}\<\Otimes{Y} \<\<(hg)^*\big)$};
  \node(21) at (1.25,-2) [scale=.87] {$\cd f\<\Otimes\sX\<\< f^*\cd g\<\Otimes\sX\<\< f^*\<\<g^*(\cd{\>\>h}\<\Otimes{\<\<Z} \!h^*\<)$};
  \node(22) at (2.53,-2) [scale=.87] {$\cd f\<\<\Otimes\sX\!f^*\cd g\<\Otimes\sX\<\< f^*\<(g^*\cd{\>\>h}\<\Otimes Y \<\<g^*\<h^*\<)$};
  \node(23) at (3.79,-2) [scale=.87] {$\ \:\cd f\<\<\Otimes\sX\!f^*\<(\cd g\<\Otimes Y\<\< g^*\cd{\>\>h})\<\Otimes\sX\<\< f^*\<\<g^*\<h^*$};
  \node(32) at (2.53,-3) [scale=.87] {$\cd f\<\Otimes\sX\<\< f^*\cd g\<\Otimes\sX\<\< f^*\<\<g^*\cd{\>\>h}\<\Otimes\sX \!f^*\<\<g^*\<h^*$};
  \node(33) at (3.79,-3) [scale=.87] {$\cd f\<\Otimes\sX\<\< f^*\cd {\>\>hg}\<\Otimes\sX \<\<f^*\<\<g^*\<h^*$};
  \node(34) at (5,-3) [scale=.87] {$\cd f\<\Otimes\sX\<\< f^*\cd{\>\>hg}\<\Otimes\sX \<\<f^*\<(hg)^*$};
  \node(41) at (1.25,-4) [scale=.87] {$\cd{gf}\<\Otimes\sX\<\< (gf)^*(\cd{\>\>h}\<\Otimes{\<\<Z} \<\<h^*\<)$};
  \node(42) at (2.53,-4) [scale=.87] {$\cd{gf}\<\Otimes\sX\<\< (gf)^*\cd{\>\>h}\<\Otimes\sX \<\<(gf)^*h^*$};
  \node(43) at (3.79,-4) [scale=.87] {$\cd {\>\>hgf}\<\Otimes\sX \!(gf)^*\<h^*$};
  \node(44) at (5,-4) [scale=.87] {$\cd{\>\>hgf}\<\Otimes\sX\<\<(hgf)^*$}; 
 %horizontal
    \draw[->] (11)--(12);
    \draw[->] (12)--(13);
    \draw[->] (13)--(14);
    \draw[->] (21)--(22);  
    \draw[->] (33)--(34);
    \draw[->] (41)--(42);
  %vertical
    \draw[->] (11)--(21);
    \draw[->] (2.53, -1.21)--(22);
    \draw[->] (14)--(34);
    \draw[->] (21)--(41);
    \draw[->] (22)--(32);
    \draw[->] (23)--(33);
    \draw[->] (32)--(42);
    \draw[->] (33)--(43);
    \draw[->] (34)--(44);
 %slanted
    \draw[->] (12)--(23);
    \draw[->] (23)--(32);
    \draw[->] (42)--(43); 
    \draw[->] (43)--(44);
  %labels 
    \node at (2.8,-1.6) [scale=.83] {\circled1};
    \node at (1.66,-3.05) [scale=.83] {\circled2};
    \node at (3.16,-3.5) [scale=.83] {cf.~(\ref{hidden ps^!})};
\epic
\endCD
\end{equation}
\end{minipage}
\hspace{-14.5mm}
}
\end{figure}

To see that subdiagram \circled 1 commutes when applied to, say, $E\in\D(W)$, replace $\cd g$, $g^*\cd{\>\>h}$ and $g^*h^*\<\<E$ by q-flat resolutions to reduce to the analogous question for ordinary complexes and nonderived tensor products, which is now easily settled.

Similarly, for commutativity of \circled 2 replace\- $\cd{\>\>h}$ and $h^*E$ by q-flat resolutions, and argue as in the middle of \cite[p.\,124]{li}. 

Checking commutativity of the remaining subdiagrams is straightforward.
\end{cosa}

\begin{cosa}\label{bchadmirado}
Consider, in $\SS$, an independent  square  
\begin{equation}\label{orientedcommsq}
\CD
    \begin{tikzpicture}[yscale=.9]
      \draw[white] (0cm,0.5cm) -- +(0: .75\linewidth)
      node (E) [black, pos = 0.32] {$Y^\prime$}
      node (F) [black, pos = 0.56] {$Y$};
      \draw[white] (0cm,2.65cm) -- +(0: .75\linewidth)
      node (G) [black, pos = 0.32] {$X^\prime$}
      node (H) [black, pos = 0.56] {$X$};
      \draw [->] (G) -- (H) node[above=1pt, midway, sloped, scale=0.75]{$v$};
      \draw [->] (E) -- (F) node[below=1pt, midway, sloped, scale=0.75]{$u$};
      \draw [->] (G) -- (E) node[left=1pt,  midway, scale=0.75]{$g$};
      \draw [->] (H) -- (F) node[right=1pt=1pt, midway, scale=0.75]{$f$};
      \node (C) at (intersection of G--F and E--H) [scale=0.75] {$\Dd$};
    \end{tikzpicture}
\endCD
  \end{equation}
By Proposition~\ref{theta2}, the associated map $\theta_\Dd\colon u^*\!f_*\to g_*v^*$ is an isomorphism.\va1

\begin{subcosa}
With notation as in \eqref{couni+},
the  functorial \emph{flat base-change isomorphism}
$$
\bar{\>\mathsf B}_{\mkern-.5mu\Dd}(G)\colon v^*\!f^!_\upl G\to g^!_\upl u^*\<G\qquad(G\in\Dqcpl (Y))
$$
is defined in case (a) of \S\ref{examples} as follows.

\pagebreak[3]
 If $f$ (hence $g$) is proper, then  $\<\bar{\>\mathsf B}_{\mkern-.5mu\Dd}$ is, as in \eqref{base-change}, the $\D(X')$-map  adjoint to the composite map
      \[
        g_*v^*\!f^!_\upl \xto{\>\bar\theta_{\<\Dd}^{-\<1\<}\<} u^*\!f_*f^!_\upl \xto{\!u^*\!\bar{\counitimes{\!f}}}  u^*\<.
      \]
That in this case $\<\bar{\>\mathsf B}_{\mkern-.5mu\Dd}(G)$ \emph{is an isomorphism for all}\va1 $G\in \Dqcpl (Y)$
is a basic fact of Grothendieck duality theory \cite[Corollary 4.4.3]{li}, \cite[Theorem 5.3]{Nk}.

\pagebreak[3]

When $f$ is not necessarily proper, there exists 
a factorization  \smash{$f=\bar f\liftcirc \<{\underset{\lift1.7,-,}{f}}$}
with $\bar f$ proper and \smash{${\underset{\lift1.7,-,}{f}}$} a \emph{localizing immersion} \cite[Theorem 4.1]{Nk}. 
Localizing immersions are set-theoretically injective maps that on
sufficiently small affine sets correspond to localization of rings.\va{.4} They are flat monomorphisms,
and if of finite type, open immersions, 
see \cite[2.7, 2.8.8, 2.8.7, 2.8.3]{Nk}.\va{.8} 
 They are essentially \'etale, so \smash{${\underset{\lift1.7,-,}{f}}^!{}^{}_{\!\!\!\upl}={\underset{\lift1.7,-,}{f}}^*\<$.}
Localizing immersions\va{1.3} remain so after base change \cite[2.8.1]{Nk}. Hence
 $\Dd$ decomposes into two fiber squares\looseness=-1
\[ 
 \begin{tikzpicture}[scale=.9]
      \draw[white] (0cm,0.5cm) -- +(0: \linewidth)
      node (E) [black, pos = 0.4] {$Y'$}
      node (F) [black, pos = 0.59] {$Y$};
      \draw[white] (0cm,2.65cm) -- +(0: \linewidth)
      node (G) [black, pos = 0.4] {$\bar X'$}
      node (H) [black, pos = 0.59] {$\bar X$};
      \draw[white] (0cm,4.8cm) -- +(0: \linewidth)
      node (I) [black, pos = 0.4] {$X'$}
      node (J) [black, pos = 0.59] {$X$};
      \draw [->] (G) -- (H) node[above, midway, sloped, scale=0.75]{$h$};
      \draw [->] (E) -- (F) node[below=1pt, midway, sloped, scale=0.75]{$u$};
      \draw [->] (G) -- (E) node[left, midway, scale=0.75]{$\bar g\>$};
      \draw [->] (H) -- (F) node[right, midway, scale=0.75]{$\>\bar f$};
      \draw [->] (I) -- (J) node[above=1pt, midway, sloped, scale=0.75]{$v$};
      \draw [->] (I) -- (G) node[left, midway, scale=0.75]{${\underset{\lift 1.7,-,}{g}}\>$};
      \draw [->] (J) -- (H) node[right, midway, scale=0.75]{$\>{\underset{\lift1.6,-,}{f}}$};
      \node (K) at (intersection of I--H and G--J)   [scale=0.85] {\lift2,\underbar{$\Dd$},};
      \node (L) at (intersection of G--F and H--E) [scale=1.15] {$\lift1.15,\overline{\Dd},$};
 \end{tikzpicture}
 \]
where \smash{${\underset{\lift 1.7,-,}{g}}$ is a localizing immersion, so that 
${\underset{\lift 1.7,-,}{g}}^!={\underset{\lift 1.7,-,}{g}}^*$.}\va{3}

Let  \smash{$\<\bar{\>\mathsf B}\<(\overline\Dd,\underbar{$\Dd$})$} be the composite isomorphism,
in $\D(X')$,
$$
v^*\!f^!_\upl \iso v^*\!{\underset{\lift1.7,-,}{f}}^!{}^{}_{\!\!\!\upl}\>\>\>\bar{\!f}{}^!_{\!\upl} = v^*\<\<{\underset{\lift 1.7,-,}{f}}^*\>\bar{\!f}{}^!_{\!\upl}
\overset{\ps^*}{=\!=}{\underset{\lift 1.7,-,}{g}}^{\<*}\<h^{\<*}\>\bar{\!f}{}^!_{\!\upl}\underset{\lift1.2,
\!\bar{\>\>\lift.75,\mathsf B,}^{}_{\mkern-.5mu\sss\bar\Dd},}{\iso} {\underset{\lift 1.7,-,}{g}}^{\<*}\bar g^{\>!}_\upl u^*
= {\underset{\lift 1.7,-,}{g}}^!{}^{}_{\!\!\!\upl}\>\bar g^{\>!}_\upl u^*\overset{\ps^{\>!}_{\lift.15,\halfsize\upl,}}{=\!=}
g^{\>!}_\upl u^*.
$$

Arguing as in the proof of \cite[Theorem 4.8.3]{li}, one shows that $\<\bar{\>\mathsf B}\<(\overline\Dd,\underbar{$\Dd$})$ depends only on $\Dd$, and not on its decomposition. We may therefore denote this
functorial isomorphism simply by  $\<\bar{\>\mathsf B}_{\mkern-.5mu\Dd}$.  (See also
\cite[5.2, 5.3]{Nk}.)\va1

In particular, we have the $\D(X')$-isomorphism
\begin{equation}\label{Bcd}
  \<\bar{\>\mathsf B}_{\mkern-.5mu\Dd}(\OY)\colon v^*\cd f= v^*\!f^!_\upl\OY\iso g^{\>!}_\upl u^*\OY=\cd g.
\end{equation}

\smallskip  
Case (b) of \S\ref{examples} can be treated analogously, see \cite[Theorem 7.3.2(2)]{Nk0}.

\end{subcosa}

\begin{subcosa}
Now, referring to \eqref{orientedcommsq}, we define the $\D_{\sX'}\mkern.5mu$-isomorphism
$$
\bchadmirado{\Dd}(G)\colon v^*\!f^!G\to g^! u^*\<G\qquad(G\in\D_Y)
$$
to be the natural composition
$$
v^*(\cd f\totimes \<f^*\<G)\iso v^*\cd f\totimes v^*\!f^*\<G \underset{\eqref{Bcd}}{\iso} 
\cd g\totimes v^*\!f^*\<G\overset{\id\!\totimes \!\ps^*}{=\!=\!=\!=}\cd g\totimes g^*u^*\<G .
$$

It results from \cite[Exercise 4.9.3(c)]{li}  that if $G\in\Dqcpl(Y)$ then 
\begin{equation}\label{B=barB}
\bchadmirado{\Dd}(G)= \<\bar{\>\mathsf B}_{\mkern-.5mu\Dd}(G).
\end{equation}

It is left to the reader to verify that $\bchadmirado{\Dd}$ is a degree\kf-0 functorial map.

It is also left to the reader to use the definition of $\>\bchadmirado{\Dd}$ to expand the \emph{vertical and horizontal transitivity} diagrams 
\eqref{basechange1} and \eqref{basechange2} and to verify that the expanded diagrams commute, using e.g., transitivity for $\<\bar{\>\mathsf B}_{\mkern-.5mu\Dd}$ (see \cite[p.\,205, (3)]{li} and \cite[p.\,208, Theorem~4.8.3]{li}---whose proof, in view of Nayak's compactification theorem \cite[Theorem 4.1]{Nk}, extends to \emph{essentially} finite\kf-type maps), transitivity for $\theta_\Dd$ (cf.~\cite[Prop.\, 3.7.2, (ii) and (iii)]{li}),  and the ``dual" \cite[pp.\,106--107]{li} of the last diagram in \cite[3.4.2.2]{li}, as treated  in the first paragraph of \cite[p.\,104]{li}.

\end{subcosa}
\end{cosa}

\begin{cosa}\label{integral}
Let $f\colon X\to Y$ be a \emph{confined} $\SS$\kf-map (see \S\ref{examples}). 
We now define a degree\kf-0 functorial map
$\couni{\<\<\!f}\colon f^{}_{\<\<*}f^!\to \id$ that satisfies transitivity (see \S\ref{f_*}).\va1

The \emph{projection map} $p(F,G)\ (F\in\Dqc(X),G\in\Dqc(Y))$ is the natural composition,
in $\Dqc(Y)$,
\begin{equation}\label{projection}
\fst F\Otimes Y G \to \fst f^*\<(\fst F\Otimes Y G)\to\fst (f^*\<\<\fst F\Otimes \sX f^*G)\to
\fst(F\Otimes \sX f^*G).
\end{equation}
 
 This $p(F,G)$ is an \emph{isomorphism} \cite[3.9.4]{li}. Denote its inverse by $\tilde p\>(F,G)$.\va1

 From \eqref{couni+} we have a $\Dqc(Y)$-map $\fst \cd f\to \OY$. 
Using this map, let $\bar{\couni{\<\<\!f}}(G)$ be the natural functorial composition
\begin{equation*}\label{defint}
f_*(\cd f\Otimes \sX f^*G)\xto{\tilde p(\cd f\<,\>\>G)} \fst\cd f \Otimes Y G \lto \OY\Otimes Y G\iso G.
\end{equation*}

\begin{sublem} \label{extend int}
This\/ $\bar{\couni{\<\<\!f}}$ extends to a degree-$0$ map $\couni{\<\<\!f}$ of graded endofunctors of\/~$\D_Y$.
\end{sublem}

\begin{proof} Set $D\set\cd f$, and write $\otimes$ for $\Otimes\sX$ or $\Otimes Y$, as the case may be. Unwinding definitions,  interpret  the assertion as being  that for any 
$\Dqc(Y)$-map $\alpha\colon A\to B[i\>]\ (i\in\ZZ)$,
the border of the following natural diagram commutes:\looseness=-1

\smallskip
\begin{small}
\[
\def\1{\mkern-3mu\fst (D\<\otimes\<\< f^*\!A)}
\def\2{A}
\def\3{\mkern-17mu\fst D\<\otimes\<\< A}
\def\4{B[i\>]}
\def\5{\fst\<\big(D\<\<\otimes\! f^*\<(B[i\>])\big)}
\def\6{\fst\<\big(D\<\<\otimes\!(f^*\<\<B)[i\>]\big)}
\def\7{\makebox[0pt]{\hss$\mkern18mu\fst\<\big ((D\<\otimes\<\<f^*\<\<B)[i\>]\big)$\hss}}
\def\8{\mkern-14mu\OY\!\otimes\<\< A}
\def\9{\makebox[0pt]{\hss$\big(\fst\<(D\<\otimes\<\<f^*\<\<B)\<\big)[i\>]$\hss}}
\def\ten{(\fst D\<\otimes\<\< B)[i\>]}
\def\lvn{\mkern8mu(\<\<\OY\!\otimes\<\< B)\<[i\>]}
\def\twv{\fst D\<\otimes\<\< B[i\>]}
\def\thn{\mkern6mu\OY\!\otimes\<\< B[i\>]}
\def\frn{A}
\minCDarrowwidth=.25in
\CD
\1 @>\tilde p(D\<,\>A)>> \3 @>{\!\bar{\couni{\!\!f}}(\OY)\otimes\>\id\,}>> \8 @>>> \2
         \\[-2pt]
@V\textup{via}\;\alpha VV @V\textup{via}\;\alpha VV @V\textup{via}\;\alpha VV 
    @V\textup{via}\;\alpha VV
         \\[-5pt]   
  \5 @>\tilde p(D\<,\>B[i\>])>\lift-6.45,{\textup{\normalsize\circled1}},> \twv @>{\!\bar{\couni{\!\!f}}(\OY)\otimes\>\id\,}>> \thn @>>>
   \4
         \\[-2pt]
@| @V\vartheta_{0i}VV @V\vartheta_{0i}V\mkern30mu\lift3,\rotatebox{30}{\hbox to 42pt{\rightarrowfill}},V 
         \\[-4pt]
\6 @. \ten @>{\!\!(\bar{\couni{\!\!f}}(\OY)\otimes\>\id)[i\>]\>}>> \lvn 
         \\
@V\fst\vartheta_{0i}VV  @A\tilde p(D\<,B)[i\>]AA
         \\
\7 @= \9
\endCD
\]
\end{small}
\vskip6pt
Commutativity of the unlabeled subdiagrams is evident. 
To prove commutativity of subdiagram \circled1 replace $\tilde p\>$ by $p$ (reversing the associated arrows), and then look at the $(\LL f^*\!\<\dashv\< \Rf)$-adjoint diagram, which is the border of the natural diagram

\[
\mkern-3mu
\def\1{D\<\otimes\<\<f^*\<\<\big(B[i\>]\big)}
\def\2{f^*\<\<\fst D\<\otimes\<\<f^*\<\<\big(B[i\>]\big)}
\def\3{f^*\<(\fst D\<\otimes\<\<B[i\>])}
\def\4{f^*\<\<\big(\<(\fst D\<\otimes\<\<B)[i\>]\big)}
\def\5{f^*\<\<\big(\<(\fst (D\<\otimes\<\<f^*B)\<)[i\>]\big)}
\def\6{f^*\<\<\fst (D\<\otimes\<\<f^*B)[i\>]}
\def\7{(D\<\otimes\<\<f^*B)[i\>]}
\def\8{D\<\otimes\<\<(f^*B)[i\>]}
\def\9{\big(f^*\<(\fst D\<\otimes\<\<B)\<\big)[i\>]}
\def\ten{(f^*\<\<\fst D\<\otimes\<\<f^*\<B)[i\>]}
\def\lvn{f^*\<\<\fst D\<\otimes\<\<f^*\<(B)[i\>]}
 \bpic[xscale=3.4, yscale=1.4]
 
  \node(11) at (1,-1)[scale=.9]{$\3$};
  \node(13) at (3,-1)[scale=.9]{$\2$};
  \node(14) at (3.95,-1)[scale=.9]{$\1$};

  \node(23) at (3,-2)[scale=.9]{$\lvn$};
  \node(24) at (3.95,-2)[scale=.9]{$\8$};
 
  \node(31) at (1,-3)[scale=.9]{$\4$};
  \node(32) at (2,-3)[scale=.9]{$\9$};
  \node(33) at (3,-3)[scale=.9]{$\ten$};
  \node(34) at (3.95,-3)[scale=.9]{$\7$};
 
  \node(41) at (1,-4)[scale=.9]{$\5$};
  \node(42) at (2,-4)[scale=.9]{$\6$};
  \node(43) at (3,-4)[scale=.9]{$\7$};

%horizontal  
  \draw[->] (11)--(13);  
   \draw[->] (13)--(14);  
    \draw[->] (23)--(24);  
  \draw[-,double distance=2pt] (31)--(32);  
   \draw[->] (32)--(33);  
     \draw[->] (33)--(34);  
  \draw[-,double distance=2pt] (41)--(42);  
   \draw[->] (42)--(43);  
    \draw[-,double distance=2pt] (43)--(34);
  
 %vertical
  \draw[->] (11)--(31) node[midway, left, scale=.7]{$f^*\vartheta_{0i}$};  
   \draw[-,double distance=2pt] (13)--(23);
    \draw[-,double distance=2pt] (14)--(24);
  \draw[->] (23)--(33) node[midway, right, scale=.7]{$\vartheta_{0i}$};  
   \draw[->] (24)--(34) node[midway, right, scale=.7]{$\vartheta_{0i}$};  
   \draw[->] (31)--(41) node[midway, right, scale=.7]{$f^*\<\<\big(p(D\<,B)[i\>]\big)$};  
    \draw[->] (32)--(42) node[midway, right, scale=.7]{$\big(f^*(p(D\<,B)\big)[i\>]$};  
      \draw[->] (33)--(43);  
      
 %label
  \node at (2,-2){\circled2};
 \epic
\]

To show that subdiagram \circled2 commutes, replace $f_*D$ and $B$ by quasi-isomorphic q-flat complexes, and $\vartheta$ by $\vartheta'$ (see \ref{vartheta}), to reduce the question to the analogous one for ordinary complexes and nonderived functors,  which situation  is readily handled. Details, as well as commutativity of the other subdiagrams, are left to the reader. Thus the adjoint diagram commutes,
whence so does \circled1, and the  conclusion results.
\end{proof}

\begin{subprop}\label{transint}
Let\/ $f\colon X\to Y$ and\/ $g\colon Y\to Z$ be $\SS$\kf-maps. Then with\/ $(-)^!$~ 
as in\/~\eqref{def!}$,$ \ $\ps^!$ as in\/~\eqref{ps!}$,$ and $\couni{}$ as in \textup{\ref{extend int},}
the transitivity diagram\/~\eqref{transitivity} commutes.
\end{subprop}

\begin{proof} Global duality asserts the existence, for \emph{any} $\SS$\kf-map $f\colon X\to Y\<$, of a right adjoint $f^\times$ for the functor $\fst \colon\Dqc(X)\to\Dqc(Y)$ (see \cite[4.1]{li}). For confined 
$f$, the restriction of $f^\times$ to $\Dqcpl(Y)$ can be identified with the functor~$f^!_\upl$ from \S\ref{shriek}(i); in particular, the relative dualizing complex $\cd f$ in~\eqref{def!} can be identified
with $f^\times\OY$. 
Also, by \cite[4.7.2 and 4.7.3(a)]{li},  $\chi^f_C$~in~\eqref{chi} extends\va{-1} to an isomorphism $f^!C\set \cd f\Otimes\sX f^*C\iso f^\times C$ for all $C\in\Dqc(Y)$; and by their very definitions, this extended $\chi^f_C$ and~$\bar{\couni{\<\<\!f}}(C)\colon f^{}_{\<\<*}f^!C\to C$ correspond under the adjunction
$\Rf\dashv f^\times$. 

Thus identifying $f^!$ with $f^\times$ via the extended isomorphism $\chi^f$ turns $\couni{\<\<\!f}$~into the counit map $\couni{\<\<\!f}^\times\colon  f_*f^\times\to \id$.
Furthermore,\va{-.7} as in the proof of \cite[4.9.5]{li}, that identification of $f^!$ with~$f^\times$ turns 
$\ps^!$ in \eqref{ps!}\va{.5} into the natural pseudo\-functorial isomorphism 
$$
\ps^\times\<\<\colon f^\times\< g^\times\iso (gf)\<^\times.
$$ 

The proof of \cite[4.1.2]{li} shows that commutativity of diagram~\eqref{transitivity} 
with $(-)^\times$, $\ps^\times$  and~$\couni{\<\<\!f}^\times$ in place of 
$(-)^!$, $\ps^!$  and~$\couni{\<\<\!f}$, respectively,  holds by definition of~$\ps^\times\<$. The conclusion follows.
\end{proof}
\end{cosa}

\begin{cosa}\label{2nd diag} 
It remains to show that with $\Dd$ the independent square \eqref{orientedcommsq}, 
diagrams~\eqref{first diagram} and  \eqref{second diagram} commute. 

\penalty-2000
\begin{subcosa} According to the definitions in sections~\ref{bchadmirado} and~\ref{integral}, commutativity of \eqref{first diagram} amounts to commutativity of 
the following $\D(Y')$-diagram, in which $G\in\Dqc(Y)$,  $\otimes$ stands for $\Otimes{}$ with the appropriate subscript, labels on maps tell how those maps arise, and unlabeled maps are the natural ones.\looseness=-1
\[
\mkern-3mu
\def\1{u^*\!\fst(\cd f\otimes f^*\<G)}
\def\2{g_*v^*\<(\cd f\otimes\< f^*\<G)}
\def\3{g_*(v^*\cd f\otimes v^*\!f^*\<G)}
\def\4{u^*\<(\fst\cd f\otimes G)}
\def\5{u^*\!\fst\cd f\otimes u^{\<*}\<G}
\def\6{g_*v^*\cd f\otimes u^{\<*}\<G}
\def\7{g_*(v^*\cd f\otimes g^*u^{\<*}\<G)}
\def\8{\,g_*\cd g\otimes u^{\<*}\<G}
\def\9{g_*(\cd g\otimes g^*u^{\<*}\<G)}
\def\ten{u^{\<*}\<(\OY\otimes G)}
\def\lvn{u^{\<*}\<\OY\otimes u^{\<*}\<G}
\def\twv{\mathcal O_{Y'}\otimes u^{\<*}\<G}
\def\thn{u^{\<*}\<G}
 \bpic[xscale=3.2,yscale=1.4]
  \node(11) at (1,-1) {$\1$};
  \node(12) at (2.03,-1) {$\2$};
  \node(14) at (4,-1) {$\3$};

  \node(21) at (1,-2) {$\4$};
  \node(22) at (2.03,-2) {$\5$};
  \node(23) at (2.97,-2) {$\6$};
  \node(24) at (4,-2) {$\7$};

  \node(33) at (2.97,-3) {$\8$};
  \node(34) at (4,-3) {$\9$};

  \node(41) at (1,-4) {$\ten$};
  \node(42) at (2.03,-4) {$\lvn$};
  \node(43) at (2.97,-4) {$\twv$};

  \node(51) at (1,-5) {$\thn$};
  \node(52) at (2.03,-5) {$\twv$};
  
 %horizontal
   \draw[->](11)--(12) node[midway, above, scale=.75]{$\bar\theta_\Dd$};
   \draw[->](12)--(14);
   
   \draw[->](21)--(22);
   \draw[->](22)--(23) node[midway, above, scale=.75]{$\bar\theta_\Dd$};
   \draw[->](24)--(23) node[midway, above, scale=.75]{$\tilde p$};

   \draw[->](34)--(33);

   \draw[->](43)--(42);
   \draw[->](42)--(41);

   \draw[->](52)--(51);

 %vertical
   \draw[->](11)--(21) node[midway, left=1pt, scale=.75]{$\tilde p$};
   \draw[->](14)--(24)node[midway, right=1pt, scale=.75]{$\ps^*$};
   
   \draw[->](21)--(41) node[midway, left=1pt, scale=.75]{$\bar{\couni{\!\<\<f}}$};
   \draw[->](22)--(42) node[midway, left=1pt, scale=.75]{$\bar{\couni{\!\<\<f}}$};
   \draw[->](23)--(33) node[midway, right, scale=.75]{\eqref{Bcd}};
   \draw[->](24)--(34) node[midway, right, scale=.75]{\eqref{Bcd}};

   \draw[->](33)--(43) node[midway, right=1pt, scale=.75]{$\bar{\couni{\!\<\<g}}$};
   
   \draw[->](41)--(51);
   \draw[->](42)--(52);

 %slanted
  \draw[-, double distance = 2pt] (43)--(52);
  
 %labels 
   \node at (2.5,-1.45){\circled1};
   \node at (2.5,-3){\circled2};

 \epic
\]

Commutativity of subdiagram \circled1 is given by \cite[3.7.3]{li}.

Subdiagram \circled2, without $\otimes \,u^{\<*}\<G$, is just \eqref{first diagram} applied to $\OY$.
This commutes by the definition of   
$\<\bar{\>\mathsf B}_{\mkern-.5mu\Dd}(\OY)\ (=\bchadmirado{\Dd}(\OY)$,
see \eqref{B=barB}).

Commutativity of the remaining subdiagrams is straightforward to verify.

\end{subcosa}

\begin{subcosa}
As for \eqref{second diagram}, since we are now dealing exclusively with confined maps, we may, as in the proof of Proposition~\ref{transint},
identify $(-)^!$ with a right adjoint of~$(-)_{\<*}\>$, and $\couni{\<\<(-)}$ with the corresponding
counit map.  

Let $\psi_\Dd\colon v_*g^!\to f^!u_*$ be the natural composite functorial map
$$
v^{}_{\<*}g^!\to f^!\<\<\fst v^{}_{\<*}g^! \overset{\ps_*}{=\!=} f^!u_*g_*g^! \xto{\couni{\!g}\>} f^!u_*.
$$
The left adjoints of the target and source of 
$\psi_{{\Dd}}$ are then $u^*\!\fst$ and $g_*v^*$ respectively; and the corresponding ``conjugate" map is just $\theta_{\Dd}\>$, cf.~\cite[Exercise 3.10.4]{li}.
Since $\theta_\Dd$ is an isomorphism, therefore so is $\psi_\Dd\>$, and $\psi_\Dd^{-\<1}$ is the map
conjugate to $\theta_\Dd^{-\<1}$ (see \cite[3.3.7(c)]{li}). This means that $\psi_\Dd^{-\<1}$ is the image of the identity map under the sequence of natural isomorphisms  (where $\Hom$ denotes maps of functors)
\begin{multline*}
\!\!\Hom(f^!u_*,f^!u_*)
\<\<\iso\!\Hom(\fst f^!u_*,u_*)
  \<\<\iso\!\Hom(u^*\!\fst f^!u_*, \id)    \\
\:\ \quad\xrightarrow[\textup{via}\;\theta_{\Dd}^{-\<1}]{\lift.3,\sim,}
\Hom(g_*v^*\! f^!u_*, \id)
  \<\<\iso\!\Hom(v^*\! f^!u_*,  g^!)       
\<\<\iso\!\Hom(f^!u_*, v_*g^!).
\end{multline*}
Explicating, one gets that $\psi_\Dd^{-\<1}$ is the natural composition\va{-2}
$$
f^!u_*  \!\xto{\!\eta^{}_{v}\>} 
v_*v{}^*\!f^!u_*  \!\to
v_*g^!g_{*}\>v{}^*\!f^!\>u_*  \!\xto{\<\<\textup{via}\;\theta_{\Dd}^{-\<1}}
v_*g^!u^*\!\fst f^!\>u_* \!\xto{\!\couni{\<\<\!f}\>\>}
v_*g^!u^*\<u_*\!\xto{\!\!\epsilon^{}_u\>\>}
v_*g^!.
$$
By the definition of $\bchadmirado{\Dd}$ when $g$ is proper (\S\ref{bchadmirado}), it results that
$\psi_\Dd^{-\<1}$ is the natural composition\va{-2}
$$
f^!u_*  \!\xto{\!\eta^{}_{v}\>} 
v_*v{}^*\!f^!u_*  \!\xto{\bchadmirado{\Dd}\>\>}
v_*g^!u^*\<u_*\!\xto{\!\!\epsilon^{}_u\>\>}
v_*g^!,
$$
that is, $\psi_\Dd^{-\<1}=\phi_\Dd$.

Commutativity of the following natural diagram, whose top row composes, by definition, to~the map induced by $\psi^{}_{{\Dd}}=\phi_\Dd^{-\<1}\<$, and whose bottom row composes to the identity, is an obvious  consequence of  Proposition~\ref{transint}. Commutativity of \eqref{second diagram} results.
  \[
   \bpic[xscale=3, yscale=1.5]
      \node (11) at (0,-1) {$f^!\>u^{}_*u^!$};
      \node (12) at (1,-1){$f^!\>u^{}_*g^{}_*\>g^!\<u^!$};
      \node (13) at (2,-1) {$f^!\<\fst\>v^{}_{\<*}\>g^!\<u^!$};
      \node (14) at (3,-1) {$v^{}_{\<*}\>g^!\<u^!$};

      \node (23) at (2,-2) {$f^!\<\fst\>v^{}_{\<*}v^!\<\<f^!$};
      \node (24) at (3,-2) {$v^{}_{\<*}v^!\<\<f^!$};

      \node (31) at (0,-3) {$f^!$};
      \node (33) at (2,-3) {$f^!\<\fst f^!$};
      \node (34) at (3,-3) {$f^!$};
      
      %Horizontales
      \draw [->] (12) -- (11) node[above, midway, scale=0.75]{$f^!\>u^{}_*\couni{\!\<{g}}$};
      \draw [-, double distance=2pt] (13) -- (12)
                              node[above=1pt, midway, scale=0.75]{$f^!\!\ps_*$};
      \draw [->] (14) -- (13) ;
 
      \draw [->] (24) -- (23) ;
      
      \draw [->] (33) -- (31) node[below=1pt, midway, scale=0.75]{$f^!\<\couni{\!\<\<f}$};
      \draw [->] (34) -- (33) ;
      
      %verticales
    
      \draw [->] (11) -- (31) node[left, midway, scale=0.75]{$f^!\couni{\!u}$};
      
      \draw [-, double distance=2pt] (13) -- (23)
                              node[left, midway, scale=0.75]{$f^!\<\fst\>v^{}_{\<*}\ps^!$};
      \draw [->] (23) -- (33) node[left, midway, scale=0.75]{$f^!\<\fst\couni{\!v}$};
      
      \draw [-, double distance=2pt] (14) -- (24)
                              node[right=1pt, midway, scale=0.75]{$v_*\ps^!$};
      \draw [->] (24) -- (34) node[right, midway, scale=0.75]{$\couni{\!v}$};
   \epic
  \]
\end{subcosa}
\end{cosa}

\section{Example:  Classical Hochschild homology of scheme-maps.}
\label{comparison}
This section illustrates some of the foregoing with a few remarks about earlier-known  Hochschild homology and cohomology functors on schemes, especially with regard to their relation with  the bivariant  functors arising from Example~\ref{ex-setup}(b). Global Hochschild theory goes back to Gerstenhaber and Shack, and has subsequently been developed by several more authors. Here we concentrate on the functors defined by C\u{a}ld\u{a}raru and Willerton (\cite{c1} and~\cite{cw}). \va2

 For smooth schemes
over a characteristic-zero field, bivariant  homology
groups coincide with  classical Hochschild homology groups; but the classical Hochschild cohomology groups 
are only direct summands of the bivariant ones (\S\S\ref{Cld}--\ref{coh-different}).
Even in this special case, then,
the bivariant theory has more operations on homology.

\begin{cosa}
Let $f \colon X \to Y$ be a quasi-compact quasi-separated scheme\kf-map, and  
$$
\delta=\delta_{\<\<f}: X\to X\times_Y X
$$ 
the associated diagonal map---which is  quasi-compact and quasi-separated,
 \cite[p.\,294, (6.1.9)(i), (iii), and p.\,291, (6.1.5)(v)]{EGA1}.

The  \emph{pre-Hochschild complex} of $f$  is
$$
\Hsch{\<f}\set\LL\delta^*\delta_*\OX.
$$
(When $f$ is flat,  the prefix ``pre-" can be dropped, see~\cite[p.\,222, 2.3.1]{BF}.)

\goodbreak

The complex  $\Hsch{\<f}$ gives rise to  \emph{classical Hochschild cohom\-ology functors}
$$
\CH\CH^{\>\>i}_{X|Y}(F\>)\set H^{i\>}\R\>\sHom_{\<X}(\CH_f, F\>)\qquad (i\in\mathbb Z,\ F\in\Dqc(X)),
$$
and their global counterparts (cf.~\cite[p.\,217, 2.1.1]{BF})
$$
\HH^{\>i }_{X|Y}(F\>)\set\ext^{\>i}_{\<X}(\CH_f, F\>)= \textup{H}^{\>i  }\big(X,\R\> \sHom^{}_{\<X}(\CH_f, F\>)\big).
$$

When $X$ is affine, say $X=\spec(A)$, 
and $Y=\spec(k)$ with $k$ a field, this terminology is compatible with the
 classical one for $A$-modules.

The global Hochschild cohomology
$$
\HH^*_{X|Y}(F\>)\set\oplus_{i\in\ZZ}\HH^{\>i }_{X|Y}(F\>)=\ext^*_\sX\<(\CH_f, F\>)
$$
is a symmetric graded module over the commutative-graded ring
$$
H_\sX \set\oplus_{i\ge0}\,\textup{H}^i(X\<\<,\OX\<),
$$
see Proposition~\ref{E is H-graded}. 

The sheafified version of the adjunction $\LL\delta^*\!\dashv\<\delta_*$  
(see e.g., \cite[3.2.3(ii)]{li}), gives, furthermore,
\begin{align*}
\HH^{\>i }_{X|Y}(F\>)
&\cong \textup{H}^{\>i}\big(X\!\times_Y\!X,\, \delta_*\R\>\sHom^{}_{\<X}(\LL\delta^*\<\delta_*\OX, F\>)\big)\\
&\cong \textup{H}^{\>i}\big(X\!\times_Y\!X,\, \R\>\sHom^{}_{\<X\times_{\<Y}X}(\delta_{\<*}\OX, \delta_{\<*}F\>)\big)\\
&=\ext^i_{\<X\times_YX}(\delta_{\<*}\OX,\delta_{\<*}F\>).
\end{align*}
\end{cosa}

\begin{subprop}\label{coh comm} 
Under Yoneda composition, the classical Hochschild cohomology associated to\/ $f,$
$$
\HH^*_{\<\<X|Y}(\OX\<)\cong\oplus_{i\in\ZZ}\Hom_{\>\D(X\times_Y X)}\<\<(\<\delta_{\<*}\OX,\delta_{\<*}\OX[i\>]),
$$ 
is a graded-commutative\/ $H_{\<X}$-algebra, of which\/ $H_{\<X}$ is\va1 a graded-ring retract. 
\looseness=-1

\end{subprop}

\begin{proof}
Commutativity is well-known, cf.~ \cite[\S2.2]{BF}.  Here is one quick way to see it.
Let $\D_*\subset\D(X\times_YX)$ be the full subcategory whose objects are the complexes  
$\delta_{\<*}G\ (G\in\D(X))$.  With $p\colon X\times_YX\to X$ the first projection, set\looseness=-1
$$
E\tst\<F\set \delta_{\<*}(p_*E\Otimes{\<X}p_*F\>)\qquad(E,\;F\in \D_*).
$$
There are obvious functorial isomorphisms
\[
\lambda\colon(\delta_{\<*}\OX\tst -)\iso \id_{\D_*}\>, \qquad \rho\colon (-\tst \<\delta_{\<*}\OX)\iso \id_{\D_*}\!.
\] 
Then $(\tst, \delta_{\<*}\OX, \lambda, \rho)$ is a $\ZZ$-graded unital product,
and  the commutativity follows (see \ref{unital->gradedcomm}).

The  $H_{\<X}$-algebra structure is given by~\ref{* and T}(iii) (with $f$  replaced by $\delta$),
as is a left inverse for the structure map (with $f$ replaced by $p$).
\end{proof}

\begin{cosa}\label{graded centers} 
As in \S\ref{unital->gradedcomm},  $\HH^*_{\<\<X|Y}(\OX\<)$ is a graded-algebra retract of the graded center $\CC_*$ of $\D_*$. There is also a natural graded-ring homomorphism from $\CC_*$ to the graded center~$\CC$ of~$\D(X)$, induced by the essentially surjective functor \mbox{$p_*\colon\D_*\to\D(X)$.} Thus there is a natural graded-ring homomorphism
\begin{equation}\label{characteristic}
\varpi\colon \HH^*_{\<\<X|Y}(\OX\<)\to\CC.
\end{equation}
For flat $f$, this is canonically isomorphic to the \emph{characteristic homomorphism} that plays an important role
in \cite{BF} (where nonflat maps are also treated). It takes a $\D(X\!\times_Y\!X)$-map $\alpha\colon\delta_{\<*}\OX\to\delta_{\<*}\OX[i\>]$ to the natural functorial composition
$$
A\cong\OX\<\Otimes{X}\!A\cong p_*\delta_{\<*}\OX\<\Otimes{X}\!A\xto{\!\textup{via}\;\alpha\<}p_*\delta_{\<*}\OX[i\>]\<\Otimes{X}\!A\,\cong \,\OX[i\>]\<\Otimes{X}\!A\,\cong\, A[i\>].
$$

One checks, for example, that in~\ref{coh comm}, the left inverse---induced by $p_*$---for 
$H_{\<X}\to \HH^*_{\<X|S}(\OX)$ is the composition $\textup{ev}_{\OX}\smallcirc \varpi$  
(see~\eqref{evaluation}).

\end{cosa}  

\begin{cosa}
One has also the \emph{sheafified Hochschild homology functors} 
$$
\CH\CH_i^{X|Y}\<\<(F\>)\set  H^{-i}(\Hsch{\<f}\Otimes{\<\<X}F\>)\qquad (i\in\mathbb Z,\ F\in\Dqc(X)),
$$
and their global counterparts, 
$$
\textup{HH}_i ^{X|Y}\<\<(F\>)\set 
\tor_i ^{X}\<\<(\CH_f, F\>)=
\textup{H}^{-i}(X, \CH_f\Otimes{\<\<X}\> F\>).
$$

\goodbreak
The functorial \emph{projection isomorphisms} \cite[p.\,139, 3.9.4]{li}\va{-1}
$$
\pi(E,F\>)\colon\delta_{\<*}(\delta^*\<\delta_{\<*}E\,\Otimes{\<\<X\times_Y X}\,F\>)
\<\iso\<\delta_{\<*}E\,\Otimes{\<\<X\<\times_Y\<X}\,\delta_{\<*}F
\<\iso\<
\delta_{\<*}(E\,\Otimes{\<\<X}\,\delta^*\<\delta_{\<*}F\>)
$$
$(E,F\in\Dqc(X))$, give, furthermore,
\begin{align*}
\textup{HH}_i ^{X|Y}\<\<(F)
&\cong\textup{H}^{-i}(X\!\times_Y\!X,
\delta_{\<*}(\delta^*\<\delta_{\<*}\OX\Otimes{\<\<X}\> F\>))\\
&\cong \textup{H}^{-i}(X\!\times_Y\!X,\,
\delta_{\<*}\OX\Otimes{\<\<X\<\times_{\<Y}\<X} 
\delta_{\<*}F\>)\\
&=\tor_i ^{X\<\times_{\<\<Y}\<X}\<\<(\delta_{\<*}\OX\<, \delta_{\<*}F\>).
\end{align*}

\end{cosa}

\begin{cosa}\label{Cld}
C\u{a}ld\u{a}raru and Willerton work over a ``geometric category of spaces" in which some form of Serre duality holds (see~\cite[end of Introduction]{cw}), for example, the category of smooth 
projective varieties over an algebraically closed field $\mathbf k$ of characteristic zero.
What  they call the Hochschild cohomology of such a variety~$X$ is 
simply $\HH^*_{X|\spec(\mathbf k)}(\OX)$.

Their Hochschild homology,  
\[
\HH_i^{\text{cl}}\!\big(X) \set \Hom_{\D(X\times_{\mathbf k}\> X)}(\delta_{\<*}\sHom(\Omega^n_{X|\spec(\mathbf k)}[n],\OX\<),\>\delta_{\<*}\CO_{X}[-i\>]\big) \quad\ \ (i\in \ZZ),
\]
(where $n=\dim X$ and \smash{$\Omega^n_{X|\spec(\mathbf k)}$} is the sheaf of relative differential $n$-forms)
 is shown in \mbox{\cite[\S4.2]{cw}} to be isomorphic to the global Hochschild homology
$\HH_{\>i }^{X|\spec(\mathbf k)}(\OX\<)$.     (The ``cl" in the notation
indicates either  ``C\u{a}ld\u{a}raru" or ``classic.")
Their definitions and arguments actually apply to any
essentially smooth $f\colon X\to Y\>$ (\S\ref{shriek});\vspace{.7pt}  so when 
such an $f$ is given we can substitute~$Y$ for 
$\spec(\mathbf k)$ in the preceding.
\end{cosa}

\begin{cosa}
Also, it is indicated near the beginning\- of \cite[\S5]{cw}
that in their setup,  Hochschild homology is isomorphic to the bivariant $\HH_*(X)$ (\S\ref{ho-co}) associated with Example~\ref{ex-setup}(b).
This can be seen, more generally, as follows.

\penalty-2000
First, for any flat $f\colon X\to Y$,  with  $\pi_i\colon X\!\times_Y\!X \to X\ (i=1,2)$  the usual projections, and $p(-,-)$ the projection isomorphism in \eqref{projection}, one has, for any $F\in\Dqc(X)$,  the natural composite isomorphisms
\begin{equation}\label{Hsch and Hh}
 \begin{aligned}
  \mkern-10mu\zeta^{}_i(F\>)\colon 
  \delta^*\<\delta_*\OX\Otimes{\sX}F \cong
&\;\pi_{i*}\delta_*(F\>\Otimes{\sX}\>\delta^*\<\delta_*\OX)\\[2pt]
  \xrightarrow{\!\!\pi_{i*}p(F\<\<,\>\>\delta_*\OX)^{-\<1}\!\!\!\!}
&\;\pi_{i*}(\delta_*F\Otimes{\sX}\>\delta_*\OX)\cong
  \pi_{i*}(\delta_*\OX\Otimes{\sX}\>\delta_*F)\\[2pt]
  \xrightarrow{\pi_{i*}p(\<\OX\<\<,\>\delta_*F\>)\,}
&\;\pi_{i*}\delta_{\<*}(\OX\!\Otimes{\sX}\delta^*\<\delta_*F)
  \cong\delta^*\<\delta_*F.
 \end{aligned}
\end{equation}

It can be shown that the isomorphisms $\zeta^{}_1$ and $\zeta^{}_2$ are in fact equal. \vspace{1pt}

Now suppose the map $x\colon X\to S$ is flat, with Gorenstein fibers. Then, as is well-known, the complex $\omega_x\set x^!\OS$ is \emph{invertible,} that is, each point of $X$ has a neighborhood 
$U$ over which the restriction of $x^!\OS$ is $\D(U)$-isomorphic to
$\mathcal O_U[m]$ for some $m$ (depending on $U$, but constant on any connected component of $X$). The complex $\omega_x^{-\<1}\set\R\>\sHom(\omega_x,\OX)$\vspace{.7pt} is also invertible, and, in $\D(X)$,
$\omega_x\otimes_{\OX}\omega_x^{-\<1}=\omega_x\Otimes{\<X}\omega_x^{-\<1}\cong\OX$. There are natural isomorphisms
  \begin{align*}
 \HH_i(X)=\ext^{-i}_{\<\<X}(\Hsch{\<X},\omega_x)
 \iso &\ext^{-i}_{\<\<X}(\delta^*\<\delta_{\<*}\OX\Otimes{\<X}\omega_x^{-\<1}\<, \>\OX)\\
\iso &\ext^{-i}_{\<\<X}(\delta^*\<\delta_{\<*}\omega_x^{-\<1}\<,\> \OX)\qquad\textup{(see \eqref{Hsch and Hh})}\\
\iso &\ext^{-i}_{\<X\<\<\times_{\<\<S}X}(\delta_{\<*}\omega_x^{-\<1},\>\delta_{\<*}\OX).
 \end{align*}
In particular, if $x$ is \emph{essentially smooth,} 
of constant relative dimension $n$ \mbox{\cite[5.4]{Nk}},
then  $\omega_x\cong \Omega^n_{X|S}[n]$, yielding in this case that 
$\HH_i(X)\cong \HH_i^{\text{cl}}(X).$
\end{cosa}

\begin{cosa}\label{coh-different}
For cohomology, the situation is different. Referring to Example~\ref{ex-setup}(b), let $x\colon X\to S$ be the unique $\SS$-map,
and $\delta\colon X\to X\times_S X$ the diagonal. 

There are natural functorial maps 
$\delta_{\<*}\to\delta_{\<*}\LL\delta^*\<\delta_{\<*}\to\delta_{\<*}$ composing to the identity, so the natural identifications 
$$
\HH^*_{\<X|S}(\OX)\<\cong\ext^*_{\<X}(\delta_{\<*}\OX,\delta_{\<*}\OX)\!\!\quad\textup{and}\quad\!\!
\HH^*\<(X)\<\cong\ext^*_{\<X}(\delta_{\<*}\OX,\delta_{\<*}\LL\delta^*\<\delta_{\<*}\OX\<)
$$
entail that \emph{the classical Hochschild cohomology\/ $\HH^*_{\<X|S}(\OX)$ is, as a graded group, a 
direct summand of the bivariant cohomology} $\HH^*\<(X)$. 

The 
projection $\HH^*\<(X)\!\twoheadrightarrow\!\HH^*_{\<X|S}(\OX)$ can also be viewed as the map
\begin{equation*}
\HH^*\<(X)=\ext^*_{\<X}(\Hsch{x}\>, \Hsch{x})\to \ext^*_{\<X}(\Hsch{x}\>, \OX)=\HH^*_{\<\<X|S}(\OX\<).
\end{equation*}
induced by $\epsilon_\delta(\OX)\colon\Hsch{x}=\LL\delta^*\<\delta_{\<*}\OX\to\OX$.\va2

Since  $\HH^*_{\<\<X|S}(\OX\<)$ is graded-commutative (Proposition~\ref{coh comm}), the composition of  $\varpi$ in \eqref{characteristic} and~ $\textup{ev}_{\Hsch{x}}$ in \eqref{evaluation} gives a natural homomorphism of graded algebras over $H_{\<X}$,
$$
\HH^*_{\<\<X|S}(\OX\<)\to\HH^*_{\<\<X|S}(\Hsch{x})=\HH^*\<(X),
$$
with image in the graded center 
of  $\HH^*\<(X)$.

Thus $\HH^*\<(X)$ \emph{has a natural structure of graded} $\HH^*_{\<\<X|S}(\OX\<)$-algebra.

\end{cosa}

\end{document}